\documentclass[11pt]{amsart}
\usepackage[utf8]{inputenc}
\usepackage[T1]{fontenc}
\usepackage[margin=1.2in]{geometry}
\usepackage{amssymb}
\usepackage{amsmath}
\usepackage{mathtools}
\usepackage{color}
\usepackage{amsthm}
 \usepackage{lmodern}
\usepackage{tikz}
\usetikzlibrary{matrix}
 \usepackage{caption} 
 \input xy
\xyoption{all}
\usepackage{graphics}
\usepackage{amsfonts}
\usepackage{amssymb}
\usepackage{amscd}
\numberwithin{equation}{section}
\newtheorem{theorem}{Theorem}[section]

\newtheorem{proposition}[theorem]{Proposition}
\newtheorem{lemma}[theorem]{Lemma}
\newtheorem{prop}[theorem]{Proposition}
\theoremstyle{definition}
\newtheorem{definition}[theorem]{Definition}
\newtheorem{remark}[theorem]{Remark}
\newtheorem{example}[theorem]{Example}
\newtheorem{Conjecture}[theorem]{Conjecture}

\newcommand\q{\mathfrak q}

\newcommand\half{\tfrac{1}{2}}

\DeclareMathOperator{\ad}{ad}

\newcommand\ov{\overline}
\newcommand\hvee{\mathbf{h}^\vee}

\newcommand\be{\beta}

\newcommand\g{\mathfrak g}

\newcommand\h{\mathfrak h}

\newcommand\D{\Delta}
\renewcommand\l{\lambda}
\newcommand\Dp{\Delta^+}

\renewcommand\d{\delta}
\renewcommand \dim{\mathrm{ dim}\,}

\renewcommand\a{\alpha}
\renewcommand\aa{\mathfrak a}

\renewcommand\k{\mathbf k}

\newcommand{\Z}{\mathbb Z}
\newcommand\nat{\mathbb N}

\newcommand\s{\sigma}

\renewcommand\aa{\mathfrak a}
\newcommand\la{\lambda}

\newcommand\C{\mathbb C}
\newcommand\R{\mathbb R}

\newcommand\si{\sigma}

\newcommand{\ZZ}{\mathbb{Z}}

\newcommand{\sdim}{\text{\rm sdim}}

\newcommand{\langlerangle}{\langle \, . \, , 
. \, \rangle}

\newcommand{\vac}{|0\rangle}
\newcommand{\vacc}{{\mathbf 1}}
\newcommand{\bea}{\begin{eqnarray}}
\newcommand{\eea}{\end{eqnarray}}
\newcommand{\Ws}{W_k^{\min}(\g)}
\newcommand{\Wu}{W^k_{\min}(\g)}

\begin{document}

\title{Extremal unitary representations of big $N=4$  superconformal algebra}
\author[Victor~G. Kac, Pierluigi M\"oseneder Frajria,  Paolo  Papi]{Victor~G. Kac\\Pierluigi M\"oseneder Frajria\\Paolo  Papi}

\begin{abstract} In this paper we give a detailed proof of the classification of extremal (=massless) unitary highest weight representations in the Neveu Schwarz and Ramond sectors of the $big\ N=4$ superconformal algebra which can be found in \cite{Gunaydin}. Our results agree with the general conjectures about classification  of unitary highest weight representation of minimal $W$-algebras attached to basic Lie superalgebras formulated in \cite{KMP1}, \cite{KMPR}, and complete their proof for the $big\ N=4$ superconformal algebra.
\end{abstract}

\maketitle

\tableofcontents
\section{Introduction}
A number  of works on unitarity of highest weight representation of infinite-dimensional Lie algebras and superalgebras was published in the 1980's; one can find an extensive list of references in  \cite{AKMP, ET1, Gunaydin, Kazamachar, M}. In several papers, including \cite{DL}, \cite{KMP}, the problem was put in the framework of vertex operator algebras (VOA), for which the eigenvalues of the zero mode $L_0$ of the conformal vector lie in $\half\Z_{\ge 0}$. In this language, given a conjugate linear involution $\phi$ of a VOA $V$,  a Hermitian form on an irreducible representation $M$ of $V$ is called {\it $\phi$-invariant} if the  adjunction relation 
$$(a^M_n)^*=(-1)^{\D+2\D^2}\phi(a)^M_{-n}$$
holds for any quasi-primary element $a\in V$ of conformal weight $\D(\in \half\Z_{\ge 0})$, and $M$ is called {\it unitary} if the Hermitian form is positive definite.
\par
We proved in \cite{KMP} that a VOA $V$ endowed with a  conjugate linear involution $\phi$ admits a (unique up to a constact factor) $\phi$-invariant Hermitian form iff any $a\in V$ of conformal weight $1$ is quasi-primary. (Such VOA are called of strong CFT type. They are strongly generated by quasi-primary elements.).\par
A class of vertex algebra which has been extensively investigated in physical literature is that of superconformal algebras, which consists of  the Virasoro, Neveu-Schwarz, $N=2,3,4$ and big $N=4$  algebras. The study of their unitary representations has been a hot topic in the late 80's (see \cite{M,ET1,ET2,ET3,Gunaydin} and references therein). The  superconformal algebras were later realized as (an extension of) a special case of a relevant class of vertex algebras, the
 minimal simple $W$-algebras $\Ws$, associated to a basic simple finite-dimensional Lie superalgebra $\g$, its ``minimal'' $sl_2$-subalgebra $\mathfrak s$, and the level $k\in \C$ (\cite{KRW, KW1}). In particular, the big $N=4$ algebras is obtained from $W^k_{\min}(D(2,1;a))$ by adding four fermions and one boson (cf. \eqref{WN=3embedding}).  \par
  In the paper \cite{KMP1} we began to study in a systematic way the  unitarity  of the VOA $\Ws$.
 We showed that, unless $k$ is a collapsing level (i.e., $\Ws$ ``collapses'' to its affine subalgebra), the VOA $\Ws$ can be unitary only if the conjugate linear involution of $\g$ which induces that of $\Ws$ is ``almost compact'', and then it is essentially unique. We went  on to classify in \cite{KMP1} the values of $k$ for which the minimal simple $W$-algebras are unitary, and to classify their unitary highest weight representations. In \cite{KMPR} we extended these results to the Ramond sector. It turned out that different methods work for non-extremal and extremal representations (physicists call them massive and massless, respectively \cite{ET1, Gunaydin, M}), which, roughly  speaking, means that they enter or do not enter in ``continuous''  families, respectively. We refer to  Section 2 for  precise definitions and statements. \par
The non-extremal unitary highest weight representations of minimal $W$-algebras were  studied using the generalized Fairlie construction \cite{KMP, KMPSF}, and 
the determinant of the Hermitian form,  which allowed us  to completely and uniformly classify them.
 However there seems to be no  unified approach to the construction of massless unitary representations for minimal $W$-algebras, though a unified conjectural list is given in \cite{KMP1}.
 In all known cases they are constructed using a free field realization \cite{ET1, KMP1,  M}, or a coset construction \cite{Gunaydin}. \par
 The  main purpose of the present paper is to prove unitarity of the extremal  highest weight modules, both untwisted and Ramond twisted,  over the big $N=4$ superconformal algebra, confirming our conjectures \ref{NS} and \ref{R} and the results of \cite{Gunaydin}
(of which we were unaware when writing \cite{KMP1}, \cite{KMPR}). \par
It is well-known that for a compact Lie group $G$ and its closed connected subgroup $A$ the coset construction of the $N=2$ superconformal algebra is possible if the manifold $G/A$ carries an invariant integrable complex 
 structure \cite{Kazamachar,Kazamanew}. 
According to the paper \cite{STVP} of  Sevrin et al., there is a coset construction of the  big $N=4$ superconformal algebra if $G/A$ admits an invariant  hypercomplex structure. A general construction of such a structure was given in Joyce \cite{Joyce}, and in Section 6 we recall this construction, following \cite{BGP}.
Applying this construction in Section 7 to the coset $SU(n)/SU(n-2)$ we obtain in  Section 8  a vertex algebra homomorphism
\begin{equation}\label{01}
W^{k}_{\min}(D(2,1;a))\to V^{k_1}(sl_n)\otimes F,
\end{equation}
where $k=-\frac{(n-1)(k_1+1)}{k_1+n},\,a=\frac{k_1+1}{n-1}$ and $F$ is a vertex algebra of free fermions.\par
In Section 9 we explain how homomorphism \eqref{01} implies unitarity of non-twisted extremal modules over the big $N=4$ superconformal algebra, cf \cite{Gunaydin}.
The construction of fields of conformal weight $\frac{3}{2}$ of superconformal algebras uses the affine cubic Dirac operators, constructed by Kac and Todorov in \cite{KacT} (see also \cite[Section 4.3]{DK}, \cite{KMPD}), and used to construct unitary modules over the $N=1$ superconformal algebras. Kazama and Suzuki in \cite{Kazamachar} and \cite{Kazamanew} extended this construction to  $N=2$ superconformal algebras, as coset modules attached to integrable complex
structures on compact homogeneous spaces. We recall these constructions in Sections 3,4,5, as an introduction to the more difficult construction in the remainder of the paper. 
\par
In \cite{Gunaydin}, the unitarity of highest weight modules  in the Ramond sector for the big $N=4$ algebra  is  deduced 
from analogous results in the Neveu-Schwarz sector by using  {\it spectral flow}. In 
\cite{KMPSF} we generalize this approach by providing a spectral flow for all minimal $W$-algebras. In Section 10  we directly   construct the  extremal unitary representations of the big $N=4$ algebra in the Ramond sector,  without resorting to the untwisted case. \par
Unfortunately we were unable to find so far the  geometric structure on $G/A$, for which the coset costruction produces all the massless unitary representations of the remaining minimal $W$-algebras $\Ws$, where $\g=spo(2|m)$ with $m>4$, $F(4)$ and $G(3)$, whose conjectural list is given in \cite{KMP1}, and in \cite{KMPR} for the Ramond sector.\par
 Little is known about the classification of unitary highest weight representations of $W_k(\g,\mathfrak s)$ where $\mathfrak s$ is a non-minimal $sl_2$ subalgebra of $\g$. It is proved in \cite{KMP} that these VOA are of strong CFT  type. It is also proved in \cite[Remark 7.10]{KMP} that for a 
 non-collapsing level $k$, $W_k(\g,\mathfrak s)$ can be unitary  only if the grading of $\g$ defined by a non-zero  semisimple element of $\mathfrak s$ is compatible with the parity of the superalgebra $\g$. In the case when $\g$ is a Lie algebra this means that a non-zero nilpotent element of $\mathfrak s$ is even.

\section{Unitarity of highest weight representations of minimal $W$-algebras}
Let $\g$ be a simple finite-dimensional Lie superalgebra, over $\C$, with a reductive even part $\g_{\bar 0}$ and invariant
non-degenerate bilinear form $(\cdot |\cdot)$, with restriction to $\g_{\bar 0}$ non-degenerate.
 Let $\mathfrak s=Span\{e,x,f\}$, where $[e,f]=x, [x,e]=e, [x,f]=-f$ be an $sl_2$ subalgebra of $\g_{\bar 0}$. To the datum
 $(\g,\mathfrak s, k\in\C)$ one associates the universal quantum affine $W$-algebra $W^k(\g,\mathfrak s)$ of level $k$ by the quantum Hamiltonian reduction
 \cite{KRW}, \cite{KW1}. If $k$ is different from the critical level $k_{crit}$ (which we assume), the vertex algebra $W^k(\g,\mathfrak s)$ has a unique maximal ideal, and the quotient by this ideal is a simple  $W$-algebra, denoted by $W_k(\g,\mathfrak s)$. These VOA are of strong CFT  type \cite{KMP}.
 
 The {\it minimal} $W$-algebras $\Wu$ correspond to a choice of $\mathfrak s$, called minimal, for which the
 $ad\,x$-eigenspace decomposition is of the form 
 \begin{equation}\label{1.1}
 \g=\g_{-1}\oplus\g_{-1/2}\oplus\g_0\oplus\g_{1/2}\oplus\g_1,\quad \text{ where $\g_{-1}=\C f,\,
 \g_{1}=\C e$}.\end{equation}
 We normalize the bilinear form $(\cdot |\cdot)$ by the condition $(x|x)=\half$. Then $k_{crit}=-h^\vee$, where 
 $h^\vee$ is half of the eigenvalue of the Casimir operator on $\g$. The decomposition \eqref{1.1} and the numbers $h^\vee$ are listed in \cite[Tables 1-3]{KW1}.

As explained in the Introduction, in  order to define unitarity of a $W$-algebra, one needs a conjugate linear involution $\phi$ of $\g$, which fixes the subalgebra $\mathfrak s$ pointwise. Then, provided that $k\in\R$, $\phi$ induces a conjugate linear involution  of the vertex algebra $\Wu$ that  descends to $\Ws$.

It is proved in \cite[Proposition 7.2]{KMP1} that for  a non-collapsing level $k\in\R$, any conjugate linear involution $\phi$ of the vertex algebra $\Wu$ is necessarily induced by a conjugate linear involution $\phi$ of $\g$ fixing $\mathfrak s$. (Recall that $k$ is called a collapsing level if $\Ws$ is isomorphic to its affine part.) Moreover, it is proved in \cite[Proposition 8.9]{KMP1} that the vertex algebra $\Ws$
can be  unitary only if the centralizer $\g^\natural$ of $\mathfrak s$ in $\g$ is a semisimple subalgebra of $\g_{\bar0}$, and the conjugate linear  involution $\phi$ is {\sl almost compact}, i.e. it restricts to a compact  involution of $\g^\natural$, and it leaves $\{e,x,f\}$ fixed. We write $\g^\natural=\oplus_i\g^\natural_i$, where $\g^\natural_i$ are simple components of $\g^\natural$.

We prove in  \cite{KMP1} that an almost compact conjugate linear involution of $\g$ exists if and only if $\g$ is from the following lists:
\begin{equation}\label{isunitary}
psl(2|2),\ spo(2|m) \text{ for }m\ge0, D(2,1;a)\text{ for }a\in \R,\ F(4),\ G(3);
\end{equation}
\begin{equation}\label{nonunitary}
sl(2|m)\text{ for }m\ge3,\ osp(4|m)\text{ for }m>2\text{ even},
\end{equation}
and it is essentially unique. 
Recall that, for $k\ne k_{crit}$,  the vertex algebra $\Wu$ is of strong CFT type with Virasoro field $L=\sum_{n\in\Z}L_nz^{-n-2}$, and it is strongly and freely generated by the operators $L_n$, $n\in\Z$, and the Fourier coefficients of the primary fields $J^{\{a\}}(z)=\sum_{n\in\Z} J^{\{a\}}_{n}z^{-n-1}$, $a\in\g^\natural$, of conformal weight 1, and $G^{\{u\}}(z)=\sum_{n\in\half+\Z}G^{\{u\}}_nz^{-n-\tfrac{3}{2}}$, $u\in\g_{-1/2}$, of conformal weight $\tfrac{3}{2}$ \cite[Theorems 4.1 and 5.1]{KW1}. The $\lambda$-brackets among these generators are displayed in \cite{KRW, KW1, AKMPP}.\par
When $\g=spo(2|m)$, $m=0,1$, and $2$, the $W$-algebra $\Wu$ is the universal Virasoro, Neveu-Schwarz, and $N=2$ vertex algebra, respectively, for which unitarity  is well understood and is as follows.
Up to isomorphism,  these vertex algebras depend only on the central charge $c(k)$, given by  
\begin{equation}\label{ccc}c(k)=\frac{k\,d}{k+h^\vee}-6k+h^\vee-4,\text{ where $d=\sdim\g$.}\end{equation} Putting $k=\frac{1}{p}-1$ in \eqref{ccc} in all three cases, we obtain 
\begin{align}
\label{1}c(k)&=1-\frac{6}{p(p+1)} \quad \text{for   Virasoro vertex algebra,}\\
\label{22}c(k)&=\frac{3}{2}\left(1-\frac{8}{p(p+2)}\right)\quad \text{for   Neveu-Schwarz vertex algebra,}\\
\label{3}c(k)&=3\left(1-\frac{2}{p}\right)\quad \text{for   $N=2$ vertex algebra.}
\end{align}
The following theorem is a result of several papers, published in the 80s in physics and mathematics literature, see e.g. \cite{AKMPP, ET3, Gunaydin,  M} for references.
\begin{theorem}The complete list of unitary $N=0,1,$ and $2$ vertex algebras is as follows: either $c(k)$ is given by \eqref{1}, \eqref{22}, or \eqref{3}, respectively, for $p\in \mathbb Z_{\ge 2},$  or  $c(k)\ge 1, \frac{3}{2}$ or $3$, respectively. 
\end{theorem}
\par
For the remaining cases, we proved in \cite[Proposition 8.19]{KMP1} that, for $k\ne k_{crit}$, the minimal $W$-algebra $\Ws$ is not unitary for $\g$ from the list \eqref{nonunitary}, except when $\g=sl(2|m)$, $m\ge3$, and the level is the collapsing level $k=-1$. Furthermore, we proved in \cite[Corollary 11.2]{KMP1} that, for $\g$ from the list \eqref{isunitary}, the vertex algebra $\Ws$ is non-trivial   unitary for $k\ne k_{crit}$ if and only if $k$ lies in the {\sl unitary range}, given in the  following Table \ref{tabel0}, where we also display the critical values of $k$ and the collapsing values   $k_0$ for which $\dim W_{k_0}^{\min}(\g)=1$:
\renewcommand{\arraystretch}{1.4}
\begin{center}
\begin{tabular}{c | c| c |c   }
$\g$&
unitary range&
$k_{crit}$& $k_0$\\
\hline
$psl(2|2)$&$-(\nat+1)$&$0$&$-1$\\
\hline
$spo(2|3)$&$-\tfrac{1}{4}(\nat+2)$&$-\half$&$-\half$\\
\hline
$spo(2|m),\,m\ge4$&$-\tfrac{1}{2}(\nat+1)$&$\tfrac{m}{2}-2$&$-\half$\\
\hline
$D(2,1;\tfrac{m}{n})$&$-\tfrac{mn}{m+n}\nat,\ m,n\in\nat\text{ coprime},\, (m,n)\ne(1,1)$&$0$& \text{none}\\
\hline
$F(4)$&$-\tfrac{2}{3}(\nat+1)$&$2$&$-\tfrac{2}{3}$\\
\hline
$G(3)$&$-\tfrac{3}{4}(\nat+1)$&$\tfrac{3}{2}$&$-\tfrac{3}{4}$
\end{tabular}
 \captionof{table}{\label{tabel0}}
\end{center}
 In our paper \cite{KMP1}, we also studied unitarity of irreducible highest weight $\Wu$-modules $L^W(\nu,\ell_0)$, where $\g$ is one of the Lie superalgebras from Table \ref{tabel0} (with the exception of $spo(2|m)$, $m\le 2$, for which the answer is well known), and $k$ lies in the unitary range. These modules are parametrized by pairs $\nu\in (\h_\R^\natural)^*$ and $\ell_0\in\R$. The following are necessary conditions for  unitarity of $L^W(\nu,\ell_0)$: 
 \begin{enumerate}
 \item[(NS1)] the affine levels $M_i(k)$  for $\g^\natural_i$,  explicitly displayed in \cite[Table 2]{KMP1}, are non-negative integers;
 \item[(NS2)] $\nu\in P^+_k=\{\text{dominant integral weights for }\g^\natural\text{ such that }\nu(\theta_i^\vee)\le M_i(k)\}$ where $\theta_i$ are highest roots of $\g_i^\natural$.
 \item[(NS3)]$\ell_0\ge A(k,\nu)$, where $A(k,\nu)$ is defined in \cite[formula (8.11)]{KMP1}, and $\ell_0=A(k,\nu)$ if $\nu$ is an extremal weight (i.e. $\nu(\theta_i^\vee)>M_i(k)+\chi_i$ for some $i$, $\chi_i$ being displayed in \cite[Table 2]{KMP1}). \end{enumerate}
 \begin{theorem}\cite{KMP1} The highest weight modules
\begin{equation}\label{hwu}
\left\{L^W(\nu,\ell_0)\mid \nu\in P^+_k,\nu\text{ non-extremal, }\ell_0\ge A(k,\nu)\right\}
\end{equation}
are unitary.
\end{theorem}
\begin{Conjecture}\label{NS}
The set of unitary highest weight representations of $\Wu$ is the union of \eqref{hwu} and 
\begin{equation*}
\left\{L^W(\nu,A(k,\nu))\mid \nu\in P^+_k,\nu\text{ extremal}\right\}.
\end{equation*}
\end{Conjecture}

 Actually in \cite{KMP1} we studied unitarity of the $\Wu$-modules; however it has been proved in \cite[Theorem 5.1]{AKMP} that any unitary $\Wu$-module descends to the simple $W$-algebra $\Ws$. 
 
The Ramond case was dealt with in the paper \cite{KMPR}. We had to introduce a suitable notion of extremality, called {Ramond extremality} (cf. \cite[(9.3)]{KMPR}) and a constant $A_R(k,\nu)$ (cf. \cite[(6.31)]{KMPR}  in place of $A(k,\nu)$. 
There are two types  of these modules satisfying the necessary conditions of unitarity:
\begin{enumerate}
\item[(R1)] the modules $L_R^W(\nu,\ell)$ with $\nu\in P^+_k$ not Ramond extremal and $\ell\ge  A_R(k,\nu)$;
\item[(R2)] the modules $L_R^W(\nu,\ell)$ with the weight $\nu\in P^+_k$  Ramond extremal,
in which case $
\ell=A_R(k,\nu)$.
\end{enumerate}
In case (R1) we prove that these modules are indeed unitary
using    the spectral flow \cite{KMPSF} (in \cite{KMPR} we proved the result either using  Kac-Wakimoto free field realization or assuming true the exactness of the twisted quantum Hamiltonian reduction functor which, in the untwisted case,  holds by the work of  Arakawa). In case (R2) we don't know how to prove unitarity. Summing up
\begin{theorem} \cite{KMPR} The highest weight Ramond twisted modules
\begin{equation}\label{hwur}
\left\{L^W_R(\nu,\ell_0)\mid  \nu \text{ is not  Ramond extremal}, \ell_0\ge A_R(k,\nu)\right\}
\end{equation}
are unitary.
\end{theorem}
\begin{Conjecture}\label{R}
The set of unitary highest weight representations of $\Wu$ is the union of \eqref{hwur} and
\begin{equation*}
\left\{L^W_R(\nu,A(k,\nu))\mid\nu \text{ is   Ramond extremal}\right\}.
\end{equation*}
\end{Conjecture}
In \cite{KMP1}, \cite{KMPR} we rephrased the results of \cite{M,ET1} on unitarity of  highest weight  modules for the $N=3,4$ superconformal algebras to provide support for 
Conjectures \ref{NS} and \ref{R}. We do the same in the present paper for the big $N=4$ superconformal algebra, so that our conjectures  hold for $\g=psl(2|2), spo(2|3)$ and $D(2,1;a)$, but remain open for the extremal and Ramond extremal  representations of the minimal $W$-algebra,  attached to $\g=spo(2|m)$ with $m>4$, $F(4)$ and $G(3)$.
\begin{remark} As in the present paper, the big $N=4$ superconformal algebra turns out to be a challenging and relevant case to test general conjectures. Very recently, 
Creutzig, Gaiotto and Linshaw found the decomposition rules  related to the conformal embedding  (discovered in \cite[Proposition 4.2 (III)]{AKMPP-JJM})
$$ V_{-\frac{a+3}{2}}(sl_2)\otimes V_{-\frac{a^{-1}+3}{2}}(sl_2)\hookrightarrow W^{\min}_{\frac{1}{2}}(D(2,1;a)).$$ Their result  (cf \cite[Corollary 2.6]{CGL}) provides evidences for some conjectures about vertex algebras which emerge in gauge
theory constructions associated to the geometric Langlands program. 
\end{remark}
\section{Compact complex homogeneous spaces and cubic Dirac operators}\label{setup}
Let $\mathbb G$ be a connected  compact group, $\mathbb A$ a connected closed subgroup. We say that  $\mathbb M=\mathbb G/\mathbb A$ carries an invariant  complex structure if $\mathbb M$ is a complex manifold  with $\mathbb G$ acting holomorphically. Let $\g_\R$ be the Lie algebra of $\mathbb G$ and $\aa_\R$ the Lie subalgebra of $\mathbb A$. Fix a $\mathbb G$-invariant negative definite  symmetric bilinear  form $(\cdot,\cdot)$ on $\g_\R$ and let $\mathfrak q_\R$ be the orthocomplement of $\aa_\R$ in $\q_\R$.  Let $\g$, $\aa$, and $\mathfrak q$ be the complexifications of $\g_\R$, $\aa_\R$, and $\q_\R$ respectively. 
 Let $J\in Hom_{\aa}(\q,\q)$ be the $\mathbb A$-invariant integrable almost complex structure corresponding to the complex  structure on $\mathbb M$. 
We make the further assumption that $(J(q),J(q'))=(q,q')$ for all $q,q'\in\q$. 
 
As explained in the
 exposition  given in Section 2 of \cite{BGP},  there is a parabolic subalgebra $\mathfrak p=\mathfrak m_J\oplus \mathfrak n^+_J$ of $\g$ such that $\aa\subset \mathfrak m_J$ and  $[\aa,\aa]=[\mathfrak m_J,\mathfrak m_J]$. It follows that 
 $$
 \mathfrak m_J=\aa\oplus \mathfrak t
 $$
 with $\mathfrak t$ contained in the center of $\mathfrak m$. Let $\mathfrak p^-_J=\mathfrak m_J\oplus \mathfrak n^-_J$ be the opposite parabolic subalgebra. It can be assumed that $J(x)=\pm\sqrt{-1}x$ for all $x\in\mathfrak n^\pm_J$ and that 
$J(\mathfrak t)\subset \mathfrak t$.

Since $J^2=-I_{\q}$, we can decompose $\mathfrak t$ as $\mathfrak t^+_J\oplus \mathfrak t^-_J$, with $\mathfrak t^\pm_J$ the eigenspaces corresponding to the $\pm\sqrt{-1}$ eigenvalues of $J$.
Set
\begin{equation}\label{decu}
\mathfrak u_J^\pm=\mathfrak t^\pm_J\oplus \mathfrak n^\pm_J
\end{equation}
and note that $\mathfrak u_J^\pm$ are Lie subalgebras of $\g$.

Choosing a Cartan subalgebra $\h_\aa$ of $\aa$, we choose $\h=\h_\aa\oplus \mathfrak t$ as Cartan subalgebra of $\g$ and let $\D$ be the corresponding set of roots. Set
\begin{equation}\label{dpl}
\Dp(J)=\{\a\in\D\mid \g_\a\subset \mathfrak n^+\}
\end{equation}
and note that we can always choose a set of positive roots $\Dp$ so that $\Dp(J)\subset \Dp$.
  
 Note furthermore that, since we are assuming  that $(Jq,Jq')=(q,q')$, then  $\mathfrak u^\pm_J$ are isotropic with respect 
to $(\cdot,\cdot)$ and, due to the nondegeneracy of $(\cdot,\cdot)_{|\mathfrak q\times \mathfrak q}$, $\mathfrak u^+_J$ and $\mathfrak u_J^-$ are paired by $(\cdot,\cdot)$. In particular $\dim\,\mathfrak u^+_J=\dim \mathfrak u_J^-$. Note also that $[\aa,\mathfrak u^\pm_J]\subset \mathfrak u^\pm_J$.

\begin{remark}
If $J\in Hom_{\aa}(\mathfrak q,\mathfrak q)$  is such that 
 $J^2=-I_\q$, let $\mathfrak u^\pm_J$ be the eigenspaces corresponding to the eigenvalues $\pm\sqrt{-1}$.  Extend $J$ to $\g$ by setting $J_{|\aa}=0$.  Then  $\mathfrak u^\pm_J$ are subalgebras of $\g$ if and only if
\begin{equation}\label{NT}
J[X,Y]=[JX,Y]+[X,JY]+J[JX,JY],\text{ for all $X,Y\in\g$} .
\end{equation}
\end{remark}

Extend $(\cdot,\cdot)$ to a bilinear symmetric nondegenerate invariant form on $\g$. Let $C_\g$ be the corresponding Casimir operator. Let $\hvee$ be the bilinear form defined by $\hvee(x,y)=\half(C_\g\cdot x,y)$. Let $\k$ be a symmetric invariant form on $\g$ such that $\k+\hvee$ is nondegenerate. Write 
\begin{equation}\label{decsimpleideals}
\g=\oplus_{i=1}^r\g_i
\end{equation}
with $\g_i$ either simple ideals or subspaces of the center of $\g$ such that 
  $\k_{|\g_i\times \g_i}=k_i(\cdot,\cdot)_{|\g_i\times \g_i}$. Note that and $\hvee_{|\g_i\times \g_i}=h^\vee_i(\cdot,\cdot)_{|\g_i\times \g_i}$. This implies that $(\k+\hvee)_{|\g_i\times \g_i}=(k_i+h^\vee_i)(\cdot,\cdot)_{|\g_i\times \g_i}$ with $k_i+h^\vee_i\ne0$.

By construction
$$
\aa=\bigoplus_{i=0}^r(\aa\cap\g_i)
$$
and, since the decomposition \eqref{decsimpleideals} is orthogonal,
$$
\q=\bigoplus_{i=0}^r(\q\cap\g_i).
$$
Let $C_\aa$ be the Casimir element of $\aa$ corresponding to  $(\cdot,\cdot)_{|\aa\times \aa}$ and set $\ov{\mathbf h}^\vee(a,a')=\half(C_\aa\cdot a,a')$ for $a,a'\in\aa$. Let  $\{2\ov{h}_{ij}^\vee\}$ be the set of eigenvalues of the action of $C_\aa$ on $\aa\cap\g_i$ and let $\mathfrak a\cap \g_i=\oplus_j \mathfrak a_{ij}$ be the corresponding eigenspace decomposition.
Since $\aa$ is reductive and its center is contained in a Cartan subalgebra of $\g$, the action of $C_\aa$ on $\g$ (hence on $\q$) is semisimple.

If $V$ is a vector space, let $\ov V$ be the space $V$ considered as a totally odd space. If $V$ is equipped with a nondegenerate symmetric form $B$, we let  $F(\ov V,B)$ (or simply $F(\ov V)$) be the fermionic vertex algebra generated by the fields $\ov v$, $v\in V$, with $\l$-bracket 
$$
[\overline v_\l\overline v']=B(v,v').
$$

Let $\{v_i\}$ be a basis of $V$ and let $\{v^i\}$ be its dual basis with respect to $B$ (i. e. $B(v_i,v^j)=\d_{ij}$). It is well-known (see e.g. \cite[Section 3.4]{KMPX}), that the map 
\begin{equation}\label{conftheta}
\Theta_V:so(V,B))\to F(\ov V),\ T\mapsto \half\sum :\ov{T(v_i)}\ov v^i:
\end{equation}
extends to define a vertex algebra homomorphism
$$
\Theta_V:V^1(so(V,B))\to F(\ov V).
$$
Here $V^1(so(V,B))$ is the universal affine vertex algebra with $\la$-bracket
$$
[T_\la S]=[T,S]+\half\la Tr(TS)\vac,\ T,S\in so(V,B).
$$

Fix a basis $\{q_{ij}\}$ of $\mathfrak q\cap\g_i$ and let $\{q^{ij}\}$ be the dual basis w.r.t. $(\cdot,\cdot)$. Set 
$$
G_i=\sum_{j}:q_{ij}\overline{q}^{ij}:-\frac{1}{6}\sum_{j,s}:\overline{[q_{ij},q_{is}]}_{\mathfrak q}\overline{q}^{ij}\overline{q}^{is}:\in V^\k(\g)\otimes F(\overline{\mathfrak q})
$$
(cf. \cite{KacT}, \cite[Section 4.3]{DK} and  \cite[Lemma 4.5]{KMPD}).
Here, if $a\in\g$, we let $a_{\mathfrak q}$ denote the orthogonal projection of $a$ onto $\mathfrak q$.
The cubic Dirac operator corresponding to the pair $(\g,\mathfrak a)$ is
$$G=\sum_iG_i.
$$
Set also 
$$
\widetilde G_i=\frac{1}{2\sqrt{(k_i+h^\vee_i)}}G_i,\quad \widetilde G=\sum_i \widetilde G_i.
$$
Let $L^\g$ be the Virasoro vector in $V^\k(\g)$ given by Sugawara construction and let
$$
L^{\overline{\mathfrak q}}=\half\sum_{ij}:\partial(\overline{q}_{ij})\overline{q}^{ij}:$$
be the Virasoro vector for $F(\overline{\mathfrak q})$.  Let $\a$ be the bilinear form on $\mathfrak a$, defined by 
$$
\a=(\k+\hvee)_{|\aa\times\aa}-\ov{\mathbf h}^\vee.
 $$

Set, for $a\in\aa$, 
\begin{equation}\label{xa}
a^{\mathfrak a}=a+\Theta_{\q}(ad(a)).
\end{equation}
Recall  that the map 
 $$
 a\mapsto a^{\mathfrak a}
 $$
 extends to define an embedding
 $$
 V^\a(\mathfrak a)\to V^\k(\g)\otimes F(\overline{\mathfrak q}).
 $$

 Let $L^{\mathfrak a}$  be the image in $V^\k(\g)\otimes F(\overline{\mathfrak q})$ of the Virasoro element of $ V^\a(\mathfrak a)$.
Then, if 
\begin{equation}\label{L}L=L^\g-L^{\mathfrak a}+L^{\overline{\mathfrak q}},\end{equation}
we have
\begin{equation}\label{NSG}[L_\l L]=(\partial+2\l)L+\frac{\l^3}{12}C\vac,\quad [L_\l 
\widetilde G]=(\partial+\frac{3}{2}\l)\widetilde G,\quad
[\widetilde G_\l \widetilde G]=\half L+\frac{\l^2}{12}C\vac
\end{equation}
with
central charge
\begin{equation}\label{centralCharge}
C=\half\dim(\mathfrak q)+\sum_{i,j}\frac{\dim\g_i k_i-(k_i+h^\vee_i-\ov{h}_{ij}^\vee)\dim\aa_{ij} }{k_i+h^\vee_i}.
\end{equation}

If $J\in SO(\q,(\cdot,\cdot))$,  then $J$ extends to an automorphism $J$ of $V^k(\g)\otimes F(\overline{\mathfrak q})$ such that $J(a)=a$ for $a\in V^k(\g)$ and $J(\ov q)=\ov{J(q)}$ for $q\in\mathfrak q$.
Following \cite{STVP}, we set  
$$\widetilde G_J=J(\widetilde G).$$ If $J\in Hom_\aa(\q,\q)$, then $J(a^\aa)=a^\aa$ for all $a\in\aa$. In particular, $J(L)=L$. It follows that
\begin{equation}\label{NSGJ}
 [L_\l \widetilde G_J]=(\partial+\frac{3}{2}\l)\widetilde G_J,\quad[\widetilde G_J{}_\l \widetilde G_J]=\half L+\frac{\l^2}{12}C\vac.
\end{equation}

\section{The $N=2$ vertex algebra}The universal  $N=2$   vertex operator algebra is  freely generated by the fields $L$, $\widetilde G^\pm$, and $\widetilde J$, where $L$ a Virasoro  field with central
charge $C$, the fields $\widetilde J$ (resp. $\widetilde G^{\pm}$) are primary of
conformal weight $1$ (resp.~$3/2$), and the remaining
$\lambda$-brackets are as follows:
\begin{equation}\label{reN=2}
  [\widetilde J_{\lambda} \widetilde J ]= \tfrac{\lambda}{3}C\vac \, , \,
  [\widetilde G^{\pm}{}_{\lambda} \widetilde G^{\pm} ]=0 \, , \,
  [\widetilde J_{\lambda} \widetilde G^{\pm}] =\pm \widetilde G^{\pm} \, , \,
  [{\widetilde G^+}{}_{\lambda} \widetilde G^- ] = L + (\tfrac{1}{2}\partial+\lambda)
  \widetilde J+\tfrac{\lambda^2}{6}C \vac .
\end{equation}
\begin{proposition}\label{uniqueN=2}
Assume that in a VOA $V$ with conformal vector $L$ of central charge $C$ there exist primary odd fields $\widetilde G^\pm$ of conformal weight $\frac{3}{2}$.  Set   $\widetilde J=\widetilde G^+_{(1)}\widetilde G^-$ and assume furthermore that
\begin{align}\label{n22}[\widetilde G^+{}_\l\widetilde G^+]&=0,\\
\label{n23}[\widetilde G^-{}_\l\widetilde G^-]&=0,\\
\label{n21}
[\widetilde G^+{}_\l\widetilde G^-]&=L+\left(\half\partial +\l \right)\widetilde J +\l^2c \vac
\end{align}
with $c\in\C$.
Then $C=6c$ and relations \eqref{reN=2} hold. Therefore $L$, $\widetilde G^\pm$, and $\widetilde J$ generate a quotient of the universal  $N=2$ vertex algebra.\end{proposition} 
\begin{proof}
From
$$
[L_\l[\widetilde G^+{}_\mu \widetilde G^-]]=[\widetilde G^+{}_\mu[L_\l \widetilde G^-]]+[[L_\l \widetilde G^+]_{\l+\mu}\widetilde G^-]
$$
we obtain
$$
\begin{aligned}
&[L_\l L]+\half[L_\l\partial \widetilde J]+\mu [L_\l\widetilde J]=[\widetilde G^+{}_\mu(\partial \widetilde G^-+\tfrac{3}{2}\l \widetilde G^-)]+[(\partial \widetilde G^++\tfrac{3}{2}\l \widetilde G^+)_{\l+\mu}\widetilde G^-]\\
&=[\widetilde G^+{}_\mu\partial \widetilde G^-]+\tfrac{3}{2}\l [\widetilde G^+{}_\mu \widetilde G^-]+[\partial \widetilde G^+_{\l+\mu}\widetilde G^-]+\tfrac{3}{2}\l [\widetilde G^+{}_{\l+\mu}\widetilde G^-]\\
&=\partial[\widetilde G^+{}_\mu \widetilde G^-]+\mu[\widetilde G^+{}_\mu \widetilde G^-]+\tfrac{3}{2}\l [\widetilde G^+{}_\mu \widetilde G^-]-(\l+\mu)[\widetilde G^+_{\l+\mu}\widetilde G^-]+\tfrac{3}{2}\l [\widetilde G^+{}_{\l+\mu}\widetilde G^-]\\
&= 
  \partial L+ \mu \partial\widetilde J+\half
   \partial^2\widetilde J+\frac{c}{2} \lambda ^3\vac+\half \lambda ^2 \widetilde J
  +  \lambda  \mu\widetilde J +2
   \lambda  L+ \lambda 
   \partial \widetilde J
\end{aligned}
$$
from which it follows that 
$[L_\l \widetilde J]=(\partial+\l)\widetilde J$,
so $\widetilde J$ is primary of conformal weight 1.
It follows that
$$
[L_\l\partial \widetilde J]=\partial^2\widetilde J+2\l\partial \widetilde J+\l^2\widetilde J,
$$
hence $[L_\l L]=
  \partial L+\frac{c}{2} \lambda ^3\vac+2
   \lambda  L$. This implies that $C=6c$ and, in particular 
   $$
   [{\widetilde G^+}{}_{\lambda} \widetilde G^- ] = L + (\tfrac{1}{2}\partial+\lambda)
  \widetilde J+\tfrac{\lambda^2}{6}C \vac.
  $$
From
$$
[\widetilde G^+{}_\l[\widetilde G^+{}_\mu \widetilde G^-]]=-[\widetilde G^+{}_\mu[\widetilde G^{+}{}_\l \widetilde G^-]]
$$
we obtain
$$
\begin{aligned}
&\half\partial \widetilde G^+{} +\tfrac{3}{2}\l \widetilde G^+{}+\half\partial [\widetilde G^+{}_\l\widetilde J]+(\half\l+\mu) [\widetilde G^+{}_\l\widetilde J] \\&=-\half\partial \widetilde G^+- \tfrac{3}{2}\mu \widetilde G^+{}-\half\partial[\widetilde G^+{}_\mu\widetilde J]-(\half\mu+\l)[\widetilde G^+{}_\mu\widetilde J],\end{aligned}
$$
from which it follows that $\widetilde G^+{}_{(j)} \widetilde J=0$ for $j\ge1$ and
$
\mu \widetilde G^+{}_{(0)}\widetilde J=-\tfrac{3}{2}\mu \widetilde G^+-\half\mu \widetilde G^+{}_{(0)}\widetilde J
$. This relation  means that $\widetilde J_{(0)} \widetilde G^+=\widetilde G^+$, or
 \begin{equation}\label{n25}
[\widetilde J_\l \widetilde G^+]=\widetilde G^+.
\end{equation}
From
$$
[\widetilde G^-{}_\l[\widetilde G^+{}_\mu \widetilde G^-]]=[[\widetilde G^-{}_\l \widetilde G^+]_{\l+\mu}\widetilde G^-]
$$
we obtain
$$
\begin{aligned}
&\half\partial \widetilde G^-{} +\tfrac{3}{2}\l \widetilde G^-{}+\half\partial [\widetilde G^-{}_\l\widetilde J]+(\half\l+\mu) [\widetilde G^-{}_\l\widetilde J] =\partial \widetilde G^-+ \tfrac{3}{2}(\l+\mu) \widetilde G^--\half(\l-\mu)[\widetilde J_{\l+\mu}\widetilde G^-],\end{aligned}
$$
from which it follows that $\widetilde J_{(j)}\widetilde G^- =0$ for $j\ge1$ and
$
\half\mu\widetilde J_{(0)} \widetilde G^-+\tfrac{3}{2}\mu \widetilde G^-=-\mu \widetilde J_{(0)}\widetilde G^-
$. This relation means that $\widetilde J_{(0)}\widetilde G^-=-\widetilde G^-$, or 
 \begin{equation}\label{n26}
[\widetilde J_\l \widetilde G^-]=-\widetilde G^-.
\end{equation}
Now compute

$$
\begin{aligned}
[[\widetilde G^+{}_\l \widetilde G^-]_\mu[\widetilde G^+{}_\nu \widetilde G^-]]&=[(L+\left(\half\partial +\l \right)\widetilde J +\frac{\l^2}{6}C)_\mu(L+\left(\half\partial +\nu \right)\widetilde J +\frac{\nu^2}{6}C\vac)]\\
&=\l\nu[\widetilde J_\mu\widetilde J]+f_1(\l,\mu,\nu),
\end{aligned}
$$
where $f_1(\l,\mu,\nu)$ is a polynomial expression in $\l,\nu$ such that the coefficient of $\l\nu$ is $0$.
On the other hand
$$
\begin{aligned}
[[\widetilde G^+{}_\l \widetilde  G^-]_\mu[\widetilde G^+{}_\nu \widetilde G^-]]&=[\widetilde G^+{}_\nu[[\widetilde G^+{}_\l \widetilde G^-]_\mu \widetilde G^-]]+[[[\widetilde G^+{}_\l \widetilde G^-]_\mu \widetilde G^+{}]_{\mu+\nu}\widetilde G^-]\\
&=[\widetilde G^+{}_\nu[( L+\left(\half\partial +\l \right)\widetilde J )_\mu \widetilde G^-]]+[[(L+\left(\half\partial +\l \right)\widetilde J )_\mu \widetilde G^+{}]_{\mu+\nu}\widetilde G^-]\\
&=\l([\widetilde G^+{}_\nu[\widetilde J_\mu \widetilde G^-]]+[[\widetilde J_\mu \widetilde G^+]_{\mu+\nu} \widetilde G^-])+f_2(\l,\mu,\nu)\\
&=\l(-[\widetilde G^+{}_\nu \widetilde G^-]+[\widetilde G^+_{\mu+\nu} \widetilde G^-])+f_2(\l,\mu,\nu)\\
&=\l(\mu\widetilde J +\frac{\mu^2+2\mu\nu}{6}C \vac)+f_2(\l,\mu,\nu),
\end{aligned}
$$
where $f_2(\l,\mu,\nu)$  is an expression polynomial in $\l,\nu$ such that the coefficient of $\l\nu$ is $0$.
This implies that
\begin{equation}\label{n27}
[\widetilde J_\mu \widetilde J]=\frac{\mu}{3}C \vac.
\end{equation}
\end{proof}

\section{Coset realisation of $N=2$ vertex algebra}\label{cosetN=2}
In this Section we review the construction of $N=2$ coset models  attached to an integrable complex structure on a homogeneous space  as given in \cite{Kazamachar} and \cite{Kazamanew}.

Following the setting of Section \ref{setup}, let $J\in SO(\q,(\cdot,\cdot))\cap Hom_\aa(\q,\q)$ be the integrable almost complex structure attached to the complex structure on $\mathbb M=\mathbb G/\mathbb A$.

Set
\begin{equation}\label{Gtildes}
G^{\pm}= G\pm\sqrt{-1} G_J,\quad
\widetilde G^\pm=\widetilde G\pm\sqrt{-1} \widetilde G_J.
\end{equation}

Fix a basis $\{u_i\}$ of $\mathfrak u^+_J$ and let $\{u^i\}$ be the basis of $\mathfrak u_J^-$ dual to $\{u_i\}$. With these notations the explicit expression for $G^\pm$ is
\begin{equation}\label{gpgm}
G^{+}=2\sum_i:u_i\overline{u}^i:-\sum_{i,j}:\overline{[u_i,u_j]}\overline{u}^i\overline{u}^j:,\ G^-=2\sum_i:u^i\overline{u}_i:-\sum_{i,j}:\overline{[u^i,u^j]}\overline{u}_i\overline{u}_j:.
\end{equation}

Set 
\begin{equation}\label{Jtilde}
\widetilde J=-\sum_j:\overline u_{j}\overline u^{j}:+\sum_{i,j} \frac{1}{k_i+h_i^\vee}\sum_j[u_{ij},u^{ij}]+\sum_{i,j,t,q} \frac{1}{k_i+h_i^\vee}([u_{q},[u^{ij},u_{ij}]],u^{t}):\ov u_{t}\ov u^{q}:.
\end{equation}

Choose the root vectors $x_\a$, $\a\in\D$,  so that $(x_\a,x_{\be})=\d_{\a,-\be}$. Let $\Delta_i$ be the set of roots of $\g_i$. As a  basis of $\mathfrak n^+_J\cap \g_i$, we can choose 
$$
\{x_\a\mid \a\in\Dp(J)\cap \D_i\}.
$$
Its dual basis  is 
$$
\{x_{-\a}\mid \a\in\Dp(J)\cap\D_i\}.
$$

It follows that
\begin{equation}\label{rotilde}
\sum_j[u_{ij},u^{ij}]=\sum_{\a\in\Dp(J)\cap\D_i}[x_{\a},x_{-\a}]=\sum_{\a\in\Dp(J)\cap\D_i}h_\a=2h_{(\rho_i-\rho^\aa_i)}
\end{equation}
where $\rho_i$, $\rho^\aa_i$ are the $\rho$-vectors of $\g_i$ and $\aa_i$ respectively. Set 
$$\widetilde \rho=\sum_i\frac{1}{k_i+h_i^\vee}h_{\rho_i-\rho^\aa_i},
$$
so that we can write
$$
\tilde J=-\sum_i:\overline u_i\overline u^i:+2\widetilde \rho+ 2\sum_{t,q}([\widetilde \rho,u^t],u_q):\ov u^q\ov u_t:.
$$
Since $\widetilde \rho\in\h$, it is clear that $\mathfrak u_J^\pm$ are stable under $ad(\widetilde\rho)$, so we can write
$$
\begin{aligned}
\sum_{t,q}([\widetilde \rho,u_q],u^t):\ov u_t\ov u^q:&=\half \sum_{t,q}([\widetilde \rho,u_q],u^t):\ov u_t\ov u^q:+\half\sum_{t,q}([\widetilde \rho,u^q],u_t):\ov u^t\ov u_q:\\
&=\half \sum_{q}:\ov{[\widetilde \rho,u_q]}\ov u^q:+\half\sum_{q}:\ov{[\widetilde \rho,u^q]}\ov u_q:=\Theta_\q(ad(\widetilde\rho)).
\end{aligned}
$$
Observe also that 
$$
\sum_i:\overline u_i\overline u^i:=-\tfrac{\sqrt{-1}}{2}\sum_i:\overline{J u_i}\overline u^i:-\tfrac{\sqrt{-1}}{2}\sum_i:\overline{Ju^i}\overline u_i:=-\sqrt{-1}\Theta_\q(J),
$$
so
\begin{equation}\label{betterJtilde}
\tilde J=\sqrt{-1}\Theta_\q(J)+2\widetilde \rho+ 2\Theta_\q(ad(\widetilde \rho)).
\end{equation}

We are now ready to state the main result of this section.

\begin{proposition}\label{ParabisN=2} Let 
\begin{equation}\label{smallc}
c_i=\frac{2 k_i\dim \mathfrak u^+_J\cap\g_i +\sum_{j,r}([u_{ij},u_{ir}],[u^{ir},u^{ij}])}{4(k_i+h_i^\vee)},\ c=\sum_i c_i
\end{equation}
and let $C$ be as in \eqref{centralCharge}. 
Then 
$C=6c$ and  the fields $L$, $\widetilde G^\pm$ (given in \eqref{Gtildes}) and  $\widetilde J$ (given in \eqref{betterJtilde}) generate a quotient of  a $N=2$ superconformal algebra with central charge $C$.
\end{proposition}

In light of Proposition \ref{uniqueN=2}, the proof of Proposition \ref{ParabisN=2} reduces to checking that  relations \eqref{n22}, \eqref{n23}, \eqref{n21} hold with $\widetilde G^\pm$ as in \eqref{Gtildes},  $\widetilde J$ given by \eqref{betterJtilde}, and
 $c$ as in \eqref{smallc}. This computation will take up the rest of this section.

We will use the following formulas  for $\la$-brackets in $V^{\mathbf k}(\g)\otimes F(\ov\q)$ that we checked  using the Mathematica\texttrademark  \ package \verb|lambda|. 
 If 
 $a_i\in 
\g$ and $x_i,y_i\in\mathfrak q$, then
\begin{equation}\label{f1}\begin{aligned}
&[:a_1\overline  
x_1:_\l :a_2\overline  
y_1:]=:[a_1,a_2]\overline  
x_1\overline y_1:{ +(x_1,y_1)\partial[a_1,a_2]}+(x_1,y_1):a_2a_1:\\
&+\mathbf k(a_1,a_2):\partial \overline x_1\,\overline y_1: +\mathbf k(a_1,a_2)\l :\overline x_1\overline y_1:+\l(x_1,y_1)[a_1,a_2]+\half \l^2\mathbf k(a_1,a_2)(x_1,y_1)\vac \\
&=:[a_1,a_2]\overline  
x_1\overline y_1:+{ (x_1,y_1):a_1a_2:}+\mathbf k(a_1,a_2):\partial \overline x_1\,\overline y_1: \\
&+\mathbf k(a_1,a_2)\l :\overline x_1\overline y_1:+\l(x_1,y_1)[a_1,a_2]+\half \l^2 \mathbf k(a_1,a_2)(x_1,y_1)\vac,
\end{aligned}
\end{equation}
\begin{align}\label{f2}[:a_1\overline
x_1:_\la:\overline  y_1\overline y_2\overline y_3:]&=(x_1,y_1):a_1\overline y_2\overline
y_3:-(x_1,y_2):a_1\overline  y_1\overline y_3:+(x_1,y_3):a_1\overline 
y_1\overline y_2:,
\end{align} 
\begin{equation}\label{fff3}
\begin{aligned}
[:\ov x_1&\overline{x}_2:_\la : y_1\overline{y}_2:]= (x_2,y_2):y_1\ov x_1:-(x_1,y_2):y_1\ov x_2:,
\end{aligned}
\end{equation}
\begin{equation}\label{ffff3}
\begin{aligned}
[:x_1&\overline{x}_2:_\la : \overline{y}:]= (x_2,y)x_1,
\end{aligned}
\end{equation}
\begin{equation}\label{fffff3}
\begin{aligned}
[:\ov x_1 \ov x_2 \ov  x_3:_\la \ov y]=(y,x_3):\ov x_1\ov x_2:-(y,x_2) : \ov x_1\ov x_3:+(y,x_1):\ov x_2\ov x_3:,
\end{aligned}
\end{equation}
\begin{equation}\label{1lambda2}
\begin{aligned}
[:\ov x_1 \ov x_2 :_\la \ov y]=-(y,x_1)x_2+(y,x_2)x_1.
\end{aligned}
\end{equation}
Let $M$  be the matrix $((x_i,y_j))$. We use the standard notation for minors, so $M_{i_1,i_2,\ldots}^{j_1,j_2,\ldots}$ is the minor obtained by taking off rows ${i_1,i_2,\ldots}$ and columns ${j_1,j_2,\ldots}$. With this notation we have
\begin{equation}\label{f3}
\begin{aligned}
[:\overline{x}_1&\overline{x}_2\overline{x}_3:_\la :\overline{y}_1\overline{y}_2\overline{y}_3:]=-\half\l^2det(M)\vac -\l\sum_{i,j}(-1)^{i+j}det(M_i^j):\ov x_i\ov y_j:\\
&+\sum_{i<j,r<s}(-1)^{i+j+r+s}det(M^{r,s}_{i,j}):\ov x_i\ov x_j\ov y_r\ov y_s:-\sum_{i,j}(-1)^{i+j}det(M_i^j):\partial \ov x_i\ov y_j:,
\end{aligned}
\end{equation}

\begin{equation}\label{ff3}
\begin{aligned}
[:\overline{x}_1&\overline{x}_2:_\la :\overline{y}_1\overline{y}_2\overline{y}_3:]=\sum_{i,j<k}(-1)^{i+j+k}det(M_i^{jk}):\ov x_i\ov y_j\ov y_k:+\la\sum_{j}(-1)^{j}det(M^j)\ov y_j.
\end{aligned}
\end{equation}

If $a\in\g$, we let $a_{\mathfrak u^+_J}$, $a_{\mathfrak u_J^-}$, $a_{\mathfrak a}$ denote the projections of $a$ onto $\mathfrak u^+_J$, $\mathfrak u_J^-$, and $\mathfrak a$ respectively corresponding to the decomposition
$$
\g=\mathfrak u_J^-\oplus \mathfrak a\oplus \mathfrak u^+_J.
$$
Assume first that $\aa=\mathfrak m$ so that $u^\pm_J=\mathfrak n^\pm_J$. In this case 
$$
\mathfrak u^\pm_J=\oplus_i (\mathfrak u^\pm_J\cap\g_i).
$$
Choose a basis $\{u_{ij}\}$ of $\mathfrak u_J^+\cap \g_i$ and let $\{u^{ij}\}$ be the basis of $\mathfrak u^-_J\cap\g_i$ dual to   $\{u_{ij}\}$.
Then 
$$
\widetilde G^\pm=\sum_i \widetilde G^\pm_i
$$
with 
\begin{equation}\label{gpgmi}
\begin{aligned}
\widetilde G^{+}_i=\frac{1}{2\sqrt{k_i+h^\vee_i}}\left(2\sum_{j}:u_{ij}\overline{u}^{ij}:-\sum_{j,t}:\overline{[u_{ij},u_{it}]}\overline{u}^{ij}\overline{u}^{it}:\right),\\ \widetilde G^-_i=\frac{1}{2\sqrt{(k_i+h^\vee_i)}}\left(2\sum_j:u^{ij}\overline{u}_{ij}:-\sum_{j,t}:\overline{[u^{ij},u^{it}]}\overline{u}_{ij}\overline{u}_{it}:\right),
\end{aligned}
\end{equation}
so that 
$[\widetilde G^\pm_\l \widetilde G^\pm]=\sum_i[(\widetilde G^\pm_i)_\l \widetilde G^\pm_i]$. We can therefore assume
 that $\g$ is either simple or abelian so that, dropping the index $i$, we can write
\begin{equation}\label{gsimpleformulas}
\begin{aligned}
&\widetilde G^\pm=\frac{1}{2\sqrt{k+h^\vee}}G^\pm,\\
&\widetilde J=\sum_j:\overline u^{j}\overline u_{j}:+ \frac{1}{k+h^\vee}\sum_j[u_{j},u^{j}]+ \frac{1}{k+h^\vee}\sum_{j,t,q}([u_{q},[u^{j},u_{j}]],u^{t}):\ov u_{t}\ov u^{q}:,\\
&c=\frac{2 k\dim \mathfrak u^+_J +\sum_{j,r}([u_{j},u_{r}],[u^{r},u^{j}])}{4(k+h^\vee)},\\
&L=\tfrac{1}{2(k+h^\vee)}(\sum_i:a_ia_i:+\sum_i:u^iu_i:+\sum_i:u_iu^i:-\sum_{j}:(a_j)^{\mathfrak a}(a_j)^{\mathfrak a}:)\\&+\half\sum_i:\partial \ov u_i\ov u^i:+\half\sum_i:\partial \ov u^i\ov u_i:,
\end{aligned}
\end{equation}
and it is enough to check that
\begin{align}\label{gpgpfinale}
[G^+{}_\l G^+]&=0,\\
\label{gmgmfinale}
[G^-{}_\l G^-]&=0,\\
[G^+{}_\l G^-]&=4(k+h^\vee)(L+\left(\half\partial +\l \right)\widetilde J +\l^2c \vac).\label{gpgmfinale}
\end{align}

We first compute $[{G^+}{}_\l G^-]$. Using \eqref{gpgm}, we find
$$
 \begin{aligned}
 &[{G^+}_\l G^-]=4\sum_{i,r}[:u_i\overline{u}^i:_{\l}:u^r\overline{u}_r:]-2\sum_{i,r,s}[:u_i\overline{u}^i:_{\l}:\overline{[u^r,u^s]}\overline{u}_r\overline{u}_s:]\\
 &-2\sum_{i,j,r}[:\overline{[u_i,u_j]}\overline{u}^i\overline{u}^j:_\l :u^r\overline{u}_r:]+\sum_{i,j,r,s}:\overline{[u_i,u_j]}\overline{u}^i\overline{u}^j:
_\l:\overline{[u^r,u^s]}\overline{u}_r\overline{u}_s:].
 \end{aligned}
 $$
Formulas \eqref{f1} and \eqref{f2} give

\begin{align}\label{f111}
&\begin{aligned}\sum_{i,r}[:u_i\overline{u}^i:_\l:u^r\overline{u}_r:]&=\sum_{i,r}:[u_i,u^r]\overline  
u^i\overline u_r:+\sum_{i}(:u_iu^i:+k:\partial \overline u^i\,\overline u_i:)\\
& +k\l \sum_i:\overline u^i\overline u_i:+\l \sum_i[u_i,u^i]+\half k\l^2\dim \mathfrak u^+_J,\end{aligned}
\\\label{f112}
&\sum_{i,r,s}[:u_i\overline{u}^i:_\l:\overline{[u^r,u^s]}\overline{u}_r\overline{u}_s]=2\sum_{r,s}:u_s\overline{[u^r,u^s]}\overline u_r:,\\\label{f113}
&\sum_{i,j,r}[:\overline{[u_i,u_j]}\overline{u}^i\overline{u}^j:_\l:u^r\overline u_r:]=2\sum_{i,j}:u^j\overline 
{[u_i,u_j]}\overline u^i:.
\end{align}
Using formulas  \eqref{f111},  \eqref{f112},  \eqref{f113}, we obtain 
$$
 \begin{aligned}
& 4\sum_{i,r}[:u_i\overline{u}^i:_{\l}:u^r\overline{u}_r:]-2\sum_{i,r,s}[:u_i\overline{u}^i:_{\l}:\overline{[u^r,u^s]}\overline{u}_r\overline{u}_s:]-2\sum_{i,j,r}[:\overline{[u_i,u_j]}\overline{u}^i\overline{u}^j:_\l :u^r\overline{u}_r:]\\
 &= 4\sum_{i,r}:[u_i,u^r]\overline  
u^i\overline u_r:+4\sum_{i}(:u_iu^i:+k:\partial \overline u^i\,\overline u_i:) +4k\l \sum_i:\overline u^i\overline u_i:+4\l \sum_i[u_i,u^i]\\&+2 k\l^2\dim \mathfrak u^+_J\vac -4\sum_{r,s}:u_s\overline{[u^r,u^s]}\overline u_r:-4\sum_{i,j}:u^j\overline 
{[u_i,u_j]}\overline u^i:,
\end{aligned}
$$
 and
$$
 \begin{aligned}
 &\sum_{i,j,r,s}[:\overline{[u_i,u_j]}\overline{u}^i\overline{u}^j:_\l:\overline{[u^r,u^s]}\overline{u}_r\overline{u}_s:]=
\sum_{i,j,r,s,t,q}(u^t,[u_i,u_j])(u_q,[u^r,u^s])[ :\ov u_t \overline{u}^i\overline{u}^j:_\l:\ov u^q\overline{u}_r\overline{u}_s:]=\\
&\l^2\sum_{i,j}([u_i,u_j],[u^j,u^i])+2\l\sum_{i,t,q}([u_q,u^i]_{\mathfrak u^+_J},[u_i,u^t]_{\mathfrak u_J^-}):\ov u_t\ov u^q:+4\l\sum_{i,j,r}([u_i,u_j],[u^j,u^r]):\ov u^i\ov u_r:\\
&+2\sum_{i,t,q}([u_q,u^i]_{\mathfrak u^+_J},[u_i,u^t]_{\mathfrak u_J^-}):\partial\ov u_t\ov u^q:+4\sum_{i,j,r}([u_i,u_j],[u^j,u^r]):\partial\ov u^i\ov u_r:\\
&-4\sum_{i,s,t,q}([u^t,u_i]_{\mathfrak u_J^-},[u^s,u_q]_{\mathfrak u^+_J}):\ov u^i\ov u^q\ov u_t\ov u_s:+\sum_{i,j,r,s}([u_i,u_j],[u^r,u^s])):\ov u^i\ov u^j\ov u_r\ov u_s:.
 \end{aligned}
$$
Since
$$
\begin{aligned}
 &\sum_{i,r}:[u_i,u^r]\overline  
u^i\overline u_r:-\sum_{r,s}:u_s\overline{[u^r,u^s]}\overline u_r:-\sum_{i,j}:u^j\overline 
{[u_i,u_j]}\overline u^i:\\
&=\sum_{i,r}:[u_i,u^r]\overline  
u^i\overline u_r:-\sum_{r,s,t}([u^r,u^s],u_t):u_s\overline{u}^t\overline u_r:-\sum_{i,j,t}([u_i,u_j],u^t):u^j\overline 
{u}_t\overline u^i:\\&= \sum_{i,r}:[u_i,u^r]\overline  
u^i\overline u_r:  
-\sum_{r,t}:[u_t,u^r]_{\mathfrak u^+_J}\overline{u}^t\overline u_r:-\sum_{i,t}:[u^t,u_i]_{\mathfrak u_J^-}\ov u_t\overline u^i:=\sum_{i,r}:[u_i,u^r]_{\mathfrak a}\overline  
u^i\overline u_r:
\end{aligned}
 $$
and
$$
\begin{aligned}
&\sum_{i,j,r,s}([u_i,u_j],[u^r,u^s])):\ov u^i\ov u^j\ov u_r\ov u_s:
\\&=\sum_{i,j,r,s}\left(([[u_i,,u^r],u_j],u^s))
+\sum_{i,j,r,s}([u_i,[u_j,u^r]],u^s))\right):\ov u^i\ov u^j\ov u_r\ov u_s:\\
&=2\sum_{i,j,r,s}([u_i,,u^r],[u_j,u^s])):\ov u^i\ov u^j\ov u_r\ov u_s:,\end{aligned}
$$
so that
$$
\begin{aligned}
&\sum_{i,j,r,s}([u_i,u_j],[u^r,u^s])):\ov u^i\ov u^j\ov u_r\ov u_s:-4\sum_{i,s,t,q}([u^t,u_i]_{\mathfrak u_J^-},[u^s,u_q]_{\mathfrak u^+_J}):\ov u^i\ov u^q\ov u_t\ov u_s:\\
&=2\sum_{i,j,r,s}([u_i,,u^r],[u_j,u^s])):\ov u^i\ov u^j\ov u_r\ov u_s:-4\sum_{i,s,t,q}([u^t,u_i]_{\mathfrak u_J^-},[u^s,u_q]):\ov u^i\ov u^q\ov u_t\ov u_s:\\
&=2\sum_{i,j,r,s}([u^r,u_i]_{\mathfrak a}+[u^r,u_i]_{\mathfrak u^+_J},[u^s,u_j]):\ov u^i\ov u^j\ov u_r\ov u_s:-2\sum_{i,s,t,q}([u^t,u_i]_{\mathfrak u_J^-},[u^s,u_q]):\ov u^i\ov u^q\ov u_t\ov u_s:\\
&=2\sum_{i,j,r,s}([u^r,u_i]_{\mathfrak a},[u^s,u_j]_{\mathfrak a}):\ov u^i\ov u^j\ov u_r\ov u_s:,
 \end{aligned}
$$
the final outcome is that
\begin{equation}\label{G+G-}
 \begin{aligned}
 &[{G^+}_\l G^-]=4\sum_{i,r}:[u_i,u^r]_{\mathfrak a}\overline  
u^i\overline u_r:+4\sum_{i}(:u_iu^i:+k:\partial \overline u^i\,\overline u_i:) \\
&+4k\l \sum_i:\overline u^i\overline u_i:+4\l \sum_i[u_i,u^i]+2 k\l^2\dim \mathfrak u^+_J+\l^2\sum_{i,j}([u_i,u_j],[u^j,u^i])\\
&+2\l\sum_{i,t,q}([u_q,u^i]_{\mathfrak u^+_J},[u_i,u^t]_{\mathfrak u_J^-}):\ov u_t\ov u^q:+4\l\sum_{i,j,r}([u_i,u_j],[u^j,u^r]):\ov u^i\ov u_r:\\
&+2\sum_{i,t,q}([u_q,u^i]_{\mathfrak u^+_J},[u_i,u^t]_{\mathfrak u_J^-}):\partial\ov u_t\ov u^q:+4\sum_{i,j,r}([u_i,u_j],[u^j,u^r]):\partial\ov u^i\ov u_r:\\
&+2\sum_{i,j,r,s}([u^r,u_i]_{\mathfrak a},[u^s,u_j]_{\mathfrak a}):\ov u^i\ov u^j\ov u_r\ov u_s:.
 \end{aligned}
\end{equation}
 Let $\{a_i\}$ be an orthonormal basis of $\mathfrak a$,  so that
 $$
 [u_i,u^r]_{\mathfrak a}=\sum_j( [u_i,u^r],a_j)a_j,
 $$ 
 hence
\begin{equation}\label{pra}
 \begin{aligned}
 &\sum_{i,r}:[u_i,u^r]_{\mathfrak a}\overline  
u^i\overline u_r:= \sum_{i,r,j}([u_i,u^r],a_j):a_j\overline  
u^i\overline u_r:\\
&=\half  \sum_{i,r,j}([u_i,u^r],a_j):a_j\overline  
u^i\overline u_r:+\half  \sum_{i,r,j}([u_i,u^r],a_j):a_j\overline  
u^i\overline u_r:\\
&=-\half  \sum_{r,j}:a_j\overline{[a_j,u^r]} 
\overline u_r:-\half  \sum_{i,j}:a_j  
\overline{[a_j,u_i]}\ov u^i :=-\sum_{j}:a_j\Theta_\q(ad(a_j)):,
\end{aligned}
\end{equation}
 and
\begin{equation}\label{hereistheerror}
 \begin{aligned}
 &\sum_{i,j,r,s}([u^r,u_i]_{\mathfrak a},[u^s,u_j]_{\mathfrak a}):\ov u^i\ov u^j\ov u_r\ov u_s:=\sum_{i,j,r,s,t}([u^r,u_i],a_t)(a_t,[u^s,u_j]):\ov u^i\ov u^j\ov u_r\ov u_s:\\
&=\frac{1}{4}\sum_{i,s,t}:\ov u^i\ov{ [a_t,u^s]}\ov{ [u_i,a_t]}\ov u_s:+\frac{1}{4}\sum_{j,r,t}:\ov{[a_t,u^r]}\ov u^j\ov u_r\ov {[u_j,a_t]}:\\&+\frac{1}{4}\sum_{r,s,t}:\ov{[a_t,u^r]}\ov{[a_t,u^s]}\ov u_r\ov u_s:
+\frac{1}{4}\sum_{i,j,t}:\ov u^i\ov u^j\ov{[u_i,a_t]}\ov {[u_j,a_t]}:\\
&=-\frac{1}{4}\sum_{i,s,t}:\ov{ [a_t,u_i]}\ov u^i\ov{ [a_t,u^s]}\ov u_s:-\frac{1}{4}\sum_{j,r,t}:\ov{[a_t,u^r]}\ov u_r\ov {[a_t,u_j]}\ov u^j:\\&-\frac{1}{4}\sum_{r,s,t}:\ov{[a_t,u^r]}\ov u_r\ov{[a_t,u^s]}\ov u_s:
-\frac{1}{4}\sum_{i,j,t}:\ov{[a_t,u_i]}\ov u^i\ov {[a_t,u_j]}\ov u^j:.
 \end{aligned}
\end{equation}
{ Using quasi-associativity, we find that
\begin{align*}
:\ov x_1\ov x_2\ov y_1\ov y_2:&=::\ov x_1\ov x_2::\ov y_1\ov y_2::\\&-(x_2,y_1):\partial\ov x_1\ov y_2:+(x_2,y_2):\partial\ov x_1\ov y_1:+(x_1,y_1):\partial\ov x_2\ov y_2:-(x_1,y_2):\partial\ov x_2\ov y_1:,
\end{align*}
so
$$
\begin{aligned}
&\sum_{i,s,t}:\ov{[a_t,u_i]}\ov u^i\ov{[a_t,u^s]}\ov u_s:=\sum_t:\sum_i:\ov{[a_t,u_i]}\ov u^i:\sum_s:\ov{[a_t,u^s]}\ov u_s::\\
&-\sum_{i,s,t}(u^i,[a_t,u^s]):\partial\ov{[a_t,u_i]}\ov u_s:+\sum_{i,s,t}(u^i,u_s):\partial\ov{ [a_t,u_i]}\,\ov{ [a_t,u^s]}:\\
&+\sum_{i,s,t}([a_t,u_i],[a_t,u^s]):\partial\ov u^i\ov u_s:-\sum_{r,s,t}([a_t,u_i],u_s):\partial\ov u^i\ov{ [a_t,u^s]}:\\
&=\sum_t:\sum_i:\ov{[a_t,u_i]}\ov u^i:\sum_s:\ov{[a_t,u^s]}\ov u_s::\\
&+\sum_{i,t}:\partial\ov{ [a_t,u_i]}\,\ov{ [a_t,u^i]}:+\sum_{i,s,t}([a_t,u_i],[a_t,u^s]):\partial\ov u^i\ov u_s:\\
&=\sum_t:\sum_i:\ov{[a_t,u_i]}\ov u^i:\sum_s:\ov{[a_t,u^s]}\ov u_s::\\
&-\sum_{q,t,s}(u_s, [a_t,[a_t,u^q]]):\partial\ov u_q\ov u^s:-\sum_{i,s,t}(u_i,[a_t,[a_t,u^s]]):\partial\ov u^i\ov u_s:.
\end{aligned}
$$
Similarly
$$
\begin{aligned}
&\sum_{j,r,t}:\ov{[a_t,u^r]}\ov u_r\ov{[a_t,u_j]}\ov u^j:=\sum_t:\sum_r:\ov{[a_t,u^r]}\ov u_r:\sum_j:\ov{[a_t,u_j]}\ov u^j::\\
&-\sum_{s,t}:\partial\ov u^s\ov{ [a_t,[a_t,u_s]]}:-\sum_{j,t}:\partial\ov{ [a_t,[a_t,u_j]]}\ov u^j:,
\end{aligned}
$$
as well as
$$
\begin{aligned}
&\sum_{r,s,t}:\ov{[a_t,u^r]}\ov u_r\ov{[a_t,u^s]}\ov u_s:=\sum_t:\sum_r:\ov{[a_t,u^r]}\ov u_r:\sum_s:\ov{[a_t,u^s]}\ov u_s::\\
&-\sum_{s,t,r}(u_r,[a_t,[a_t,u^s]]):\partial\ov u^r\ov u_s:-\sum_{r,t,s}(u_s, [a_t,[a_t,u^r]]):\partial\ov u_r\ov u^s:
\end{aligned}
$$
and
$$
\begin{aligned}
&\sum_{i,j,t}:\ov{[a_t,u_i]}\ov u^i\ov{[a_t,u_j]}\ov u^j:=\sum_t:\sum_i:\ov{[a_t,u_i]}\ov u^i:\sum_j:\ov{[a_t,u_j]}\ov u^j::\\
&-\sum_{j,t,r}(u_j,[a_t,[a_t,u^r]]):\partial\ov u_r\ov u^j:-\sum_{i,t,r}( u_i,[a_t,[a_t,u^r]]):\partial\ov u^i\ov u_r:.
\end{aligned}
$$
Thus \eqref{hereistheerror} gives 
\begin{equation}\label{hereiscorrect}
 \begin{aligned}
 &\sum_{i,j,r,s}([u^r,u_i]_{\mathfrak a},[u^s,u_j]_{\mathfrak a}):\ov u^i\ov u^j\ov u_r\ov u_s:=-\sum_{t}:\Theta_\q(a_t)\Theta_\q(a_t):\\
&+\sum_{j,t,r}(u_j,[a_t,[a_t,u^r]]):\partial\ov u_r\ov u^j:+\sum_{i,t,r}( u_i,[a_t,[a_t,u^r]]):\partial\ov u^i\ov u_r:.
 \end{aligned}
\end{equation}}
 Substituting \eqref{pra} and \eqref{hereiscorrect} in \eqref{G+G-}, we find
 $$
 \begin{aligned}
 &[{G^+}_\l G^-]=2\sum_{j}:a_ja_j:-2\sum_{j}:(a_j)^{\mathfrak a}(a_j)^{\mathfrak a}:+4\sum_{i}(:u_iu^i:+k:\partial \overline u^i\,\overline u_i:) \\&+4k\l \sum_i:\overline u^i\overline u_i:+4\l \sum_i[u_i,u^i]+\l^2\sum_{i,j}([u_i,u_j],[u^j,u^i])\\
 &+2\l\sum_{i,t,q}([u_q,u^i]_{\mathfrak u^+_J},[u_i,u^t]_{\mathfrak u_J^-}):\ov u_t\ov u^q:+4\l\sum_{i,j,r}([u_i,u_j],[u^j,u^r]):\ov u^i\ov u_r:\\
&+2\sum_{i,t,q}([u_q,u^i]_{\mathfrak u^+_J},[u_i,u^t]_{\mathfrak u_J^-}):\partial\ov u_t\ov u^q:+4\sum_{i,j,r}([u_i,u_j],[u^j,u^r]):\partial\ov u^i\ov u_r:\\
&+2 k\l^2\dim \mathfrak u^+_J {  +2\sum_{j,t,r}(u_j,[a_t,[a_t,u^r]]):\partial\ov u_r\ov u^j:+2\sum_{i,t,r}( u_i,[a_t,[a_t,u^r]]):\partial\ov u^i\ov u_r:,}\end{aligned}
 $$
so, by the explicit expression for $L$ given in \eqref{gsimpleformulas},
$$
 \begin{aligned}
 &[{G^+}_\l G^-]=4(k+h^\vee)L+2\sum_i\partial[u_i,u^i]-2h^\vee\sum_i:\partial \ov u_i\ov u^i:-2h^\vee\sum_i:\partial \ov u^i\ov u_i:\\
 &+2k\sum_{i}\partial: \overline u^i\,\overline u_i: +4k\l \sum_i:\overline u^i\overline u_i:+4\l \sum_i[u_i,u^i]\\
 &+\l^2\sum_{i,j}([u_i,u_j],[u^j,u^i])\vac +2\l\sum_{i,t,q}([u_q,u^i]_{\mathfrak u^+_J},[u_i,u^t]_{\mathfrak u_J^-}):\ov u_t\ov u^q:\\&+4\l\sum_{i,j,r}([u_i,u_j],[u^j,u^r]):\ov u^i\ov u_r:+2\sum_{i,t,q}([u_q,u^i]_{\mathfrak u^+_J},[u_i,u^t]_{\mathfrak u_J^-}):\partial\ov u_t\ov u^q:\\&+4\sum_{i,j,r}([u_i,u_j],[u^j,u^r]):\partial\ov u^i\ov u_r:+2 k\l^2\dim \mathfrak u^+_J\vac \\&+2\sum_{j,t,r}(u_j,[a_t,[a_t,u^r]]):\partial\ov u_r\ov u^j:+2\sum_{i,t,r}( u_i,[a_t,[a_t,u^r]]):\partial\ov u^i\ov u_r:.\end{aligned}
 $$
Note that
$$
 \begin{aligned}
 &\sum_{i}([u_q,u^i]_{\mathfrak u^+_J},[u_i,u^t]_{\mathfrak u_J^-})\\&=\sum_{i}([u_q,u^i],[u_i,u^t])-\sum_{i}([u_q,u^i]_{\mathfrak a},[u_i,u^t]_{\mathfrak a})-\sum_{i}([u_q,u^i]_{\mathfrak u_J^-},[u_i,u^t]_{\mathfrak u^+_J})\\
 &=\sum_{i}([u_q,u^i],[u_i,u^t])-\sum_{i}([u_q,u^i]_{\mathfrak a},[u_i,u^t]_{\mathfrak a})-\sum_{i}([u_q,u^i],u_r)(u^r,[u_i,u^t])\\
 &=\sum_{i}([u_q,u^i],[u_i,u^t])-\sum_{i,t,q}([u_q,u^i]_{\mathfrak a},[u_i,u^t]_{\mathfrak a})-\sum_{i,r}([u_r,u_q],u^i)(u_i,[u^t,u^r])\\
  &=\sum_{i}([u_q,u^i],[u_i,u^t])-\sum_{i}([u_q,u^i]_{\mathfrak a},[u_i,u^t]_{\mathfrak a}):-\sum_{r}([u_q,u_r],[u^r,u^t])
    \end{aligned}
 $$
 and, since
\begin{equation}\label{ujuj}
([u_q,u^i],[u_i,u^t])=([u_q,u_i],[u^i,u^t])+([u_q,[u^i,u_i]],u^t),
\end{equation}
we obtain
 $$
 \begin{aligned}
&\sum_{i}([u_q,u^i]_{\mathfrak u^+_J},[u_i,u^t]_{\mathfrak u_J^-})=\sum_{i}([u_q,[u^i,u_i]],u^t)-\sum_{i}([u_q,u^i]_{\mathfrak a},[u_i,u^t]_{\mathfrak a})\\
&=\sum_{i}([u_q,[u^i,u_i]],u^t)-\sum_{i,j}([u_q,u^i],a_j)(a_j,[u_i,u^t])\\
&=\sum_{i}([u_q,[u^i,u_i]],u^t)-\sum_{j}(u_q,[a_j,[a_j,u^t]]).
  \end{aligned}
 $$
Next observe that
$$
 \begin{aligned}
&\sum_{j}([u_i,u_j],[u^j,u^r])=\sum_{j}(u_i,[u_j,[u^j,u^r]])+\sum_{j}(u_i,[u^j,[u_j,u^r]])\\
&-\sum_{j}(u_i,[u^j,[u_j,u^r]])+\sum_{j}(u_i,[a_j,[a_j,u^r]])-\sum_{j}(u_i,[a_j,[a_j,u^r]]),
 \end{aligned}
 $$
 so, since for all $u\in\g$,
 $$
\sum_j[u_j,[u^j,u]]+\sum_j[u^j,[u_j,u]]+\sum_j[a_j,[a_j,u]]=2 h^\vee u,
$$
we have
 $$
  \begin{aligned}
&\sum_{j}([u_i,u_j],[u^j,u^r])=2h^\vee\d_{i,r}-\sum_{j}([u_i,u^j],[u_j,u^r])-\sum_{j}(u_i,[a_j,[a_j,u^r]).
\end{aligned}
$$
Now, using \eqref{ujuj}, we rewrite the last equation as
$$
\begin{aligned}
\sum_{j}([u_i,u_j],[u^j,u^r])=&2h^\vee\d_{i,r}-\sum_{j}([u_i,[u^j,u_j]],u^r)\\
&-\sum_{j}([u_i,u_j],[u^j,u^r])-\sum_{j}(u_i,[a_j,[a_j,u^r]),
\end{aligned}
 $$
 hence
 $$
 \begin{aligned}
2\sum_{j}([u_i,u_j],[u^j,u^r])=2h^\vee\d_{i,r}-\sum_{j}([u_i,[u^j,u_j]],u^r)-\sum_{j}(u_i,[a_j,[a_j,u^r]).
\end{aligned}
$$

Substituting in $[{G^+}_\l G^-]$ we find
\begin{equation}\label{Gpgm}
 \begin{aligned}
 [{G^+}_\l G^-]=&4(k+h^\vee)L\\
 &+2\partial\left((k+h^\vee)\sum_{i}: \overline u^i\,\overline u_i:+\sum_i[u_i,u^i]+\sum_{i,t,q}([u_q,[u^i,u_i]],u^t):\ov u_t\ov u^q:\right)\\
 & +4\l\left((k+h^\vee)\sum_i:\overline u^i\overline u_i:+ \sum_i[u_i,u^i]+\sum_{i,t,q}([u_q,[u^i,u_i]],u^t):\ov u_t\ov u^q:\right)\\
&+2 k\l^2\dim \mathfrak u^+_J\vac  +\l^2\sum_{i,j}([u_i,u_j],[u^j,u^i])\vac.\end{aligned}
\end{equation}
Recalling the definitions of $\widetilde J$ and $c$ given in \eqref{gsimpleformulas},  we see that \eqref{Gpgm} is precisely \eqref{gpgmfinale}.

We now compute $[G^+{}_\lambda G^+]$. 
Write explicitly 
\begin{align*}
[{G^+}_\l G^+]= &4\sum_{i,j}[:u_i\overline{u}^i:_\l:u_j\overline{u}^j:]-4\sum_{i,j,r}[:u_i\overline{u}^i:_\l:\overline{[u_j,u_r]}\overline{u}^j\overline{u}^r:]\\&+\sum_{i,j,r,s}[:\overline{[u_i,u_j]}\overline{u}^i\overline{u}^j:_\l:\overline{[u_r,u_s]}\overline{u}^r\overline{u}^s:].
\end{align*}

It follows from   formulas \eqref{f1}, \eqref{f2}, \eqref{f3} that 
\begin{equation}\label{f8}
[:u_i\overline u^i:_\la:u_j\overline u^j:]=:[u_i,u_j]\overline  
u^i\overline u^j:,
\end{equation}
\begin{equation}\label{f9}
[:u_i\overline u^i:_\la:\overline{[u_r,u_s]}\overline{u}^r\overline{u}^s:]=[:\overline{[u_r,u_s]}\overline{u}^r\overline{u}^s:_\l:u_i\overline u^i:]=(u^i,[u_r,u_s]):u_i\overline{u}^r\overline
u^s:,
\end{equation} 
and
 \begin{equation}\label{f11}
 \begin{aligned}
&[:\overline{[u_i,u_j]}\overline{u}^i\overline{u}^j:_\l :\overline{[u_r,u_s]}\overline{u}^r\overline{u}^s:]\\
&=\l\left(([u_i,u_j],u^s)(u^j,[u_r,u_s]):\ov u^i\ov u^r:-([u_i,u_j],u^r)(u^j,[u_r,u_s]):\ov u^i\ov u^s:\right)\\
&+\l\left(-([u_i,u_j],u^s)(u^i,[u_r,u_s]):\ov u^j\ov u^r:+([u_i,u_j],u^r)(u^i,[u_r,u_s]):\ov u^j\ov u^s:\right)\\
&-\sum_t(([u_i,u_j],u^r)([u_r,u_s],u^t):\ov u^i\ov u^j\ov{u}_t\ov u^s:+([u_i,u_j],u^s)([u_r,u_s],u^t):\ov u^i\ov u^j\ov u_t\ov u^r:)\\
&-\sum_t(([u_r,u_s],u^i)([u_i,u_j],u^t): \ov u_t\ov u^j\ov u^r\ov u^s:+([u_r,u_s],u^j)([u_i,u_j],u^t): \ov u_t\ov u^i\ov u^r\ov u^s:)\\
&+([u_i,u_j],u^s)(u^j,[u_r,u_s])\partial\ov u^i\ov u^r:-([u_i,u_j],u^r)(u^j,[u_r,u_s]):\partial\ov u^i\ov u^s:\\
&-([u_i,u_j],u^s)(u^i,[u_r,u_s]):\partial\ov u^j\ov u^r:+([u_i,u_j],u^r)(u^i,[u_r,u_s]):\partial\ov u^j\ov u^s:.
\end{aligned}\end{equation} 
and apply \eqref{f8} and \eqref{f9} to obtain
$$
\begin{aligned}
&[G^+_\l G^+]\\&=4\sum_{i,j}:[u_i,u_j]\overline{u}^i\overline{u}^j:-4\sum_{i,j,r}(u^i,[u_j,u_r]):u_i\overline{u}^j\overline
u^r:+\sum_{i,j,r,s}[:\overline{[u_i,u_j]}\overline{u}^i\overline{u}^j:_\l:\overline{[u_r,u_s]}\overline{u}^r\overline{u}^s:]\\
&=\sum_{i,j,r,s}[:\overline{[u_i,u_j]}\overline{u}^i\overline{u}^j:_\l:\overline{[u_r,u_s]}\overline{u}^r\overline{u}^s:].
\end{aligned}
$$

We now apply \eqref{f11} to get
\begin{equation}\label{master}
\begin{aligned}
&[G^+{}_\la G^+]=4\l\sum_{i,j,r}([u^j,u_r],[u_i,u_j]):\ov u^i\ov u^r:\\
&-2\sum_{i,j,r,s}\left(([u_r,u_s],[u_j,u^i]):\ov u^r\ov u^j\ov{u}_i\ov u^s:+([u_r,u_s],[u_j,u^i]): \ov u_i\ov u^j\ov u^r\ov u^s:\right)\\
&+4\sum_{i,j,r}([u^j,u_r],[u_i,u_j]):\partial\ov u^i\ov u^r:.
\end{aligned}
\end{equation}

Since $\mathfrak u^+_J$ is isotropic and stable under the action of $C_\g-C_\aa$, we see that
$$
\begin{aligned}&\sum_j([u^j,u_r],[u_i,u_j])=\half(u_r,\sum_j[u^j,[u_j,u_i]])+\half(u_r,\sum_j[u_j,[u^j,u_i]])-\half(u_r,[\sum_j[u_j,u^j],u_i])\\
&=\half(u_r,(C_\g-C_\aa)u_i)-\half(u_r,[\sum_j[u_j,u^j],u_i])=-\half(u_r,[\sum_j[u_j,u^j],u_i]).
\end{aligned}
$$
By the explicit expression  given in \eqref{rotilde}, we see that $\mathfrak u^+_J$ is  stable also under the action of $\sum_j[u_j,u^j]$,
hence 
$$
\sum_j([u^j,u_r],[u_i,u_j])=0
$$
and 
$$
\begin{aligned}
&[G^+{}_\la G^+]=
&-2\sum_{i,j,r,s}\left(([u_r,u_s],[u_j,u^i]):\ov u^r\ov u^j\ov{u}_i\ov u^s:+([u_r,u_s],[u_j,u^i]): \ov u_i\ov u^j\ov u^r\ov u^s:\right).
\end{aligned}
$$

Recall that, if $a_i\in\q$, $i=1,\ldots,k$, and
$\s$ is a permutation of $1,\ldots,k$, then 
\begin{equation}\label{f30}
:\ov a_{\s(1)}\cdots \ov a_{\s(k)}:=(-1)^\s:\ov a_1\cdots \ov a_k:.
\end{equation}
This implies that
\begin{align}
[{G^+}_\l G^+]\label{f22}=-4\sum_{i,j,r,s}([u_r,u_s],[u^i,u_j]):\ov u_i\ov u^r\ov u^j\ov u^s:\,.
\end{align}
We claim that $\sum_{i,j,r,s}([u_r,u_s],[u^i,u_j]):\ov u_i\ov u^r\ov u^j\ov u^s:=0$.
By \eqref{f30}, $\ov u_i\ov u^r\ov u^j\ov u^s:\ne0$ only if $r,j,s$ are distinct. It follows from \eqref{f30} that
\begin{align*}
&\sum_{i,j,r,s}([u_r,u_s],[u^i,u_j]):\ov u_i\ov u^r\ov u^j\ov u^s:=\sum_i :\ov u_i(\sum_{j,r,s}([u_r,u_s],[u^i,u_j]):\ov u^r\ov u^j\ov u^s:):\\
&=\sum_i :\ov u_i(-\sum_{j,r,s}([[u_r,u_s],u_j],u^i):\ov u^r\ov u^j\ov u^s:):\\
&=-\sum_{j,r,s}:\ov{[[u_r,u_s],u_j]}\ov u^r\ov u^j\ov u^s:\\
&=-\sum_{r<j<s}(:\ov{[[u_r,u_s],u_j]}\ov u^r\ov u^j\ov u^s:+:\ov{[[u_j,u_s],u_r]}\ov u^j\ov u^r\ov u^s:+:\ov{[[u_r,u_j],u_s]}\ov u^r\ov u^s\ov u^j:)\\
&-\sum_{r<j<s}(:\ov{[[u_s,u_r],u_j]}\ov u^s\ov u^j\ov u^r:+:\ov{[[u_s,u_j],u_r]}\ov u^s\ov u^r\ov u^j:+:\ov{[[u_j,u_r],u_s]}\ov u^j\ov u^s\ov u^r:)\\
&=-\sum_{r<j<s}:(\ov{[[u_r,u_s],u_j]}+\ov{[[u_s,u_j],u_r]}+:\ov{[[u_j,u_r],u_s]})\ov u^r\ov u^j\ov u^s
\\
&-\sum_{r<j<s}:(\ov{[[u_j,u_s],u_r]}+\ov{[[u_r,u_j],u_s]}+\ov{[[u_s,u_r],u_j]})\ov u^j\ov u^r\ov u^s:=0
\end{align*}
by Jacobi identity. We have thus proven that  $[G^+{}\!\!_\l G^+]=0$.
The same argument proves that $[G^-{}\!\!_\l G^-]=0$, thus the proof of Proposition \ref{ParabisN=2} is complete in the special case $\aa=\mathfrak m$.

\begin{remark}\label{GGJ}It follows from \eqref{n21} that 
$$
[\widetilde G^-{}_\l \widetilde G^+]=[\widetilde G^+{}_{-\l-\partial} \widetilde G^-]=L-\left(\half\partial +\l \right)\widetilde J +\l^2c \vac,
$$
so
$$
[\widetilde G^+{}_{\l} \widetilde G^-]-[\widetilde G^-{}_\l \widetilde G^+]=(\partial +2\l)\widetilde J.
$$

On the other hand
$$
[\widetilde G^+{}_\l \widetilde G^-]=[(\widetilde G+\sqrt{-1}\widetilde G_J)_\l (\widetilde G-\sqrt{-1}\widetilde G_J)]= L+\frac{\l^2}{6}C \vac+\sqrt{-1}([(G_J)_\l G]-[G_\l G_J])
$$
and 
$$
[\widetilde G^-{}_\l \widetilde G^+]=[(\widetilde G-\sqrt{-1}\widetilde G_J)_\l (\widetilde G+\sqrt{-1}\widetilde G_J)]= L+\frac{\l^2}{6}C \vac-\sqrt{-1}([(G_J)_\l G]-[G_\l G_J]),
$$
thus 
$$
(\partial +2\l)\widetilde J=2\sqrt{-1}([(G_J)_\l G]-[G_\l G_J]).
$$
Since
$$
0=[\widetilde G^+{}_{\l} \widetilde G^+]=[(\widetilde G+\sqrt{-1}\widetilde G_J)_\l (\widetilde G+\sqrt{-1}\widetilde G_J)]=\sqrt{-1}([G_\l G_J]+[(G_J)_\l G])$$
we obtain $
[\widetilde G_\l \widetilde G_J]=-[(\widetilde G_J)_\l \widetilde G]$
and
\begin{equation}\label{GGJnew}
[\widetilde G_\l \widetilde G_J]=\tfrac{\sqrt{-1}}{4}(\partial+2\l)\widetilde J.
\end{equation}
On the other hand, it is clear that \eqref{GGJnew}, combined with \eqref{NSG} and \eqref{NSGJ}, implies \eqref{n22}, \eqref{n23}, and \eqref{n21}, so these latter relations are equivalent to \eqref{GGJnew}.
\end{remark}

We now complete the proof of Proposition \ref{ParabisN=2} with
$$
\mathfrak m=\aa\oplus \mathfrak t
$$
with $\mathfrak t$ a subspace of the center of $\mathfrak m$. In light of Remark \ref{GGJ}, it is enough to check that  \eqref{GGJnew} holds also in this case.

If $x\in\g$, set 
$$
\tilde x =x+\theta_\g(ad(x))\in V^{\mathbf k}(\g)\otimes F(\ov\g).
$$
Let  $\mathfrak u$ be  a subalgebra of $\g$ such that $\mathfrak u=\oplus_i\mathfrak u\cap\g_i$ and $(\cdot,\cdot)_{|\mathfrak u\times \mathfrak u}$ is nondegenerate. Let $\{u_{ij}\}$ be a basis of $\mathfrak u\cap\g_i$ and let $\{u^{ij}\}$ be itsd dual basis. We set
$$
G_{\mathfrak u}=\sum_{i,j} :\tilde u_{ij}\ov u^{ij}:+\frac{1}{3}\sum_{i,j,t}:\ov{[u_{ij} ,u_{it}]}\ov u^{ij}\ov u^{it}:
$$
and
$$
\widetilde G_{\mathfrak u}=\sum_{i} \frac{1}{2\sqrt{k_i+h_i}}\left(\sum_j:\tilde u_{ij}\ov u^{ij}:+\frac{1}{3}\sum_{j,t}:\ov{[u_{ij} ,u_{it}]}\ov u^{ij}\ov u^{it}:\right).
$$
 Recall that (see \cite[(4.3),(4.4)]{KMPD})
\begin{equation}\label{GU}
 [\tilde u_\l G_{\mathfrak u}]=\l (k_i+h_i)\ov u,\ u\in\mathfrak u\cap \g_i, \quad [\ov u_\l  G_{\mathfrak u}]=\tilde u.
\end{equation}
 Recall also that $G=G_\g-G_\aa$ (see \cite[Lemma 4.5]{KMPD}). In particular we have 
 $$
 G=G_\g-G_{\mathfrak m}+G_{\mathfrak m}-G_\aa.
 $$
 hence 
  $$
 \widetilde G=\widetilde G_\g-\widetilde G_{\mathfrak m}+\widetilde G_{\mathfrak m}-\widetilde G_\aa.
 $$
Set $\mathfrak f=\mathfrak n^+_J\oplus \mathfrak n^-_J$. By \eqref{GU}, $[(\widetilde G_\g-\widetilde G_{\mathfrak m})_\l (\widetilde G_{\mathfrak m}-\widetilde G_{\mathfrak a})]=0$ so
 $$
 \begin{aligned}
 [(\widetilde G_\l \widetilde G_J]=[(\widetilde G_\g-\widetilde G_{\mathfrak m})_\l J(\widetilde G_{\mathfrak g}-\widetilde G_{\mathfrak m})]+[(\widetilde G_{\mathfrak m}-\widetilde G_{\mathfrak a})_\l J(\widetilde G_{\mathfrak m}-\widetilde G_{\mathfrak a})]\\=\tfrac{\sqrt{-1}}{4}(\partial+2\l)\widetilde J_{\mathfrak m}+[(\widetilde G_{\mathfrak m}-\widetilde G_{\mathfrak a})_\l J(\widetilde G_{\mathfrak m}-\widetilde G_{\mathfrak a})]
 \end{aligned}
 $$
 where $\widetilde J_{\mathfrak m}=\sqrt{-1}\Theta_{\mathfrak f}(J)+2\widetilde \rho+ 2\Theta_{\mathfrak f}(ad(\widetilde \rho))=\sqrt{-1}\Theta_{\mathfrak f}(J)+2\widetilde \rho+ 2\Theta_{\q}(ad(\widetilde \rho))$.
 \begin{lemma}Assume \begin{equation}\label{m=ait}
\mathfrak m=\aa\oplus \mathfrak t,\ [\mathfrak m,\mathfrak m]\subseteq \aa
\end{equation}
with $\mathfrak t$ a $J$-stable subspace of the center of $\mathfrak m$. Choose a basis $\{t_i\}$ of $\mathfrak t^+$ and let $\{t^i\}$ be the basis of $\mathfrak t^-$ dual to $\{t_i\}$. 
Then
 $$
 [(\widetilde G_{\mathfrak m}-\widetilde G_{\mathfrak a})_\l J(\widetilde G_{\mathfrak m}-\widetilde G_{\mathfrak a})]=\frac{\sqrt{-1}}{4}(\partial+2\l)\widetilde J_{\mathfrak t}
$$
with $\widetilde J_{\mathfrak t}=-\sum_i :\overline {t_{i}}\overline t^{i}: $.

In particular, Proposition \ref{ParabisN=2} holds with $\aa$ as in \eqref{m=ait}.
 \end{lemma}
\begin{proof}Set $\{t_{ij}\}$ to be an orthonormal basis of $\mathfrak t\cap\g_i$. 
Recall that we set for brevity $\mathfrak f=\mathfrak n^+_J\oplus \mathfrak n^-_J$. Using \cite[Lemma 4.5]{KMPD} applied to $\mathfrak m$ we can write
$$
\widetilde G_{\mathfrak m}-\widetilde G_{\mathfrak a}=\sum_i\frac{1}{2\sqrt{k_i+h_i}}\sum_j:(\tilde t_{ij}-\theta_{\mathfrak m}(ad(t_{ij})))\ov t_{ij}:=\sum_i\frac{1}{2\sqrt{k_i+h_i}}\sum_j:(t_{ij}+\theta_{\mathfrak f}(ad(t_{ij})))\ov t_{ij}:.
$$
 It follows that
\begin{align*}
 &[(\widetilde G_{\mathfrak m}-\widetilde G_{\mathfrak a})_\l J(\widetilde G_{\mathfrak m}-\widetilde G_{\mathfrak a})]=\\&\sum_{i,r}
\frac{1}{4\sqrt{k_i+h_i}\sqrt{k_r+h_r}}\sum_{j,s}[:(t_{ij}+\theta_{\mathfrak f}(ad(t_{ij})))\ov t_{ij}:_\l:(t_{rs}+J(\theta_{\mathfrak f}(ad(t_{rs})))\ov{J( t_{rs})}:].
\end{align*}
 Since
 $$
 \begin{aligned}
 J(\theta_{\mathfrak f}(ad(t_{rs})))&=\half\sum_{h}J(:\ov{[t_{rs},u_{rh}]}\ov u^{rh}:)+\half\sum_{h}J(:\ov{[t_{rs},u^{rh}]}\ov u_{rh}:)\\
 &=\half\sum_{h}\a_{rh}(t_{rs}):\ov{J(u_{rh})}\ov {J(u^{rh})}:-\half\sum_{h}\a_{rh}(t_{rs}):\ov{J(u^{rh})}\ov{J( u_{rh})}:\\
 &=\half\sum_{h}\a_{rh}(t_{rs}):\ov u_{rh}\ov u^{rh}:-\half\sum_{h}\a_{rh}(t_{rs}):\ov u^{rh}\ov u_{rh}:=\theta_{\mathfrak f}(ad(t_{rs})),
 \end{aligned}
 $$
 we have
 $$
 \begin{aligned}
 [:(t_{ij}+\theta_{\mathfrak f}(ad(t_{ij})))\ov t_{ij}:&_\l:(t_{rs}+J(\theta_{\mathfrak f}(ad(t_{rs})))\ov{J( t_{rs})}:]\\
& =[:(t_{ij}+\theta_{\mathfrak f}(ad(t_{ij})))\ov t_{ij}:_\l:(t_{rs}+\theta_{\mathfrak f}(ad(t_{rs}))\ov{J( t_{rs})}:].
 \end{aligned}
 $$
Recall  that, similarly to the mapping $a\mapsto a^\aa$ from Section \ref{setup},  the mapping  $m\mapsto m+\theta_{\mathfrak f}(ad(m))$ extends to a vertex algebra map $V^{\mathbf k+\mathbf h^\vee-\mathbf h^\vee_{\mathfrak m}}(\mathfrak m)\to V^{\mathbf k}(\g)\otimes F(\ov{\mathfrak f})\subset V^{\mathbf k}(\g)\otimes F(\ov{\mathfrak q})$, where $\mathbf h^\vee_{\mathfrak m}(m,m')=(C_{\mathfrak m}m,m')$. Since $\mathfrak t$ is central in $\mathfrak m$ we see that $\mathbf h^\vee_{\mathfrak m}(t,t')=0$ if $t,t'\in\mathfrak t$. Applying \eqref{f1} in $V^{\mathbf k+\mathbf h^\vee-\mathbf h^\vee_{\mathfrak m}}(\mathfrak m)\otimes F(\ov{\mathfrak t})$, we find
 $$
 \begin{aligned}
& [:(t_{ij}+\theta_{\mathfrak f}(ad(t_{ij})))\ov t_{ij}:_\l:(t_{rs}+\theta_{\mathfrak f}(ad(t_{rs}))\ov{J( t_{rs})}:]=\\
 &={ (t_{ij},J(t_{rs})):(t_{ij}+\theta_{\mathfrak f}(ad(t_{ij})))(t_{rs}+\theta_{\mathfrak f}(ad(t_{rs})):}+\\
 &(k_i+h^\vee_i)\d_{i,r}\d_{j,s}(:\partial \overline t_{ij}\,\overline {J(t_{rs})}: +\la:\overline t_{ij}\,\overline {J(t_{rs})}: +\half \l^2( t_{ij},J(t_{rs})):.
 \end{aligned}
 $$
 so
  $$
  \begin{aligned}
 &[(\widetilde G_{\mathfrak m}-\widetilde G_{\mathfrak a})_\l J(\widetilde G_{\mathfrak m}-\widetilde G_{\mathfrak a})]\\
 &=\sum_{i,r}
\frac{1}{4\sqrt{k_i+h_i}\sqrt{k_r+h_r}}\sum_{j,s}{ (t_{ij},J(t_{rs})):(t_{ij}+\theta_{\mathfrak f}(ad(t_{ij})))(t_{rs}+\theta_{\mathfrak f}(ad(t_{rs})):}\\
&+\sum_{i,r}
\frac{1}{4\sqrt{k_i+h_i}\sqrt{k_r+h_r}}\sum_{j,s}(k_i+h^\vee_i)\d_{i,r}\d_{j,s}(:\partial \overline t_{ij}\,\overline {J(t_{rs})}: +\la:\overline t_{ij}\,\overline {J(t_{rs})}: +\half \l^2( t_{ij},J(t_{rs})).
\end{aligned}
 $$
 Now observe that $(t,J(t))=-(J(t),t)=-(t,J(t))$ hence $(t,J(t))=0$, so
 $$
  \begin{aligned}
&  \sum_{i,r}
\frac{1}{4\sqrt{k_i+h_i}\sqrt{k_r+h_r}}\sum_{j,s}{ (t_{ij},J(t_{rs})):(t_{ij}+\theta_{\mathfrak f}(ad(t_{ij})))(t_{rs}+\theta_{\mathfrak f}(ad(t_{rs})):}\\
 &=\sum_{(i,j)\ne(r,s)}
\frac{1}{4\sqrt{k_i+h_i}\sqrt{k_r+h_r}}{ (t_{ij},J(t_{rs})):(t_{ij}+\theta_{\mathfrak f}(ad(t_{ij})))(t_{rs}+\theta_{\mathfrak f}(ad(t_{rs})):}
\end{aligned}
$$
Moreover $(t,J(t'))=-(t',J(t))$, so, ordering the pairs $(i,j)$ lexicographically, we have
$$
  \begin{aligned}
& \sum_{(i,j)\ne(r,s)}
\frac{1}{4\sqrt{k_i+h_i}\sqrt{k_r+h_r}}{ (t_{ij},J(t_{rs})):(t_{ij}+\theta_{\mathfrak f}(ad(t_{ij})))(t_{rs}+\theta_{\mathfrak f}(ad(t_{rs})):}\\
&=\sum_{(i,j)<(r,s)}
\frac{1}{4\sqrt{k_i+h_i}\sqrt{k_r+h_r}}{ (t_{ij},J(t_{rs})):(t_{ij}+\theta_{\mathfrak f}(ad(t_{ij})))(t_{rs}+\theta_{\mathfrak f}(ad(t_{rs})):}\\
&+\sum_{(i,j)>(r,s)}
\frac{1}{4\sqrt{k_i+h_i}\sqrt{k_r+h_r}}{ (t_{ij},J(t_{rs})):(t_{ij}+\theta_{\mathfrak f}(ad(t_{ij})))(t_{rs}+\theta_{\mathfrak f}(ad(t_{rs})):}\\
&=\sum_{(i,j)<(r,s)}
\frac{1}{4\sqrt{k_i+h_i}\sqrt{k_r+h_r}}{ (t_{ij},J(t_{rs})):(t_{ij}+\theta_{\mathfrak f}(ad(t_{ij})))(t_{rs}+\theta_{\mathfrak f}(ad(t_{rs})):}\\
&-\sum_{(i,j)<(r,s)}
\frac{1}{4\sqrt{k_i+h_i}\sqrt{k_r+h_r}}{ (t_{ij},J(t_{rs})):(t_{rs}+\theta_{\mathfrak f}(ad(t_{rs})))(t_{ij}+\theta_{\mathfrak f}(ad(t_{ij})):}
\end{aligned}
$$
Since

\begin{align*}
&\int_{-\partial}^0[(t_{rs}+\theta_{\mathfrak f}(ad(t_{rs})))_\l(t_{ij}+\theta_{\mathfrak f}(ad(t_{ij}))]d\l\\&=\int_{-\partial}^0([t_{rs},t_{ij}]+\theta_{\mathfrak f}(ad(([t_{rs},t_{ij}])+\l(\mathbf k+\mathbf h^\vee-\mathbf h^\vee_{\mathfrak m})(t_{ij},t_{rs})\vac)d\l\\
&=\int_{-\partial}^0(k_i+h^\vee_i)\d_{i,r}\d_{j,s}\l\vac d\l=0,
\end{align*}

by \cite[(1.39)]{DK}, we find 
$$
\begin{aligned}:
&(t_{rs}+\theta_{\mathfrak f}(ad(t_{rs})))(t_{ij}+\theta_{\mathfrak f}(ad(t_{ij})):=:(t_{ij}+\theta_{\mathfrak f}(ad(t_{ij}))(t_{rs}+\theta_{\mathfrak f}(ad(t_{rs}))):
\end{aligned}
$$
hence 
$$
 \sum_{(i,j)\ne(r,s)}
\frac{1}{4\sqrt{k_i+h_i}\sqrt{k_r+h_r}}{ (t_{ij},J(t_{rs})):(t_{ij}+\theta_{\mathfrak f}(ad(t_{ij})))(t_{rs}+\theta_{\mathfrak f}(ad(t_{rs})):}=0.
$$
The outcome is that
  $$
  \begin{aligned}
 &[(\widetilde G_{\mathfrak m}-\widetilde G_{\mathfrak a})_\l J(\widetilde G_{\mathfrak m}-\widetilde G_{\mathfrak a})]\\
&=\sum_{i,j}
\frac{1}{4}(:\partial \overline t_{ij}\,\overline {J(t_{ij})}: +\la:\overline t_{ij}\,\overline {J(t_{ij})}: ).
\end{aligned}
 $$
Choose bases $\{t_i\}$, $\{t^i\}$ of $\mathfrak t^\pm$ as in the statement. Then we can write
 $$
 \sum_{i,j}:\overline t_{ij}\,\overline {J(t_{ij})}:=\sum_i(:\overline t_{i}\,\overline {J(t^{i})}:+:\overline t^{i}\,\overline {J(t_{i})}:)=-2\sqrt{-1}\sum_i:\overline t_{i}\,\overline {t^{i}}:.
 $$
 On the other hand
 $$
 \begin{aligned}
 \sum_{i,j}
&:\partial \overline t_{ij}\,\overline {J(t_{ij})}: =  \half\sum_{i,j}
:\partial \overline t_{ij}\,\overline {J(t_{ij})}: +\half \sum_{i,j}
:\partial \overline t_{ij}\,\overline {J(t_{ij})}:\\ 
&=  -\sqrt{-1}\sum_{i}(
:\partial \overline t_{i}\,\overline {t^{i}}: +:\overline {t_{i}}\partial \overline t^{i}: )= -\sqrt{-1}\sum_{i}\partial(:\overline {t_{i}}\overline t^{i}: ).
\end{aligned}
$$
The conclusion is that
$$
[(\widetilde G_{\mathfrak m}-\widetilde G_{\mathfrak a})_\l J(\widetilde G_{\mathfrak m}-\widetilde G_{\mathfrak a})]=\frac{\sqrt{-1}}{4}(\partial+2\l)\widetilde J_{\mathfrak t}
$$
where $\widetilde J_{\mathfrak t}=-\sum_i :\overline {t_{i}}\overline t^{i}: $

Proposition \ref{ParabisN=2} is now proved with $\aa$ as in \eqref{m=ait}, since $\widetilde J=\widetilde J_{\mathfrak m}+\widetilde J_{\mathfrak t}$.
\end{proof}

\section{The Joyce construction}\label{Joyce}
With the notation established in  Section \ref{setup},
we say that 
 $\mathbb M=\mathbb G/\mathbb A$ admits an invariant hypercomplex structure if there are two invariant complex structures on $\mathbb M$  such that, if  $J$ and $K$ are the corresponding integrable almost complex structures on $\mathfrak q$,  then $JK=-KJ$. 
 
Following \cite{BGP}, we now recall the Joyce construction. In \cite{Joyce} Joyce explains how to construct invariant hypercomplex
structures on suitable compact homogeneous spaces. His construction
can be translated to our setting as follows. Assume that $\g$ is reductive. Fix a
Cartan subalgebra $\h$ of $\g$ and denote by $\Delta$ the set of
corresponding roots. Using e.g. the Kostant's cascade \cite{Kostant}, we can find a sequence $\theta_1,\ldots,\theta_h$
of roots, where $\theta_1$ is the highest root of a simple ideal of $\g$. If $\mathfrak{s}_i \cong sl_2(\C)$ is the
subalgebra generated by the root spaces $\g_{\theta_i},
\g_{-\theta_i}$ and
\[
\mathfrak{f}_i:=\bigoplus_{(\alpha,\theta_i)\neq0,\a\ne\theta_i}\!\g_{\alpha}\;\,\,, \quad
\mathfrak{b}_i:=  \bigcap_{j=1}^i \,C_{\g}(\mathfrak{s}_j)\;\,,
\]
where $C_{\g}(\mathfrak{s}_j)$ denotes the centralizer of $\mathfrak{s}_j$ in $\g$, then one has the decomposition
\begin{equation}
\label{Jd}
\g=\mathfrak{b}_h \oplus \bigoplus_{i=1}^h \mathfrak{s}_i \oplus \bigoplus_{i=1}^h \mathfrak{f}_i\,.
\end{equation}

The Lie algebra  $\mathfrak{a}\subset\mathfrak{b}_h$ is chosen as
follows: the semisimple part of $\mathfrak{a}$ coincides with the semisimple
part of $\mathfrak{b}_h$ and the center $C_\mathfrak{a}$ of $\mathfrak{a}$ is a subspace of
the center $C'$ of $\mathfrak{b}_h$ such that $\dim C' -\dim C_\mathfrak{a}
-h \equiv 0\mod 4$.

 Introduce the Pauli matrices
\begin{equation}\label{sigma}\s_J=\begin{pmatrix}\sqrt{-1}&0\\0&-\sqrt{-1}\end{pmatrix},\ \s_K=\begin{pmatrix}0&1\\-1&0\end{pmatrix},\ 
\s_{JK}=\s_{J}\s_{K}=\begin{pmatrix}0&\sqrt{-1}\\\sqrt{-1}&0\end{pmatrix}.\end{equation}
Let $L_J$ and $L_K$ denote left multiplication on $gl(2)$  by the above matrices. The operators $L_J$, $L_K$, and $L_JL_K$  define a quaternionic  structure on  $gl(2)$. 
 Define  the following bilinear form on $gl(2)$:
\begin{equation}\label{invgl2}(A,A')= Tr(AA')-Tr(A)Tr(A').
\end{equation}
The bilinear form \eqref{invgl2} is the unique  invariant bilinear form  which restricts to the trace from on $sl(2)$ and is invariant w.r.t. left multiplication by $\s_J, \s_K$. We select orthogonal vectors $d_1,\ldots,d_h$ in $C'\cap \mathfrak q$ and choose isomorphisms $\phi_i:sl(2)\to \mathfrak s_i$ in such a way that $\phi_i$ is an isometry between the trace form on $sl(2)$ and $(\cdot,\cdot)_{|\mathfrak s_i\times  \mathfrak s_i}$. We also assume that $\phi_i(\s_J)\in\h$ and that $x_{\theta_i}:=\phi_i(\begin{pmatrix}0&1\\0&0\end{pmatrix})$ is a root vector for $\theta_i$. We can choose $d_i$ so that, extending $\phi_i$ to $gl(2)$ by setting $\phi_i(\begin{pmatrix}1&0\\0&1\end{pmatrix})=d_i$, we obtain an isomorphism $\phi_i:gl(2)\to \mathfrak q_i:= \C d_i\oplus \mathfrak s_i
 $ which  is an isometry between the form \eqref{invgl2} and $(\cdot,\cdot)_{|\mathfrak q_i\times  \mathfrak q_i}$.
 
 The orthocomplement  of $\aa$ in $\g$ is
\begin{equation}\label{q}
 \mathfrak q=C_{\mathfrak q}\oplus (\oplus_i \mathfrak q_i)\oplus (\oplus_i \mathfrak f_i)
\end{equation}
 where $C_{\mathfrak{q}}$ denotes the orthocomplement of $C_\mathfrak{a}\oplus
\sum_i \C d_i$ in $C'$.

Define  $J_i=\phi_i(\s_J)$ (resp.  $K_i=\phi_i(\s_K)$). The $\aa$-invariant hypercomplex structure on $\mathfrak q$ is obtained by setting
$$
J_{|\mathfrak{f}_i}=\ad(J_i),\  J_{|\mathfrak{q}_i}=\phi_iL_J\phi_i^{-1},\ K_{|\mathfrak{f}_i}=\ad(K_i),\  K_{|\mathfrak{q}_i}=\phi_iL_K\phi_i^{-1},
$$
 and $C_{\mathfrak{q}}$ is endowed with an arbitrary linear
hypercomplex structure that leaves $(\cdot,\cdot)$ invariant.

We are interested only in two special cases of the above construction which we now describe.
\begin{example} \label{Ex1} Consider $\mathbb G=U(2)$. 
 The Joyce construction reduces in this case to taking 
$\g=gl(2)$, $\aa=\{0\}$ and $h=1$, so that $\mathfrak q=\g$. Let $\h$ be the diagonal subalgebra of $\g$. We choose $\theta_1=\theta$ where $\theta$ denotes the $\h$-root of $\begin{pmatrix}0&1\\0&0\end{pmatrix}$. In this special case we have
$$\mathfrak s_1=sl(2),\ \mathfrak f_1=\{0\},\ \mathfrak b_1=C'=\C\begin{pmatrix}1&0\\0&1\end{pmatrix}.
$$ 
We choose $d_1=\begin{pmatrix}1&0\\0&1\end{pmatrix}$ and the invariant form given by \eqref{invgl2}, so that we can choose $\phi_1$ to be the identity. Hence
$J_1=\s_J$ and $K_1=\s_K$.
The almost complex structures $J$, $K$ and $JK$ on $\mathfrak q$  are given by left multiplication by the Pauli matrices in  \eqref{sigma}. 
\end{example}
\begin{example}\label{Ex2}{ Consider the case  $\mathbb G=SU(n)$, $n>2$, $h=1$.}
In this case $\g=sl(n)$. Let $\h$ be the diagonal subalgebra, $\mathfrak b$ the subalgebra of upper triangular matrices. If  $\D$ is the  set of $\h$-roots of $\g$, we let $\Dp$ be the positive set of roots corresponding to $\mathfrak b$. As an invariant form on $\g$, we choose the trace form.  We choose the root vectors $x_{\a}$, $\a\in\D$, to be the elementary matrices $E_{i,j}$, as usual.
We choose  $\theta_1=\theta$ where $\theta$ is  the highest root, so that, in Joyce construction, we have 
$
\mathfrak s_1=span(x_\theta,h_\theta, x_{-\theta})
$
and  $\mathfrak b_1$ is the centralizer in $\g$ of $\mathfrak s_1$. We  choose  $\aa=[\mathfrak b_1,\mathfrak b_1]$, so that $C_\aa=\{0\}$. Indeed the center $C'$ of $\mathfrak b_1$ is one-dimensional, hence $\dim C'-\dim C_\aa-1=0$.
The orthocomplement $\mathfrak q$ of $\aa$ in $\g$ is the set of matrices of the form 
$$
\begin{pmatrix}a_{11}&a_{12}&a_{13}&\cdots&a_{1n}\\
a_{21}&-\tfrac{a_{11}+a_{nn}}{n-2}&0&\cdots&a_{2n}\\
a_{31}&0&-\tfrac{a_{11}+a_{nn}}{n-2}&\cdots&a_{3n}\\
\vdots&\vdots&&\ddots&\vdots\\
a_{n1}&a_{n2}&\cdots&&a_{nn}
\end{pmatrix}.
$$
Next we fix the isomorphism $\phi_1:gl(2)\to sl(n)$:
$$
\phi_1(\begin{pmatrix}0&1\\0&0\end{pmatrix})=x_\theta,\ \phi_1(\begin{pmatrix}1&0\\0&-1\end{pmatrix})=h_\theta,\ \phi_1(\begin{pmatrix}0&0\\1&0\end{pmatrix})=x_{-\theta},\  
\phi_1(\begin{pmatrix}1&0\\0&1\end{pmatrix})=d_1,
$$
with 
$$
d_1=\sqrt{\tfrac{2-n}{n}}\begin{pmatrix}1&0&0\\
0&-\tfrac{2}{n-2}I_{n-2}&0\\
0&0&1
\end{pmatrix}.
$$
Then $\mathfrak q=\mathfrak s_1\oplus \C d_1\oplus \mathfrak f_1$ with $\mathfrak s_1=\phi_1(sl(2))$, 
$\mathfrak f_1$ is the set of matrices of the form 
$$
\begin{pmatrix}0&a_{12}&a_{13}&\cdots&0\\
a_{12}&0&0&\cdots&a_{2n}\\
a_{13}&0&0&\cdots&a_{3n}\\
\vdots&\vdots&&\ddots&\vdots\\
0&a_{n2}&a_{n3}&\cdots&0
\end{pmatrix},
$$
and the almost complex structures $J$, $K$, and $JK$ on $\mathfrak q$  are given by
$$
J(v)=[\phi_1(\s_J),v],\  v\in\mathfrak f_1,\quad J(v)=\phi_1(\s_J\phi_1^{-1}(v)),\ v\in\mathfrak s_1\oplus \C d_1,
$$
$$
K(v)=[\phi_1(\s_K),v],\ v\in\mathfrak f_1,\quad K(v)=\phi_1(\s_K\phi_1^{-1}(v)),\ v\in\mathfrak s_1\oplus \C d_1,
$$
$$
JK(v)=[\phi_1(\s_{JK}),v],\ v\in\mathfrak f_1,\quad JK(v)=\phi_1(\s_{JK}\phi_1^{-1}(v)),\ v\in\mathfrak s_1\oplus \C d_1.
$$
\end{example}

In both the special cases considered above set \begin{equation}\label{s}\mathfrak s=\phi_1(gl(2))=\mathfrak s_1\oplus \C d_1.\end{equation} Note that $\mathfrak q=\mathfrak s\oplus  \mathfrak f_1$ is an orthogonal decomposition and that both $\mathfrak s$ and $\mathfrak f_1$ are stable under the action of $ad(\phi_1(\s_J))$ and $J$ and similarly for $K$ and $JK$.

\section{The coset realisation of the $big\ N=4$ vertex algebra}
The $big\ N=4$  vertex algebra $V^{N=4}_{\tilde c,a}$ is the universal enveloping  vertex algebra of the Lie  conformal superalgebra generated by a Virasoro vector $\tilde L$ of central charge $\tilde c$, four (odd) free fermions
$\sigma^{--}$, $\sigma^{++}$, $\sigma^{+-}$, $\sigma^{-+}$ with
non-zero $\lambda$-brackets $[{\sigma^{--}}{}_{\lambda}
\sigma^{++}]=-\tfrac{\widetilde c}{6}\vac$, $[{\sigma^{+-}}{}_{\lambda}
\sigma^{-+}]=-\tfrac{\widetilde c}{6}\vac$, one (even) free boson $\xi$ with
$\lambda$-bracket $[\xi_{\lambda} \xi]=-\lambda \tfrac{\widetilde c}{6}$, and  primary even fields $\tilde J^+$, $\tilde J^0$, $\tilde J^-$, $\tilde J^{'+}$,  $\tilde J^{'0}$, $\tilde J^{'-}$  of conformal weight $1$,  primary odd fields $\tilde{G}^{++} $, $\tilde{G}^{+-} $, $\tilde{G}^{-+} $, $\tilde{G}^{--} $ of conformal weight $\tfrac{3}{2}$.  The $\lambda$-brackets for the pairs
$(\tilde{J},
\tilde{J}')$ are zero, and the non-zero $\lambda$-brackets for the pairs
$(\tilde{J}, \tilde{G})$, $(\tilde{J}', \tilde{G})$,
$(\tilde{J}, \tilde{J}) $, and $(\tilde{J}',\tilde{J}')$ are as follows:
\begin{equation}\label{prime}
\begin{aligned}
  [\tilde{J}^0 {}_{\lambda} \tilde{G}^{++}]
     &= \tilde{G}^{++}-\lambda a \sigma^{++}\, , \,
         [\tilde{J}^0 {}_{\lambda} \tilde{G}^{--}]
         = -\tilde{G}^{--}+\lambda \sigma^{--} \, , \,
         [\tilde{J}^0 {}_{\lambda} \tilde{G}^{-+}]
         = -\tilde{G}^{-+}+\lambda \sigma^{-+}\, , \\
{} [\tilde{J}^0 {}_{\lambda} \tilde{G}^{+-}]
     &= \tilde{G}^{+-}-\lambda a\sigma^{+-} \, , \,
         [\tilde{J}^+ {_{\lambda}} \tilde{G}^{--}]
         = \tilde{G}^{+-} -\lambda a \sigma^{+-}\, , \,
         [\tilde{J}^+ {}_{\lambda} \tilde{G}^{-+}]
         =-\tilde{G}^{++}+\lambda a \sigma^{++}\, , \, \\
{} [\tilde{J}^- {}_{\lambda} \tilde{G}^{++}]
     &= -\tilde{G}^{-+}+\lambda \sigma^{-+} \, , \,
         [\tilde{J}^- {}_{\lambda} \tilde{G}^{+-}]
         =\tilde{G}^{--}-\lambda \sigma^{--} \, , \, \\
{} [\tilde{J}^{\prime 0} {}_{\lambda} \tilde{G}^{++}]
     &= \tilde{G}^{++}+\lambda \sigma^{++} \, , \,
         [\tilde{J}^{\prime 0}{}_{\lambda} \tilde{G}^{--}]
         = -\tilde{G}^{--} -\tfrac{\lambda}{a}\sigma^{--}\, , \,
         [\tilde{J}^{\prime 0}{}_{\lambda} \tilde{G}^{-+}]
         = \tilde{G}^{-+}+\tfrac{\lambda}{a}\sigma^{-+}\, , \\
{}  [\tilde{J}^{\prime 0}{}_{\lambda}\tilde{G}^{+-}]
      &= -\tilde{G}^{+-}-\lambda \sigma^{+-}\, , \,
          [\tilde{J}^{\prime +}{}_{\lambda} \tilde{G}^{--}]
          = \tilde{G}^{-+}+\tfrac{\lambda}{a}\sigma^{-+}\, , \,
          [\tilde{J}^{\prime +} {}_{\lambda} \tilde{G}^{+-}]
          = -\tilde{G}^{++}-\lambda \sigma^{++}\, , \\
{}  [\tilde{J}^{\prime -}{}_{\lambda} \tilde{G}^{++}]
       &= -\tilde{G}^{+-}-\lambda \sigma^{+-}\, , \,
          [\tilde{J}^{\prime -}{}_{\lambda} \tilde{G}^{-{+}}]
          = \tilde{G}^{--}+\tfrac{\lambda}{a}\sigma^{--}\, ,\\
\end{aligned}\end{equation}
\begin{equation}\label{gnaturalhat}\begin{aligned}
{}  [\tilde{J}^0{}_{\lambda} \tilde{J}^0]
       &= \lambda \frac{\tilde{c}(a+1)}{3}\vac \, , \,
          [{{\tilde{J}^0}}{}_{\lambda} \tilde{J}^{\pm}]
          = \pm 2 \tilde{J}^{\pm} \, ,\,
          [{\tilde{J}^+}{}_{\lambda} \tilde{J}^-]
          = \tilde{J}^0 +\lambda \frac{\tilde{c} (a+1)}{6}\vac \, ,\\
          [{\tilde{J}^{\prime 0}}{}_{\lambda} \tilde{J}^{\prime 0}]
          &=\lambda \frac{\tilde{c}(a+1)}{3a}\vac\, ,\,
{} [{\tilde{J}^{\prime 0}}{}_{\lambda} \tilde{J}^{\prime \pm}]
       = \pm 2 \tilde{J}^{\prime \pm} \, , \,
          [{\tilde{J}^{\prime +}}{}_{\lambda} \tilde{J}^{\prime -}]
          = \tilde{J}^{\prime 0}+\lambda \frac{\tilde{c}(a+1)}{6a}\vac\,.
\end{aligned}\end{equation}
The non-zero $\lambda$-brackets for the pairs $(\tilde{J}
,\sigma)$ are as follows:
\begin{equation}\label{seconde}
\begin{aligned}
   [{\tilde{J}^0}{}_{\lambda} \sigma^{\pm s}] &=
        \pm \sigma^{\pm s} \, , \,
        [{\tilde{J}^{\prime 0}}{}_{\lambda} \sigma^{s \pm}] =
        \pm \sigma^{s \pm} \, , \\[1ex]
{}  [{\tilde{J}^+}{}_{\lambda} \sigma^{-\mp}] &=
        \pm a \sigma^{+ \mp} \, , \,
        [{\tilde{J}^-}{}_{\lambda} \sigma^{+\pm}]
        =\mp \tfrac{1}{a} \sigma^{- \pm} \, , \\[1ex]
{}  [{\tilde{J}^{\prime +}{}_{\lambda}} \sigma^{\mp-}] &=
        \pm \sigma^{\mp +} \, , \,
        [{\tilde{J}^{\prime -}}{}_{\lambda} \sigma^{\pm +}]
        = \mp \sigma^{\pm -}\, .
\end{aligned}\end{equation}
The non-zero $\lambda$-brackets for the pairs $(\tilde{G} ,
\sigma)$ and $(\tilde{G} , \xi)$ are as follows:
\begin{equation}\label{terze}
\begin{aligned}
     [{\tilde{G}^{\,+\pm}}{}_{\lambda} \sigma^{- \mp}] &=
         \frac{a}{2(a+1)}(\tilde{J}^0 \mp \tilde{J}^{\prime 0})
         + \left( \tfrac{a}{2}\right)^{1/2}\xi \, , \,
\\[1ex]
{}   [{\tilde{G}^{\,+\pm}}{}_{\lambda} \sigma^{- \pm}] &=
         \frac{a}{a+1} \tilde{J}^{\prime \pm} \, , \,
         [{\tilde{G}^{- \pm}}{}_{\lambda}\sigma^{+\mp}]
         = {\mp} \frac{1}{2(a+1)}
         (\tilde{J}^{\prime 0}{  \pm} \tilde{J}^0)+
         \left( \frac{1}{2a}\right)^{1/2} \xi \, , \\[1ex]
{}   [{\tilde{G}^{- \pm}}{}_{\lambda} \sigma^{+  \pm}] &=
         -\frac{1}{a+1} \tilde{J}^{\prime  \pm} \, , \,
         [{\tilde{G}^{- \mp}}{}_{\lambda}\sigma^{-\pm}]
         = \pm \frac{a}{a+1}\tilde{J}^- \, , \,
         [{\tilde{G}^{+ \mp}}{}_{\lambda} \sigma^{+ \pm}]
         ={ \pm}\frac{1}{a+1} \tilde{J}^+ \, , \\[1ex]
{}   [{\tilde{G}^{+ \pm}}{}_{\lambda}\xi] &= (\partial + \lambda)
        \left( \frac{a}{2}\right)^{1/2} \sigma^{+ \pm} \, , \,
        [{\tilde{G}^{- \mp}}{}_{\lambda} \xi] =(\partial + \lambda)
        \left( \frac{1}{2a}\right)^{1/2}\sigma^{- \mp} \, .
\end{aligned}\end{equation}
Finally, the non-zero $\lambda$-brackets between the $\tilde{G}$'s
are as follows:
\begin{equation}\label{quarte}
\begin{aligned}
      [{\tilde{G}^{\pm +}}{}_{\lambda}
          \tilde{G}^{\mp -}] &=
         \tilde{L} + \frac{1}{{ 2}(a+1)} (\partial +2\lambda)
         (\pm \tilde{J}^0 +a\tilde{J}^{\prime 0}) +
         \frac{\lambda^2}{6}\tilde{c}\vac\, , \\[1ex]
{}    [{\tilde{G}^{\pm s}}{}_{\lambda} \tilde{G}^{\mp s}] &=
         \mp \frac{a}{a+1} (\partial +2\lambda)
         \tilde{J}^{\prime s} \, , \,
         [{\tilde{G}^{s\pm}}{}_{\lambda} \tilde{G}^{s \mp}]
         = \mp \frac{1}{a+1}(\partial +2\lambda)\tilde{J}^{s}\,.
\end{aligned}\end{equation}
 The aim of this Section is to construct coset models for $V^{N=4}_{\tilde c,a}$ associated to the hypercomplex structures on $U(2)$ and $SU(n)$ described in Examples \ref{Ex1} and \ref{Ex2} respectively.
 
Since $\g=\g_0\oplus \g_1$ with $\g_0$ abelian or zero and $\g_1$ simple,   a symmetric invariant bilinear form 
 $\k$ on $\g$ can be written as $\k=k_0(\cdot,\cdot)_{|\g_0\times \g_0}+k_1(\cdot,\cdot)_{|\g_1\times \g_1}$. Recall that $\ov{\mathfrak q}$ denotes the space $\mathfrak q$ considered as a totally odd space.

\begin{proposition}\label{Apamdef} Fix  a symmetric invariant form 
 $\k$ on $\g$ such that $\k+\hvee$ is nondegenerate.
Introduce the map $A^\pm:sl(2)\to V^\k(\g)\otimes F(\bar \q)$ by extending linearly the mappings
$$
\begin{aligned}
&A^+:\s_J\mapsto J_1+ \Theta_{\mathfrak s}(ad(J_1)- J),\ A^+:\s_K\mapsto K_1+ \Theta_{\mathfrak s}(ad(K_1)- K),\\
& A^+:\s_{JK}\mapsto(JK)_1+ \Theta_{\mathfrak s}(ad((JK)_1)- JK),\\
&A^-:\s_J\mapsto \Theta_\q(J),\ A^-:\s_K\mapsto \Theta_\q(K),\ A^-:\s_{JK}\mapsto \Theta_\q(JK).
\end{aligned}
$$
Then the map 
$$
a\otimes\vac+\vac\otimes b\mapsto A^+(a)+A^-(b),\quad a,b\in sl(2)
$$
extends to define a vertex algebra homomorphism from $V^{k_1+1}(sl(2))\otimes V^{n-1}(sl(2))\to V^\k(\g)\otimes F(\bar\q)$.
\end{proposition}
\begin{proof}
First of all, we write the map $A^+$ more explicitly:
indeed 
$$
A^+(\s_J)=J_1+\Theta_{\mathfrak s}(ad(J_1)-J)
$$ 
and, if $a\in gl(2)$, then
 $$
 (ad(J_1)-J)(\phi_1(a))=[\phi_1(\s_J),\phi_1(a)]-\phi_1(\s_Ja)=-\phi_1(a\s_J)
 $$
 and similarly for $A^+(K)$ and $A^+(JK)$.  Define $R_x\in so(\mathfrak s)$ by 
 \begin{equation}\label{Re}R_x(\phi_1(a))=-\phi_1(ax).
 \end{equation}
So, if $x\in sl(2)$,
 $$
 A^+(x)=\phi_1(x)+\Theta_{\mathfrak s}(R_x).
 $$
 Since
 $J_{|\mathfrak s}(\phi_1(a))=\phi_1(\s_Ja)$ and similarly for $K$ and $JK$, we see that
\begin{equation}\label{A-}
 A^-(x)=\Theta_{\mathfrak f_1}(ad(\phi_1(x))+\Theta_{\mathfrak s}(L_x),
\end{equation}
 where $L_x\in so(\mathfrak s)$ is defined by $L_x(\phi_1(a))=\phi_1(xa)$.
 
Since $\mathfrak f_1$ and $\mathfrak s$ are orthogonal, we have
 $$
 [A^+(x)_\la A^-(y)]=[\Theta_{\mathfrak s}(R_x)_\la \Theta_{\mathfrak s}(L_x)]=\Theta_{\mathfrak s}([R_x,L_x])+\half \la Tr(R_xL_y)=\half \la Tr(R_xL_y).
 $$
 
We claim that $Tr(R_xL_y) =0$. This will prove that 
$$
[A^+(x)_\l A^-(y)]=0,\ x,y\in sl(2).
$$

 Recall the form $(\cdot,\cdot)$ on $gl(2)$ defined in \eqref{invgl2}.
Consider the orthonormal basis $\{a_1,a_2,a_3,a_4\}$ of $gl(2)$ with
$$
a_1=\frac{\sqrt{-1}}{\sqrt{2}}\begin{pmatrix}1&0\\0&1\end{pmatrix},\ a_2=\frac{\sqrt{-1}}{\sqrt{2}}\s_J,\ a_3=\frac{\sqrt{-1}}{\sqrt{2}}\s_K,\ a_4=\frac{\sqrt{-1}}{\sqrt{2}}\s_{JK}. 
$$ 
Since $(\phi_1(a),\phi_1(b))=(a,b)$ for $a,b\in gl(2)$, 
the set $\{s_1,s_2,s_3,s_4\mid s_i=\phi_1(a_i)\}$ is an orthonormal basis of $\mathfrak s$. Then
 $$
  Tr(R_xL_y)=\sum_i(R_xL_y(s_i),s_i)=\sum_i(ya_ix,a_i)
  $$
  It is enough to check that $\sum_i(a_ra_ia_s,a_i) =0$ for all $r,s$ such that $r\ne1$ and $s\ne1$. Note that, if $r\ne1$, then  $a_ra_i=a_ia_r$ if $i=1,r$ and $a_ra_i=-a_ia_r$ if $i\ne1,r$ Thus
  $$
  \sum_{i}(a_ra_ia_s,a_i)= \!\! \sum_{i=1,r}(a_ia_ra_s,a_i)-\sum_{i\ne1,r}(a_ia_ra_s,a_i)=\sum_{i=1,r}(a_ra_s,a_i^{-1}a_i)-\sum_{i\ne1,r}(a_ra_s,a_i^{-1}a_i)=0.
  $$
  
We now check that 
$$
[A^\pm(x)_\l A^\pm(y)]=A^\pm([x,y])+k^\pm\la Tr(xy)\vac,\quad x,y\in sl(2).
$$
with $k^+=k_1+1$ and $k^-=n-1$.
For this it is enough to check that
\begin{equation}\label{AmAm}
[A^\pm(a_i)_\l A^\pm(a_j)]=A^\pm([a_i,a_j])+\d_{ij}k^\pm\la \vac,\quad 2\le i,j\le 4.
\end{equation}
Let $\psi:sl(2)\to so(\q,(\cdot,\cdot))$ be the Lie algebra homomorphism  obtained by linearly extending the mapping $\s_J\mapsto J$, $\s_K\mapsto K$, $\s_{JK}\mapsto JK$. Then 
$$
\begin{aligned}
[A^-(a_i)_\l A^-(a_j)]&=[\Theta_\q(\psi(a_i))_\l \Theta_\q(\psi(a_j))]=\Theta_\q([\psi(a_i),\psi(a_j)])+\half \la Tr(\psi(a_i)\psi(a_j))\vac\\
&=A^-([a_i,a_j])+\half \la Tr(\psi(a_i)\psi(a_j))\vac.
\end{aligned}
$$
If $i\ne j$ then  $\psi(a_i)\psi(a_j)=-\psi(a_j)\psi(a_i)$, so $Tr(\psi(a_i)\psi(a_j))=0$.
If $i= j$ then  $\psi(a_i)^2=\half I_\q$, so $Tr(\psi(a_i)^2)=\half \dim\q=2(n-1)$, so \eqref{AmAm} is checked for $A^-$.

It remains only to check \eqref{AmAm}  for $A^+$. Indeed,
$$
\begin{aligned}
&[A^+(a_i)_\l A^+(a_j)]=[(\phi_1(a_i)+\Theta_{\mathfrak s}(R_{a_i}))_\l (\phi_1(a_j)+\Theta_{\mathfrak s}(R_{a_j}))]\\
&=\phi_1([a_i,a_j])+\Theta_{\mathfrak s}([R_{a_i},R_{a_j}])+\la \k(\phi_1(a_i),\phi_1(a_i))+\half\la Tr(R_{a_i}R_{a_j})\\
&=\phi_1([a_i,a_j])+\Theta_{\mathfrak s}(R_{[a_i,a_j]})+\la k_1(\phi_1(a_i),\phi_1(a_i))+\half\la Tr(R_{a_i}R_{a_j})\\
&=A^+([a_i,a_j])+\la k_1 (\phi_1(a_i),\phi_1(a_i))+\half\la Tr(R_{a_i}R_{a_j}).
\end{aligned}
$$
Recall that $(\phi_1(x),\phi_1(y))=(x,y)$ so $(\phi_1(a_i),\phi_1(a_j))=\d_{ij}$. Since $R_{a_i}R_{a_j}=-R_{a_j}R_{a_i}$ when $i\ne j$ and $R_{a_i}^2=\half I_{\mathfrak s}$ we see that
$$
Tr(R_{a_i}R_{a_j})=\half\dim(\mathfrak s)\d_{ij}=2\d_{ij}
$$
hence \eqref{AmAm}  holds also for $A^+$.
\end{proof}
As in Section \ref{setup}, we can construct the field $\widetilde G$ and the twisted fields $\widetilde G_J$, $\widetilde G_K$ and $\widetilde G_{JK}$. Recall that these fields live in $V^{\mathbf k}(\g)\otimes F(\ov{\mathfrak q})$.
From Remark  \ref{GGJ}, we obtain
that
$$
\begin{aligned}
&[\widetilde G_\la \widetilde G_J]=\frac{\sqrt{-1}}{4}(\partial+2\la)\widetilde J,\\
&[\widetilde G_\la \widetilde G_K]=\frac{\sqrt{-1}}{4}(\partial+2\la)\widetilde K,\\
&[\widetilde G_\la \widetilde G_{JK}]=\frac{\sqrt{-1}}{4}(\partial+2\la)\widetilde{JK}.
\end{aligned}
$$
Recall from \eqref{betterJtilde} that
$$
\tilde J=\sqrt{-1}\Theta_\q(J)+2\widetilde \rho+ 2\Theta_\q(ad(\widetilde \rho)).
$$
In this special case
$$
\widetilde \rho=\frac{1}{2(k_1+n)}(h_\theta+\sum_{(\a,\theta)=1}h_\alpha).
$$
Note that $(\a,\theta)=1$ if and only if $(\theta-\a,\theta)=1$. This implies that 
$$
\sum_{(\a,\theta)=1}h_\alpha=(n-2)h_\theta
$$
so
\begin{equation}\label{rhotildeN4}
\widetilde \rho=\frac{n-1}{2(k_1+n)}h_\theta.
\end{equation}
Since, in our setting, $J_1=\phi_1(\s_J)=\sqrt{-1}h_\theta$, we have
$$
\tilde J=\sqrt{-1}\Theta_\q(J)-\sqrt{-1}\frac{n-1}{k_1+n}(J_1+ \Theta_\q(ad(J_1)),
$$
which we rewrite as
$$
\begin{aligned}
&\tilde J=\sqrt{-1}\Theta_\q(J)-\sqrt{-1}\frac{n-1}{k_1+n}(J_1+ \Theta_\q(ad(J_1)-\Theta_\q(J)+\Theta_\q(J))\\
&=\sqrt{-1}\frac{k_1+1}{k_1+n}\Theta_\q(J)-\sqrt{-1}\frac{n-1}{k_1+n}(J_1+ \Theta_\q(ad(J_1)-\Theta_\q(J)).
\end{aligned}
$$
Since $J_{|\mathfrak f_1}=ad(J_1)_{|\mathfrak f_1}$ we obtain
$$
 \Theta_\q(ad(J_1))-\Theta_\q(J)=\Theta_{\mathfrak s}(ad(J_1)-J),
 $$
so
\begin{equation}\label{widetildeJ}
\widetilde J=\sqrt{-1}\frac{k_1+1}{k_1+n}A^-(J)-\sqrt{-1}\frac{n-1}{k_1+n}A^+(J),
\end{equation}
and, similarly,
\begin{equation}\label{widetildeothers}
\begin{aligned}
&\widetilde K=\sqrt{-1}\frac{k_1+1}{k_1+n}A^-(K)-\sqrt{-1}\frac{n-1}{k_1+n}A^+(K),\\
&\widetilde {JK}=\sqrt{-1}\frac{k_1+1}{k_1+n}A^-(JK)-\sqrt{-1}\frac{n-1}{k_1+n}A^+(JK).\\
\end{aligned}
\end{equation}

Set
\begin{equation}\label{varie}
\begin{aligned}h=\begin{pmatrix}1&0\\0&-1\end{pmatrix},\ e&=\begin{pmatrix}0&1\\0&0\end{pmatrix},\ f=\begin{pmatrix}0&0\\1&0\end{pmatrix},\ h_1=\begin{pmatrix}1&0\\0&0\end{pmatrix},\ h_2=\begin{pmatrix}0&0\\0&1\end{pmatrix},
\\
\widetilde d_1&=\frac{1}{\sqrt{k_0}}(d_1)_{\g_0}+\frac{1}{\sqrt{k_1+n}}(d_1)_{\g_1},
\end{aligned}\end{equation}
and
 \begin{equation}\label{campi}
 \begin{aligned}
 &\widetilde J^0= A^+(h)
  &&\widetilde J^{'0}=-A^-(h)\\
  &\widetilde J^+ =-A^+(e)
    &&\widetilde J^{'+}=-A^-(f)\\
  & \widetilde J^-=-A^+(f)
  && \widetilde J^{'-}=-A^-(e)\\
  &\widetilde G^{++}=\widetilde G+\sqrt{-1}\widetilde G_J
  &&\widetilde G^{--}=\widetilde G-\sqrt{-1}\widetilde G_J\\
   &\widetilde G^{-+}=\widetilde G_K+\sqrt{-1}\widetilde G_{JK}
  &&\widetilde G^{+-}=\widetilde G_K-\sqrt{-1}\widetilde G_{JK}\\
  &\si^{++}=\tfrac{n-1}{\sqrt{k_1+n}}\ov{\phi_1( h_2)}
  &&\si^{--}=\tfrac{k_1+1}{\sqrt{k_1+n}}\ov{\phi_1( h_1)}
\\
    &\si^{+-}=\tfrac{n-1}{\sqrt{k_1+n}}\ov{\phi_1(e)}
  &&\si^{-+}=-\tfrac{k_1+1}{\sqrt{k_1+n}}\ov{\phi_1( f)}\\
  &\xi=\tfrac{\sqrt{(k_1+1)(n-1)}}{\sqrt{2(k_1+n)}}\left(\widetilde d_1+\Theta_\q(ad(\widetilde d_1))\right).
   \end{aligned}
   \end{equation}
\begin{proposition}\label{RP} The fields \eqref{campi} together with $L$ verify relations \eqref{prime}, \eqref{seconde}, \eqref{terze}, \eqref{quarte} with central charge $\tilde c$ equal to 
$C$ given by \eqref{centralCharge} and parameter $a=\frac{k_1+1}{n-1}$.
\end{proposition}

 \begin{proof}In this special case formula \eqref{centralCharge} gives 
  \begin{equation}\label{centralcharge}
C=\frac{6 (k_1+1) (n-1)}{k_1+n}, 
\end{equation}
so, setting $\tilde c=C$ and $a=\frac{k_1+1}{n-1}$, we find
$$
\frac{\tilde c(a+1)}{6}=k_1+1,\quad \frac{\tilde c(a+1)}{6a}=n-1,
$$
hence relations \eqref{gnaturalhat} hold by Proposition \ref{Apamdef}.

We now tackle relations \eqref{quarte}. By Proposition \ref{ParabisN=2} and \eqref{widetildeJ},
 $$
 \begin{aligned}
 [(\widetilde G^{++})_\l \widetilde G^{--}]&=L + (\tfrac{1}{2}\partial+\lambda)\widetilde J+\tfrac{\lambda^2}{6}C \vac=\\
&=L + \tfrac{n-1}{2(k_1+n)}(\partial+2\lambda)
(\widetilde J^0+\tfrac{k_1+1}{n-1}\widetilde J^{'0})+\tfrac{\lambda^2}{6}C\vac\\
&=L + \tfrac{1}{2(a+1)}(\partial+2\lambda)
(\widetilde J^0+a\widetilde J^{'0})+\tfrac{\lambda^2}{6}\tilde c \vac.
\end{aligned}
$$

Since $\widetilde G^{-+}=K\widetilde G^{--}$ and $\widetilde G^{+-}=K\widetilde G^{++}$, we obtain 
$$
\begin{aligned}
[(\widetilde G^{-+})_\l \widetilde G^{+-}]&=[K\widetilde G^{--}_\l K\widetilde G^{++}]=K[(\widetilde G^{--})_{\l} \widetilde G^{++}]=K[(\widetilde G^{++})_{-\l-\partial} \widetilde G^{--}]\\
&=K(L-(\half\partial+\lambda)\widetilde J+\tfrac{\lambda^2}{6}C \vac)=L-(\half\partial+\lambda)K(\widetilde J)+\tfrac{\lambda^2}{6}C \vac).
\end{aligned}
$$

Since $KJ=-JK$, it follows that $K(\Theta_\q(J))=-\Theta_\q(J)$ and, since $L_KR_J=R_JL_K$, $K(\Theta_{\mathfrak s}(R_J))=\Theta_{\mathfrak s}(R_J)$. This implies
$
K(A^+(J))=A^+(J)$ and  $K(A^-(J))=-A^-(J)$,
hence 
\begin{equation}\label{KofJtilde}K(\widetilde J)=-\sqrt{-1}\frac{k_1+1}{k_1+n}A^-(J)-\sqrt{-1}\frac{n-1}{k_1+n}A^+(J)=\frac{n-1}{k_1+n}\widetilde J^{0}-\frac{k_1+1}{k_1+n}\widetilde J^{'0} 
\end{equation}
so
$$
\begin{aligned}
[(\widetilde G^{-+})_\l \widetilde G^{+-}]&=L +\tfrac{n-1}{2(k_1+n)}(\partial+2\lambda)
(-\widetilde J^0+\tfrac{k_1+1}{n-1} \widetilde J^{'0})+\tfrac{\lambda^2}{6}C \vac\\&=L +\tfrac{1}{2(a+1)}(\partial+2\lambda)
(-\widetilde J^0+a \widetilde J^{'0})+\tfrac{\lambda^2}{6}\tilde c \vac.
\end{aligned}
$$
Using Remark \ref{GGJ}, we find
$$
\begin{aligned}
[(\widetilde G^{++})_\l\widetilde G^{-+}]&=[(\widetilde G +\sqrt{-1}\widetilde G_{J})_\l (\widetilde G_K+\sqrt{-1}\widetilde G_{JK})]\\
&=[\widetilde G_\l \widetilde G_K]-J[\widetilde G_\l \widetilde G_K]+\sqrt{-1}(-J[\widetilde G_\l \widetilde G_{JK}]+[\widetilde G_\l \widetilde G_{JK}]))\\
 &=\tfrac{\sqrt{-1}}{4}(\partial+2\l)(\widetilde K-J(\widetilde K)) -\tfrac{1}{4}(\partial+2\l)(\widetilde{JK}-J(\widetilde{JK}).
\end{aligned}
$$
Similarly to \eqref{KofJtilde}, we have
\begin{equation}\label{JofKtilde}
\begin{aligned}
J(\widetilde K)&=-\sqrt{-1}\frac{k_1+1}{k_1+n}A^-(K)-\sqrt{-1}\frac{n-1}{k_1+n}A^+(K)\\
J(\widetilde{JK})&=-\sqrt{-1}\frac{k_1+1}{k_1+n}A^-(JK)-\sqrt{-1}\frac{n-1}{k_1+n}A^+(JK).
\end{aligned}
\end{equation}
so $\widetilde K-J(\widetilde K)=2\sqrt{-1}\frac{k_1+1}{k_1+n}A^-(K)$ and $\widetilde{JK}-J(\widetilde{JK})=2\sqrt{-1}\frac{k_1+1}{k_1+n}A^-(JK)$.
Then
  $$
   \begin{aligned}
    [(\widetilde G^{++})_\l\widetilde G^{-+}]  &=-\tfrac{1}{2}\frac{k_1+1}{k_1+n}(\partial+2\l)(A^-(K +\sqrt{-1}JK))=\frac{k_1+1}{k_1+n}(\partial+2\l)(A^-(f))\\
    &=-\frac{a}{a+1}(\partial+2\l)\widetilde J^{'+}.
\end{aligned}
 $$
Analogously, we find
$$
\begin{aligned}
[(\widetilde G^{+-})_\l\widetilde G^{--}]&=[(\widetilde G_K) -\sqrt{-1}\widetilde G_{JK})_\l (\widetilde G-\sqrt{-1}\widetilde G_{J})]\\
&=[(\widetilde G_K)_\l \widetilde G]-J[(\widetilde G_K)_\l \widetilde G_J]-\sqrt{-1}(-J[(\widetilde G_{JK})_\l \widetilde G]+[(\widetilde G_{JK})_\l \widetilde G]))\\
 &=-\tfrac{\sqrt{-1}}{4}(\partial+2\l)(\widetilde K-J(\widetilde K)) -\tfrac{1}{4}(\partial+2\l)(\widetilde{JK}-J(\widetilde{JK})\\
 &=-\tfrac{1}{2}\frac{k_1+1}{k_1+n}(\partial+2\l)(A^-(-K +\sqrt{-1}JK))=\frac{k_1+1}{k_1+n}(\partial+2\l)(A^-(e))\\
 &=-\frac{a}{a+1}(\partial+2\l)\widetilde J^{'-}.
\end{aligned}
$$
Skewsymmetry now implies 
$$
[(\widetilde G^{-+})_\l\widetilde G^{++}]=[(\widetilde G^{++})_{-\l-\partial}\widetilde G^{-+}]=\frac{a}{a+1}(\partial+2\l)\widetilde J^{'+}
$$
and
$$
[(\widetilde G^{--})_\l\widetilde G^{+-}]=[(\widetilde G^{+-})_{-\l-\partial}\widetilde G^{--}]=\frac{a}{a+1}(\partial+2\l)\widetilde J^{'-}.
$$
Next, we check 
$$
\begin{aligned}
[(\widetilde G^{-+})_\l\widetilde G^{--}]&=[(\widetilde G_K +\sqrt{-1}\widetilde G_{JK})_\l (\widetilde G-\sqrt{-1}\widetilde G_{J})]\\
&=[(\widetilde G_K)_\l \widetilde G]+J[(\widetilde G_K)_\l \widetilde G]+\sqrt{-1}([(\widetilde G_{JK})_\l \widetilde G]+J[(\widetilde G_{JK})_\l \widetilde G]))\\
 &=-\tfrac{\sqrt{-1}}{4}(\partial+2\l)(\widetilde K+J(\widetilde K)) +\tfrac{1}{4}(\partial+2\l)(\widetilde{JK}+J(\widetilde{JK}))\\
 &=\frac{n-1}{k_1+n}(\partial+2\l)(-\tfrac{1}{2}A^+(K) -\tfrac{\sqrt{-1}}{2}A^+(JK))=\frac{n-1}{k_1+n}(\partial+2\l)A^+(f)\\
 &=-\frac{1}{a+1}(\partial+2\l)\widetilde J^-
\end{aligned}
$$
and
$$
\begin{aligned}
[(\widetilde G^{++})_\l\widetilde G^{+-}]&=[(\widetilde G +\sqrt{-1}\widetilde G_{J})_\l (\widetilde G_K-\sqrt{-1}\widetilde G_{JK})]\\
&=[\widetilde G_\l \widetilde G_K]+J[\widetilde G_\l \widetilde G_K]-\sqrt{-1}([\widetilde G_\l \widetilde G_{JK}]+J[\widetilde G_\l \widetilde G_{JK}])\\
 &=\tfrac{\sqrt{-1}}{4}(\partial+2\l)(\widetilde K+J(\widetilde K)) +\tfrac{1}{4}(\partial+2\l)(\widetilde{JK}+J(\widetilde{JK}))\\
 &=\frac{n-1}{k_1+n}(\partial+2\l)(\tfrac{1}{2}A^+(K) -\tfrac{\sqrt{-1}}{2}A^+(JK))=\frac{n-1}{k_1+n}(\partial+2\l)A^+(e)\\
 &=-\frac{1}{a+1}(\partial+2\l)\widetilde J^+.
\end{aligned}
$$
Again, skewsymmetry implies 
$$
[(\widetilde G^{--})_\l\widetilde G^{-+}]=\frac{1}{a+1}(\partial+2\l)\widetilde J^{-},\quad
[(\widetilde G^{+-})_\l\widetilde G^{++}]=\frac{1}{a+1}(\partial+2\l)\widetilde J^{+}.
$$
The relations
$$
[\widetilde G^{++}{}_\la \widetilde G^{++}]=[\widetilde G^{--}{}_\la \widetilde G^{--}]=[\widetilde G^{+-}{}_\la \widetilde G^{+-}]=[\widetilde G^{-+}{}_\la \widetilde G^{-+}]=0
$$
follow from  \eqref{gpgpfinale} and \eqref{gpgmfinale}.
We are now done with relations \eqref{quarte}.

Relations \eqref{prime} require some preparation. Using \eqref{fff3} and \eqref{ff3}, one computes that, in general, if $T\in so(\q,(\cdot,\cdot))$,
$$
\begin{aligned}
[\Theta_\q(T)_\la  \widetilde G]&=\sum_{s,j}\frac{1}{2\sqrt{k_s+h_s}}:q_{sj}\ov{T(q^{sj})}:\\
&-\frac{1}{6}\sum_{s,i,j}\frac{1}{2\sqrt{k_s+h_s}}:\left(\overline{T([q_{si},q_{sj}]_\q)-[q_{si},T(q_{sj})]_\q-[T(q_{si}),q_{sj}]_\q}\right)\ov q^{si}\ov q^{sj}:\\
&-\half \la\sum_{s,i}\frac{1}{2\sqrt{k_s+h_s}}\ov{[q_{si},T(q^{si})]}_\q.
\end{aligned}
$$
We apply this formula to our special case with  $T=J$. Since, by definition, 
\begin{equation}\label{Jext}
J(a_\q)=J(a)=J(a)_\q,\quad a\in\g,
\end{equation}
 we can write
$$
\begin{aligned}
[\Theta_\q(J)_\la  \widetilde G]&=\sum_{s,j}\frac{1}{2\sqrt{k_s+h_s}}:q_{sj}\ov{J(q^{sj})}:\\
&-\frac{1}{6}\sum_{s,i,j}\frac{1}{2\sqrt{k_s+h_s}}:\left(\overline{J([q_{si},q_{sj}])-[q_{si},J(q_{sj})]-[J(q_{si}),q_{sj}]}\right)_{\mathfrak q}\ov q^{si}\ov q^{sj}:\\
&-\half \la\sum_{s,i}\frac{1}{2\sqrt{k_s+h_s}}\ov{[q_{si},J(q^{si})]}_\q.
\end{aligned}
$$
and, recalling \eqref{NT}, we  obtain
$$
\begin{aligned}
[\Theta_\q(J)_\la  \widetilde G]&=\sum_{s,j}\frac{1}{2\sqrt{k_s+h_s}}:q_{sj}\ov{J(q^{sj})}:\\
&-\frac{1}{6}\sum_{s,i,j}\frac{1}{2\sqrt{k_s+h_s}}:\overline{\left(J([J(q_{si}),J(q_{sj})])\right)_{\q}}\ov q^{si}\ov q^{sj}:
-\half \la\sum_{s,i}\frac{1}{2\sqrt{k_s+h_s}}\ov{[q_{si},J(q^{si})]}_\q.
\end{aligned}
$$
 Combining  \eqref{rotilde} and \eqref{rhotildeN4}, we find that, in our special cases,
$$
\sum_{s,i}\frac{1}{2\sqrt{k_s+h_s}}\ov{[q_{si},J(q^{si})]}_\q=-\sqrt{-1}\frac{(n-1)}{\sqrt{k_1+n}}h_\theta,
$$
so, applying \eqref{Jext} again,
$$
\begin{aligned}
[\Theta_\q(J)_\la  \widetilde G]&=\sum_{s,j}\frac{1}{2\sqrt{k_s+h_s}}:q_{sj}\ov{J(q^{sj})}:\\
&-\frac{1}{6}\sum_{s,i,j}\frac{1}{2\sqrt{k_s+h_s}}:\overline{\left(J([J(u_{si}),J(u_{sj})]_{\q})\right)}\ov u^{si}\ov u^{sj}:
+\la\sqrt{-1}\frac{(n-1)}{2\sqrt{k_1+n}}\ov {(h_\theta)_\q}.
\end{aligned}
$$
with $\{u_{si}\}$ a basis of $(\mathfrak n_J^+\oplus \mathfrak n_J^-)\cap \g_s$ and $\{u^{si}\}$ its dual basis. Since $J$ stabilizes $(\mathfrak n_J^+\oplus \mathfrak n_J^-)\cap \g_s$ and since $\sqrt{-1}(h_\theta)_\q=J_1$ we can rewrite the latter formula as
$$
\begin{aligned}
[\Theta_\q(J)_\la  \widetilde G]&=\sum_{s,j}\frac{1}{2\sqrt{k_s+h_s}}:q_{sj}\ov{J(q^{sj})}:\\
&-\frac{1}{6}\sum_{s,i,j}\frac{1}{2\sqrt{k_s+h_s}}:\overline{\left(J([u_{si},u_{sj}]_{\q})\right)}\ov{J( u^{si})}\ov{J( u^{sj})}:
+\la\frac{(n-1)}{2\sqrt{k_1+n}}\ov {J_1}.
\end{aligned}
$$
which ultimately reads
\begin{equation}\label{ThetaJlQ}
[\Theta_\q(J)_\la  \widetilde G]=\widetilde G_J+\la\frac{(n-1)}{2\sqrt{k_1+n}}\ov {J_1}.
\end{equation}
Likewise,
\begin{align}\label{ThetaKlQ}
[\Theta_\q(K)_\la  \widetilde G]&=\widetilde G_K+\la\frac{(n-1)}{2\sqrt{k_1+n}}\ov {K_1},\\
[\Theta_\q(JK)_\la  \widetilde G]&=\widetilde G_{JK}+\la\frac{(n-1)}{2\sqrt{k_1+n}}\ov {JK_1}.
\end{align}
  
 In the subsequent calculations, we will use repeatedly the fact that, if $g\in SO(\q,(\cdot,\cdot))$ and $T\in so(\q,(\cdot,\cdot))$, then 
$g(\Theta_\q(T))=\Theta_\q(gTg^{-1})$.
We are now ready to check relations \eqref{prime}: 
$$
 \begin{aligned}[(\widetilde{J}^{'0})_\l \widetilde G^{++}]&=- [A^-(h)_\l  (\widetilde G+\sqrt{-1}\widetilde G_J] =\sqrt{-1}[\Theta_\q(J)_\l  (\widetilde G+\sqrt{-1}\widetilde G_J] \\
 &=\sqrt{-1}\left( \widetilde G_J+\la\frac{n-1}{2\sqrt{k_1+n}}\ov {J_1}+\sqrt{-1}[\Theta_\q(J)_\l \widetilde G_J]\right)\\
 &=\sqrt{-1}\left( \widetilde G_J+\la\frac{n-1}{2\sqrt{k_1+n}}\ov {J_1}+\sqrt{-1}J[\Theta_\q(J)_\l \widetilde G]\right)\\
 &=\sqrt{-1}\left( \widetilde G_J+\la\frac{n-1}{2\sqrt{k_1+n}}\ov {J_1}+\sqrt{-1}J(\widetilde G_J+\la\frac{n-1}{2\sqrt{k_1+n}}\ov {J_1})\right)
 \\
 &=\sqrt{-1}\left( \widetilde G_J-\sqrt{-1}\widetilde G+\la\frac{n-1}{2\sqrt{k_1+n}}(\ov {\phi_1(\s_J)}+\sqrt{-1}\ov{J(\phi_1(\s_J))}\right)
  \\
   &=\sqrt{-1}\left( \widetilde G_J-\sqrt{-1}\widetilde G+\la\frac{n-1}{2\sqrt{k_1+n}}(\ov {\phi_1(\s_J)+\sqrt{-1}\s^2_J))}\right)
  \\
 &=\widetilde G^{++}+\la\tfrac{n-1}{\sqrt{k_1+n}}\ov{\phi_1(h_2)}=\widetilde G^{++}+\la\s^{++}.
 \end{aligned}
   $$
Similarly
$$
   \begin{aligned} [\tilde{J}^{\prime 0}{}_{\lambda} \tilde{G}^{--}]&= [\tilde{J}^{\prime 0}{}_{\lambda}(\widetilde G-\sqrt{-1}\widetilde G_{J})]  \\
 &=\sqrt{-1}\left( \widetilde G_J+\sqrt{-1}\widetilde G+\la\tfrac{n-1}{2\sqrt{k_1+n}}(\ov J_1-\sqrt{-1}\ov{J(\phi_1(\s_J))})\right)\\
 &=-\widetilde G^{--}-\la\tfrac{n-1}{\sqrt{k_1+n}}\ov{\phi_1(h_1)}=-\widetilde G^{--}-\tfrac{\la}{a}\s^{--},
 \end{aligned}
   $$
  so that
  $$
   \begin{aligned} [\tilde{J}^{\prime 0}{}_{\lambda} \tilde{G}^{+-}]&= [\tilde{J}^{\prime 0}{}_{\lambda}K(
   \widetilde G^{++})] =-K[\tilde{J}^{\prime 0}{}_{\lambda}
   \widetilde G^{++}] =-\widetilde G^{+-}-\la K(\s^{++})\\
   &=-\widetilde G^{+-}-\la\tfrac{n-1}{\sqrt{k_1+n}}\ov{ \phi_1(\s_Kh_2)}=-\widetilde G^{+-}-\la \tfrac{n-1}{\sqrt{k_1+n}}\ov{\phi_1(e)} =-\widetilde G^{+-}-\la \s ^{+-}
 \end{aligned}
   $$
   and
  $$
   \begin{aligned} [\tilde{J}^{\prime 0}{}_{\lambda} \tilde{G}^{-+}]&= [\tilde{J}^{\prime 0}{}_{\lambda}K(
   \widetilde G^{--})] =-K[\tilde{J}^{\prime 0}{}_{\lambda}
   \widetilde G^{--}] =\widetilde G^{-+}+\tfrac{\la}{a} K(\s^{--})\\
   &=\widetilde G^{-+}+\tfrac{\la}{a} \tfrac{k_1+1}{\sqrt{k_1+n}} \ov{\phi_1(\s_Kh_1)}=\widetilde G^{-+}-\tfrac{\la}{a} \tfrac{k_1+1}{\sqrt{k_1+n}} \ov{\phi_1(f)}=\widetilde G^{+-}+\tfrac{\la}{a} \s ^{-+}.
 \end{aligned}
$$
Relations involving $[\widetilde J^{'\pm}_\l \widetilde G^{\pm\pm}]$ are obtained similarly using 
$$
\widetilde J^{'\pm}=A^-(\pm K+\sqrt{-1}JK)=\pm\Theta_\q(K)+\sqrt{-1}\Theta_\q(JK).
$$
To compute $[\widetilde J^{0}_\l \widetilde G^{++}]$, we use the equation \eqref{widetildeJ} to write
$$
\widetilde J^{0}=(a+1) \widetilde J-a\widetilde J^0.
$$
By \eqref{reN=2}, $[\widetilde J_\l \widetilde G^{++}]=\widetilde G^{++}$, so
$$
[\widetilde J^{0}{}_\l \widetilde G^{++}]=(a+1)\widetilde G^{++}-a(\widetilde G^{++}+\la\s^{++})=\widetilde G^{++}-a\la \s^{++}.
$$
The remaining relations in \eqref{prime} are computed similarly.

Relations \eqref{seconde} are quickly computed directly using the fact that, if $T\in so(V,B)$ and $v\in V$, then, by \eqref{1lambda2}
\begin{equation}\label{thetaTv}
[\Theta_V(T)_\l \ov v]=\ov{T(v)}.
\end{equation}
 
 The final task is the verification of the relations in \eqref{terze}.
 In general, if $q\in\q$, then 
 $$
 [\widetilde G_\la \ov q]=\sum_r\frac{1}{2\sqrt{k_r+h^\vee_r}}\left(q_{\g_i}+\half\sum_i:\overline{[q,q_{ri}]}_\q\ov q^{ri}:\right).
 $$
 In particular, since $J\phi_1(h_1)=\phi_1(\s_Jh_1)=\sqrt{-1}\phi_1(h_1)$,
 $$
   \begin{aligned}  [{\widetilde{G}^{++}}{}_{\lambda} \sigma^{--}]&=\tfrac{k_1+1}{\sqrt{k_1+n}}\left([\widetilde G{}_{\lambda}\ov{\phi_1(h_1)}]+\sqrt{-1}[(\widetilde G_J)_\la \ov{\phi_1(h_1)}]\right)\\
   &=\tfrac{k_1+1}{\sqrt{k_1+n}}\left([\widetilde G{}_{\lambda}\ov{\phi_1(h_1)}]+J[(\widetilde G)_\la \ov{\phi_1(h_1)}]\right),
   \end{aligned}
   $$
   $$
   [\widetilde G{}_{\lambda}\ov{\phi_1(h_1)}]=\frac{1}{2\sqrt{k_0}}\phi_1(h_1)_{\g_0}+\frac{1}{2\sqrt{k_1+n}}\phi_1(h_1)_{\g_1}+\frac{1}{4\sqrt{k_1+n}}(\sum_{i}:\ov{[\phi_1(h_1),q_{1i}]_\q}\ov q^{1i}:),
   $$
 hence
 \begin{align*}
&J(  [\widetilde G{}_{\lambda}\ov{\phi_1(h_1)}])=\\&\frac{1}{2\sqrt{k_0}}\phi_1(h_1)_{\g_0}+\frac{1}{2\sqrt{k_1+n}}\phi_1(h_1)_{\g_1}+\frac{1}{4\sqrt{k_1+n}}\sum_{i}:\ov{J([\phi_1(h_1),q_{1i})]_\q)}\ov {J(q^{1i})}:).
  \end{align*}
   Note that
   $$
   \begin{aligned}
   \sum_{i}:\ov{J([\phi_1(h_1),q_{1i})]_\q}\ov {J(q^{1i})}:)&=\sum_i\a_{1i}(\phi_1(h_1))\ov{J(q_{1i})}\ov {J(q^{1i})}\\&=\sum_i\a_{1i}(\phi_1(h_1))\ov{q_{1i}}\ov {q^{1i}}= \sum_{i}:\ov{[\phi_1(h_1),q_{1i})]_\q}\ov {q^{1i}}:
   \end{aligned}
   $$
   so that, since $\phi_1(h)\in\g_1$, 
 $$
   \begin{aligned}  &[{\widetilde{G}^{++}}{}_{\lambda} \sigma^{--}]\\
   &=\tfrac{k_1+1}{\sqrt{k_1+n}}\left(\frac{1}{\sqrt{k_0}}\phi_1(h_1)_{\g_0}+\frac{1}{\sqrt{k_1+n}}\phi_1(h_1)_{\g_1}+\frac{1}{2\sqrt{k_1+n}}\sum_{i}:\ov{[\phi_1(h_1),q_{1i}]_\q}\ov q^{1i}:\right)\\
  &=\tfrac{k_1+1}{\sqrt{k_1+n}}\left(\frac{1}{2\sqrt{k_0}}\phi_1(h)_{\g_0}+\frac{1}{2\sqrt{k_0}}(d_1)_{\g_0}+\frac{1}{2\sqrt{k_1+n}}\phi_1(h)_{\g_1}+\frac{1}{2\sqrt{k_1+n}}(d_1)_{\g_1}\right)\\
  &+\frac{k_1+1}{4(k_1+n)}\sum_{i}(:\ov{[\phi_1(h),q_{1i}]_\q}\ov q^{1i}:+:\ov{[d_1,q_{1i}]_\q}\ov q^{1i}:)\\
   &=\tfrac{k_1+1}{\sqrt{k_1+n}}\left(\frac{1}{2\sqrt{k_0}}(d_1)_{\g_0}+\frac{1}{2\sqrt{k_1+n}}(d_1)_{\g_1}+\frac{1}{2\sqrt{k_1+n}}\Theta_\q(ad(d_1))\right)\\
  &+\frac{k_1+1}{2(k_1+n)}(\phi_1(h)+\Theta_\q(ad(\phi_1(h))).
   \end{aligned}
   $$
      Note that $\phi_1(h)=-\sqrt{-1}J_1=-\sqrt{-1}A^+(\s_J)+\sqrt{-1}\Theta_{\mathfrak s}(ad(J_1)-J)$ and that
   $
   \Theta_\q(ad(\phi_1(h))=-\sqrt{-1}A^-(\s_J)-\sqrt{-1}\Theta_{\mathfrak s}(ad(J_1)-J)$, so $\phi_1(h)+ \Theta_\q(ad(\phi_1(h))=A^+(h)+A^-(h)=\widetilde J^0-\widetilde J^{'0}$. Moreover $ad(d_1)_{\g_0}=0$ since $\g_0$ is in the center of $\g$. This implies
   $$
   \begin{aligned}
    &[{\widetilde{G}^{++}}{}_{\lambda} \sigma^{--}]\\   &=\tfrac{k_1+1}{\sqrt{k_1+n}}\left(\frac{1}{2\sqrt{k_0}}(d_1)_{\g_0}+\frac{1}{2\sqrt{k_1+n}}(d_1)_{\g_1}+\Theta_\q(ad(\frac{1}{2\sqrt{k_0}}(d_1)_{\g_0}+\frac{1}{2\sqrt{k_1+n}}(d_1)_{\g_1}))\right)\\
  &+\frac{k_1+1}{2(k_1+n)}(\widetilde J^0-\widetilde J^{'0})=(\tfrac{a}{2})^{1/2}\xi+\frac{a}{2(a+1)}(\widetilde J^0-\widetilde J^{'0}).     \end{aligned}
         $$
         The remaining relations in \eqref{terze} involving $\widetilde G^{\pm\pm}$ and $\s^{\pm\pm}$ are checked in the same way.
         
         We now check the relations involving $\widetilde G^{\pm\pm}$ and $\xi$:

         $$
         \begin{aligned}
         &[\widetilde G^{++}{}_\la [{\widetilde{G}^{\,++}}{}_{\mu} \sigma^{--}]]= \frac{a}{2(a+1)}[\widetilde G^{++}{}_\la(\tilde{J}^0 - \tilde{J}^{\prime 0})]
         +\left(\tfrac{a}{2}\right)^{1/2} [\widetilde G^{++}{}_\la \xi ]\\
         &=- \frac{a}{2(a+1)}[(\tilde{J}^0 -\tilde{J}^{\prime 0})_{-\l-\partial}\widetilde G^{++}] +\left(\tfrac{a}{2}\right)^{1/2} [\widetilde G^{++}{}_\la \xi ]\\
         &=- \frac{a}{2}\left((\partial +\la )\s^{++}\right) +\left(\tfrac{a}{2}\right)^{1/2} [\widetilde G^{++}{}_\la \xi ].
         \end{aligned}
         $$
On the other hand, by Jacobi identity,      since  $[\widetilde G^{++}{}_\la   \widetilde G^{++}]=0$, we have 
  $$ [\widetilde G^{++}{}_\la [{\widetilde{G}^{\,++}}{}_{\mu} \sigma^{--}]]=   -[\widetilde G^{++}{}_\mu [{\widetilde{G}^{\,++}}{}_{\la}\sigma^{--}]= \frac{a}{2}\left((\partial +\mu )\s^{++}\right) -\left(\tfrac{a}{2}\right)^{1/2} [\widetilde G^{++}{}_\mu \xi ].
  $$
  Equating and plugging $\mu=\la$, we find 
  $$
   [{\widetilde{G}^{\,++}}{}_{\la} \sigma^{--}]=\left(\tfrac{a}{2}\right)^{1/2} (\partial +\la )\s^{++}.
  $$
  The other relations in \eqref{terze} are checked likewise.
  
  The remaining relations are 
$[\widetilde J^{s}{}_\la \xi]=[\widetilde J^{'s}{}_\la \xi]=[\s^{\pm\pm}{}_\la \xi]=0$ and the relations involving $[\s^{\pm\pm}{}_\la \s^{\pm\pm}]$ and $[\xi_\la \xi]$. These are all easily checked using the explicit realizations of the fields in $V^k(\g)\otimes F(\ov\q)$. As an example we compute explicitly here $[\xi_\la \xi]$: recalling that $(d_1,d_1)=-2$ and that, if $x,y\in \g_1=sl(n)$, $Tr(ad(x)ad(y))=2n(x,y)$, we find
$$
\begin{aligned}
&[\xi_\la \xi]=\tfrac{(k_1+1)(n-1)}{2(k_1+n)}[(\widetilde d_1+\Theta_\q(ad(\widetilde d_1)))_\la(\widetilde d_1+\Theta_\q(ad(\widetilde d_1)))]\\
&=\tfrac{(k_1+1)(n-1)}{2(k_1+n)}\la\left(((d_1)_{\g_0} ,(d_1)_{\g_0} )+\frac{k_1}{k_1+n}((d_1)_{\g_1} ,(d_1)_{\g_1} )+\half Tr(ad(\widetilde d_1)^2))\right)\vac\\
&=\tfrac{(k_1+1)(n-1)}{2(k_1+n)}\la\left(((d_1)_{\g_0} ,(d_1)_{\g_0} )+\frac{k_1}{k_1+n}((d_1)_{\g_1} ,(d_1)_{\g_1} )+\frac{n}{k_1+n}((d_1)_{\g_1} ,(d_1)_{\g_1} )\right)\vac\\
&=\tfrac{(k_1+1)(n-1)}{2(k_1+n)}\la(d_1,d_1)=-\la\tfrac{(k_1+1)(n-1)}{k_1+n}\vac=-\la\frac{\widetilde c}{6}\vac,
\end{aligned}
$$
as wished.
  \end{proof} 

\section{The vertex algebra $W^{k}_{\min}(D(2,1;a))$}
  Let $\g=D(2,1;a)$ with $a\in\R$, $a\ne 0,-1$.   
The vertex algebra $W^{k}_{\min}(\g)$ is  freely generated by fields  $J^0$, $J^{+}$, $J^-$, $J^{'0}$, $J^{'+}$, $J^{'-}$, $G^{++}$, $G^{+-}$, $G^{-+}$, $G^{--}$,  $\hat L$ and $\l$-brackets between them  given explicitly in 
\cite[\S\! 8.6]{KW1}, that we now recall (with a few amendments): $\hat L$ is a conformal vector with central charge $-6k-3$, 
the fields $J$'s (resp.~$G$'s) are  primary with respect to $\hat L$
of conformal weight $1$ (resp.~$3/2$).  The non-zero
$\lambda$-brackets between these fields are as follows:
\begin{equation}\label{LlambdaG}
\begin{aligned}
  &[{J^0}{}_{\lambda}J^0] = -2\lambda ((a+1)k+1) \vac\, , \,
     [{J^{\prime 0}}{}_{\lambda}J^{\prime 0}] =-2 \lambda
       (\frac{a+1}{a}k+1)\vac \, , \\
& [{J^+}{}_{\lambda}J^-] = J^0 -\lambda ((a+1)k+1)\vac\, , \,
     [{J^{\prime +}}{}_{\lambda} J^{\prime -}]=
     J^{\prime 0}-\lambda \ (\frac{a+1}{a}k+1 )\vac \, , \\
& [{J^0}{}_{\lambda}J^{\pm}] = \pm 2 J^{\pm}\, , \,
     [{J^{\prime 0}}{}_{\lambda} J^{\prime \pm}]=\pm
     2J^{\prime \pm}\, , \\
& [{J^0}{}_{\lambda} G^{+\pm}] =  G^{+\pm}\, , \,
     [{J^0}{}_{\lambda}G^{-\pm}] =- G^{-\pm}\, , \,
   [{J^{\prime 0}}{}_{\lambda} G^{\pm+}] =  G^{\pm+}
     \, , \, [{J^{\prime 0}}{}_{\lambda} G^{\pm-}]
     =- G^{\pm-}\, , \\
& [{J^+}{}_{\lambda}G^{- -}]= -G^{+-}\, , \,
     [{J^+}{}_{\lambda}G^{-+}] = -G^{++}\, , \,
     [{J^{\prime +}}{}_{\lambda} G^{--}] =G^{-+}\, , \,
     [{J^{\prime +}}{}_{\lambda}G^{+-}]=G^{++}\, , \\
& [{J^-}{}_{\lambda} G^{++}] = -G^{-+}\, , \,
     [{J^-}{}_{\lambda}G^{+-}] = -G^{--}\, , \,
     [{J^{\prime -}}{}_{\lambda}G^{++}] = G^{+-}\, , \,
     [{J^{\prime -}}{}_{\lambda}G^{-+}]=G^{--}. \\
\end{aligned}
\end{equation}
\begin{equation}\label{GlambdaG}
\begin{aligned}
& [{G^{++}}{}_{\lambda}G^{++}] = \frac{2a}{(a+1)^2}:
     J^+J^{\prime+} : \, , \, [{G^{--}}{}_{\lambda}G^{--}] =
     \frac{2a}{(a+1)^2} : J^-J^{\prime-} : \, , \\
&[{G^{-+}}{}_{\lambda}G^{-+}] = -\frac{2a}{(a+1)^2}:
     J^-J^{\prime +}: \, , \, [{G^{+-}}{}_{\lambda}G^{+-}]
     =-\frac{2a}{(a+1)^2}:J^+J^{\prime -}:\, ,\\
& [{G^{++}}{}_{\lambda}G^{-+}] = \frac{a}{(a+1)^2}:J^0J^{\prime+}
       :-\frac{a}{a+1}(k+\frac{1}{a+1})(\partial +2\lambda)
       J^{\prime +}\, , \\
&[{G^{++}}{}_{\lambda}G^{+-}] = -\frac{a}{(a+1)^2}:
       J^{\prime 0}J^+ : +\frac{1}{a+1}(k+\frac{a}{a+1})
       (\partial +2\lambda)J^+ \, , \\
&[{G^{--}}{}_{\lambda}G^{-+}] = \frac{a}{(a+1)^2}:
       J^{\prime 0}J^- :+\frac{1}{a+1}(k+\frac{a}{a+1})
       (\partial +2\lambda)J^- \, , \\
&[{G^{--}}{}_{\lambda}G^{+-}] = -\frac{a}{(a+1)^2}:
       J^0J^{\prime -} :-\frac{a}{a+1}(k+\frac{1}{a+1})
       (\partial +2\lambda)J^{\prime -}\, .
\end{aligned}
\end{equation}
\begin{equation}\label{GlambdaGbis}
\begin{aligned}
&[{G^{++}}{}_{\lambda}G^{--}] = k\hat L+\frac{1}{4}
     (\frac{1}{a+1}:J^0J^0 :+\frac{a}{a+1}:
  J^{\prime 0}J^{\prime 0}:
   -\frac{1}{(a+1)^2}:( J^0+aJ^{\prime 0})^2:) \\
    & +\frac{a}{(a+1)^2}(:J^+J^- :+: J^{\prime +}J^{\prime -}:)
       -\frac{1}{2(a+1)}\partial (J^0+aJ^{\prime 0})\\
    & +\frac{k+1}{2(a+1)}(\partial +2\lambda)(J^0+aJ^{\prime 0})
       -\frac{\lambda}{(a+1)^2} (J^0+a^2J^{\prime 0})-
       \lambda^2 (k(k+1)+\frac{a}{(a+1)^2})\vac\, , \\
&[{G^{-+}}{}_{\lambda}G^{+-}] = -k\hat L +\frac{1}{4}
       (-\frac{1}{a+1}: J^0J^0:-\frac{a}{a+1}:
       J^{\prime 0}J^{\prime 0}:+\frac{1}{(a+1)^2}:
       (J^0-aJ^{\prime 0})^2 :)\\
    & -\frac{a}{(a+1)^2}(:J^+J^-:+:J^{\prime +}J^{\prime -}:)
       +\frac{1}{2(a+1)}\partial (J^0+aJ^{\prime 0})-
       \frac{1}{(a+1)^2}\partial J^0 \\
    & +\frac{k+1}{2(a+1)}(\partial +2\lambda)(J^0-aJ^{\prime 0})
       -\frac{\lambda}{(a+1)^2} (J^0-a^2J^{\prime 0}) +
       \lambda^2  (k(k+1)+\frac{a}{(a+1)^2})\vac.
\end{aligned}
\end{equation}

Let $\C^4_{\bar1}$ be the odd space (over $\C$) with basis $\{\s^{++},\s^{+-},\s^{-+},\s^{--}\}$ equipped with the symmetric bilinear form $\langlerangle$  given by
\begin{equation}\label{formasigma1}\begin{aligned}
&\langle \s^{++},\s^{++}\rangle=\langle \s^{++},\s^{+-}\rangle=\langle\s^{++},\s^{-+}\rangle=\langle \s^{+-},\s^{+-}\rangle=0, \langle \s^{++},\s^{--}\rangle=k
\\
&\langle \s^{+-},\s^{-+}\rangle=k,\ \langle \s^{+-},\s^{--}\rangle=\langle \s^{-+},\s^{-+}\rangle=\langle \s^{-+},\s^{--}\rangle=\langle \s^{--},\s^{--}\rangle=0.
\end{aligned}\end{equation}
and let $F(\C^4_{\bar1})$ be the corresponding fermionic vertex algebra.

 Recall that $V^{k}(\C\xi)$ is the universal affine  vertex algebra with $\l$-bracket $[\xi_\l\xi]=(\xi|\xi)\vac$ where $(\cdot|\cdot)$ is the bilinear form on $\C\xi$ defined by setting $(\xi|\xi)=k$.

Define
\begin{equation}\label{Jtildas}
\begin{aligned}
 \widetilde{J}^0 &= J^0 -\tfrac{1}{k}: \sigma^{--} \sigma^{++}:
          + \tfrac{1}{k}: \sigma^{+-} \sigma^{-+}: \, , \,
         \widetilde{J}^{\prime 0} = J^{\prime 0}  -\tfrac{1}{k}: \sigma^{--}\sigma^{++}:
           - \tfrac{1}{k}:\sigma^{+-} \sigma^{-+}:\, ,\\
  \widetilde{J}^+ &= J^++\tfrac{a}{k}: \sigma^{+-}\sigma^{++}: \, , \,
         \widetilde{J}^- = J^- + \tfrac{1}{ak}
          :\sigma^{--}\sigma^{-+}: \, , \\
  \widetilde{J}^{\prime +} &= J^{\prime +} +
          \tfrac{1}{k}:\sigma^{-+}\sigma^{++}: \, , \,
         \widetilde{J}^{\prime -} = J^{\prime -}+
          \tfrac{1}{k}:\sigma^{--}\sigma^{+-}: \, ,
          \end{aligned}
          \end{equation}
\begin{equation}\label{GtildaandG}
\begin{aligned}
  \widetilde{G}^{++} &=& \tfrac{1}{\sqrt{k}}G^{++}-\tfrac{1}{k(a+1)}:J^+ \sigma^{-+}:
          + \tfrac{a}{k(a+1)}: J^{\prime +}\sigma^{+-}:
          + \tfrac{a}{2k(a+1)}:J^0 \sigma^{++}:\\
      && -\tfrac{a}{2k(a+1)} :J^{\prime 0}\sigma^{++}:
          +\tfrac{1}{k} \left( \tfrac{a}{2}\right)^{1/2}:\xi \sigma^{++}:
          +\tfrac{a}{k^2 (a+1)}:\sigma^{++}\sigma^{+-}\sigma^{-+}:\, ,\\
   \widetilde{G}^{--} &=&\tfrac{1}{\sqrt{k}} G^{--} + \tfrac{a}{k(a+1)}: J^- \sigma^{+-}:
          -\tfrac{1}{k(a+1)}: J^{\prime -} \sigma^{-+}:
          -\tfrac{1}{2k(a+1)}: J^0 \sigma^{--}:\\
      && + \tfrac{1}{2k(a+1)}: J^{\prime 0}\sigma^{--}:
          +\tfrac{1}{k}\left( \tfrac{1}{2a}\right)^{1/2}:\xi \sigma^{--}:
          -\tfrac{1}{k^2(a+1)}:\sigma^{--}\sigma^{+-}\sigma^{-+}:\, ,\\
   \widetilde{G}^{+-} &=& -\tfrac{1}{\sqrt{k}}G^{+-} +\tfrac{1}{k(a+1)}:J^+\sigma^{--}:
          +\tfrac{a}{k(a+1)}: J^{\prime -}\sigma^{++}:
          +\tfrac{a}{2k(a+1)}: J^0 \sigma^{+-}: \,  \\
      && +\tfrac{a}{2k(a+1)}: J^{\prime 0} \sigma^{+-}:
          + \tfrac{1}{k}\left( \tfrac{a}{2}\right)^{1/2}:\xi \sigma^{+-}:
          -\tfrac{a}{k^2 (a+1)}:\sigma^{+-}\sigma^{--}\sigma^{++}:\, ,\\
   \widetilde{G}^{-+} &=&\tfrac{1}{\sqrt{k}}G^{-+}-\tfrac{1}{k(a+1)}: J^{\prime +}
          \sigma^{--}:-\tfrac{a}{k(a+1)}: J^- \sigma^{++}:
          -\tfrac{1}{2k(a+1)}:J^0 \sigma^{-+}: \,  \\
      && -\tfrac{1}{2k(a+1)}:J^{\prime 0}\sigma^{-+}:+\tfrac{1}{k}
          \left( \tfrac{1}{2a}\right)^{1/2}:\xi \sigma^{-+}:
          +\tfrac{1}{k^2 (a+1)}:\sigma^{-+}\sigma^{--}\sigma^{++}:\,,
\end{aligned}
\end{equation}
\begin{equation}\label{LL}    \widetilde L =\hat L + \tfrac{1}{2k}(: \partial \sigma^{--}\sigma^{++}:
           +:\partial \sigma^{++} \sigma^{--}:
           +:\partial\sigma^{-+}\sigma^{+-}:
           +:\partial \sigma^{+-}\sigma^{-+}:+ : \xi^2 :)\, .
\end{equation}

The above formulas establish  a vertex algebra isomorphism

 \begin{equation}\label{WN=3embedding}
V^{N=4}_{\tilde c,a}\cong W^{k}_{\min}(D(2,1;a))\otimes V^{k}(\C\xi)\otimes F(\C^4_{\bar1}).
\end{equation} 
where $\tilde c=-6 k$.
\begin{remark} Proposition \ref{RP} establishes a vertex algebra homomorphism
$$
\widetilde \Gamma:V^{N=4}_{\tilde c,a}\to V^\k(\g)\otimes F(\ov{\mathfrak q})$$
with $\tilde c=6\frac{(n-1)(k_1+1)}{k_1+n}$
hence, composing this map with the isomorphism in \eqref{WN=3embedding}, we obtain a vertex algebra  homomorphism \begin{equation}\label{embedding}\Gamma: W^{k}_{\min}(D(2,1;a))\to V^\k(\g)\otimes F(\ov{\mathfrak q})\end{equation}  with $k=-\tfrac{(n-1)(k_1+1)}{k_1+n}=-\frac{(n-1)a}{1+a}$, where $a=\frac{k_1+1}{n-1}$. 
\end{remark}

\section{Application to unitarity of massless non-twisted representations}\label{Application}
Here we describe the construction of unitary massless representations as given in \cite{Gunaydin}. The massless representations are called extremal in \cite{KMP1}.

Let $\g$ be as in Examples \ref{Ex1} and \ref{Ex2}. In this Section, we choose  $k_1\in \mathbb Z_{\ge 0}$ and $k_0=k_1+n$ so that $\k+\hvee=(k_1+n)(\cdot,\cdot)$.

Let $\a_i$ be the root of $\g$ corresponding to the root vector $E_{i.i+1}$. We can choose the set of positive roots so that $\Pi=\{\a_1,\ldots,\a_{n-1}\}$ is the corresponding set of simple roots. Let $\{\omega_1,\ldots,\omega_{n-1}\}$ be the corresponding fundamental weights of $sl(n)$. Extend these to $\g$ by setting them to be zero on the center  $\g_0$.
 If $\nu\in\h^*$, we let $L(\nu)$ be the irreducible highest weight module (w.r.t. $\Dp$) for $V^\k(\g)$ of highest weight $\nu$ generated by a highest weight vector $v_\nu$.

\begin{lemma}\label{singularNS}Fix $0\le r_1\le k_1$ and let  $\Omega(r_1,r_2)$ be  the vector in  $L(r_1\omega_1)\otimes F(\ov \q)$ defined by
$$
\begin{cases}
\Omega(r_1,0)=v_{r_1\omega_1}\otimes\vac,\\
\Omega(r_1,r_2)=v_{r_1\omega_1}\otimes (\ov{E}_{n,n-r_2})_{-1/2}\dots (\ov{E}_{n,n-1})_{-1/2}\vac,\ r_2=1,\ldots,n-2.
\end{cases}
$$
Then the vectors $\Omega(r_1,r_2)$ are singular vectors for $W^{k}_{\min}(D(2,1;a))$.
\end{lemma}
\begin{proof}To simplify notation we set $\nu=r_1\omega_1$ throughout the proof.

We first prove that $A^-(e)_1\Omega(r_1,r_2)=0$ and that $A^-(f)_0\Omega(r_1,r_2)=0$.  Since $E_{n,j}\in \mathfrak f_1$ for $1\le j\le n-2$, $[A^-(x)_m,(\ov E_{n,j})_{-1/2}]=[\Theta_\q(ad(\phi_1(x)))_m,(\ov E_{n,j})_{-1/2}]$. By \eqref{thetaTv} and \cite[(2.49)]{DK},
$$
[\Theta_\q(ad(x))_n,\ov q_{-1/2}]=\ov{ad(x)(q)}_{n-1/2}.
$$
In particular
$$
\begin{aligned}
A^-(e)_1\cdot\Omega(r_1,r_2)&=v_{\nu}\otimes \Theta_\q(ad(\phi_1(e))_1\left((\ov{E}_{n,n-r_2})_{-1/2}\cdots (\ov{E}_{n,n-1})_{-1/2}\vac\right)\\
&=\sum_jv_{\nu}\otimes \left((\ov{E}_{n,n-r_2})_{-1/2}\cdots\ov{[E_{1,n},E_{n,j} ]}_{1/2}\cdots(\ov{E}_{n,n-1})_{-1/2}\vac\right)\\
&=\sum_jv_{\nu}\otimes \left((\ov{E}_{n,n-r_2})_{-1/2}\cdots(\ov{E}_{1,j})_{1/2}\cdots(\ov{E}_{n,n-1})_{-1/2}\vac\right)=0,
\end{aligned}
$$ 
and
$$
\begin{aligned}
A^-(f)_0\cdot\Omega(r_1,r_2)&=v_{\nu}\otimes \Theta_\q(ad(\phi_1(f))_0\left((\ov{E}_{n,n-r_2})_{-1/2}\cdots (\ov{E}_{n,n-1})_{-1/2}\vac\right)\\
&=\sum_jv_{\nu}\otimes \left((\ov{E}_{n,n-r_2})_{-1/2}\cdots\ov{[E_{n,1},E_{n,j} ]}_{-1/2}\cdots(\ov{E}_{n,n-1})_{-1/2}\vac\right)=0.
\end{aligned}
$$ 

Next we check that $A^+(e)_0\Omega(r_1,r_2)=0$ and that $A^+(f)_1\Omega(r_1,r_2)=0$.  Since $E_{n,j}\in \mathfrak f_1$ for $1\le j\le n-2$ and 
$$[A^+(x)_m,(\ov{E}_{n,j})_{-1/2}]=[\Theta_{\mathfrak s}(R_x)_m,(\ov{E}_{n,j})_{-1/2}], 
$$
it follows that 
$$
\begin{aligned}
A^+(e)_0\cdot\Omega(r_1,r_2)&=(x_{\theta})_0v_{\nu}\otimes\left((\ov{E}_{n,n-r_2})_{-1/2}\cdots (\ov{E}_{n,n-1})_{-1/2}\vac\right)\\
&+v_{\nu}\otimes \Theta_{\mathfrak s}(R_{e})_0\left((\ov{E}_{n,n-r_2})_{-1/2}\cdots (\ov{E}_{n,n-1})_{-1/2}\vac\right)\\
&=(x_\theta)_0v_{\nu}\otimes \left((\ov{E}_{n,n-r_2})_{-1/2}\cdots (\ov{E}_{n,n-1})_{-1/2}\vac\right)=0,
\end{aligned}
$$
and
$$
\begin{aligned}
A^+(f)_1\cdot\Omega(r_1,r_2)&=(x_{-\theta})_1v_{\nu}\otimes\left((\ov{E}_{n,n-r_2})_{-1/2}\cdots (\ov{E}_{n,n-1})_{-1/2}\vac\right)\\
&+v_{\nu}\otimes \Theta_{\mathfrak s}(R_{f})_1\left((\ov{E}_{n,n-r_2})_{-1/2}\cdots (\ov{E}_{n,n-1})_{-1/2}\vac\right)\\
&=(x_{-\theta})_1v_{\nu}\otimes \left((\ov{E}_{n,n-r_2})_{-1/2}\cdots (\ov{E}_{n,n-1})_{-1/2}\vac\right)=0.
\end{aligned}
$$
Since $\s^{\pm\pm}\in\ov{\mathfrak s}$ and $(\mathfrak s,\mathfrak f_1)=0$, we have $[(\s^{\pm\pm})_\l\mathfrak f_1]=0$ hence $[\s^{\pm\pm}{}_\la \ov E_{n,j}]=0$. This implies that, for $s\in\{+,0,-\}$, 
$$
[J^s{}_\l \ov E_{n,j}]=[\widetilde J^s{}_\la \ov E_{n,j}]= [:\sigma^{--}\sigma^{+-}\sigma^{-+}:{}_\la \ov{E}_{n,j}]=0
$$
and
$$
[J^{'s}{}_\l \ov E_{n,j}]=[\widetilde J^{'s}{}_\la \ov E_{n,j}].
$$
In particular, $J^{'+}_0\Omega(r_1,r_2)=\widetilde J^{'+}_0\Omega(r_1,r_2)=-A^-(f)_0\Omega(r_1,r_2)=0$ and $J^{'-}_1\Omega(r_1,r_2)=\widetilde J^{'-}_1\Omega(r_1,r_2)=-A^-(e)_1\Omega(r_1,r_2)=0$.\newline
Similarly $J^{+}_0\Omega(r_1,r_2)=\widetilde J^{+}_0\Omega(r_1,r_2)=A^+(e)_0\Omega(r_1,r_2)=0$ and $J^{-}_1\Omega(r_1,r_2)=\widetilde J^{-}_1\Omega(r_1,r_2)=A^+(f)_1\Omega(r_1,r_2)=0$.

By the relations \eqref{LlambdaG}, in order to check that $(G^{\pm\pm})_m\Omega(r_1,r_2)=0$ for $m>0$, it is enough to check that $(G^{--})_{1/2}\Omega(r_1,r_2)=0$.
Observe that
$$
[: J^{'-} \sigma^{+-}:_\la \ov{E}_{n,j}]=-[:\Theta_{\ov\q}(ad(\phi_1(e)))\sigma^{+-}:_\la \ov{E}_{n,j}]=:\ov{[E_{1,n},E_{n,j}]}\sigma^{+-}:=:\ov E_{1,j}\sigma^{+-}:
$$
and
$$
[: J^{'0} \sigma^{--}:_\la \ov{E}_{n,j}]=-[:\Theta_{\ov\q}(ad(\phi_1(h)))\sigma^{--}:_\la \ov{E}_{n,j}]=:\ov{E}_{n,j}\sigma^{--}:.
$$

By \eqref{GtildaandG},
$$
\begin{aligned}
G^{--}{}&=\widetilde G^{--}-\tfrac{1}{k(a+1)}(a: J^- \sigma^{+-}: +: J^{\prime -} \sigma^{-+}:+\tfrac{1}{2}: J^0 \sigma^{--}:)\\
      &
       - \tfrac{1}{2k(a+1)}(: J^{\prime 0}\sigma^{--}:
          -2(a+1)\left( \tfrac{1}{2a}\right)^{1/2}:\xi \sigma^{--}:          +\tfrac{2}{k}:\sigma^{--}\sigma^{+-}\sigma^{-+}:)\, .
\end{aligned}
$$
Moreover, since $d_1\in\g_1$ if $n>2$,
$$
\begin{aligned}
[\xi_\la\ov{E}_{n,j}]&=\tfrac{\sqrt{(k_1+1)(n-1)}}{\sqrt{2(k_1+n)}}[(\widetilde d_1+\Theta_\q(ad(\widetilde d_1)))_\la\ov{E}_{n,j}]=\tfrac{\sqrt{(k_1+1)(n-1)}}{\sqrt{2}(k_1+n)}\ov{[d_1,E_{n,j}]}\\
&=-\tfrac{\sqrt{(k_1+1)(n-1)n}}{\sqrt{2(2-n)}(k_1+n)}\ov E_{n,j},
\end{aligned}
$$
and 
$$
[:\xi \sigma^{--}:_\la\ov{E}_{n,j}]=\tfrac{\sqrt{(k_1+1)(n-1)n}}{\sqrt{2(2-n)}(k_1+n)}:\ov E_{n,j}\s^{--}:.
$$
These formulas imply that
$$
\begin{aligned}
&[(G^{--}-\widetilde G^{--})_{1/2} ,(\ov{E}_{n,j})_{-1/2}]=-\tfrac{1}{k(a+1)}: \ov E_{1,j} \sigma^{-+}:_0
       - \tfrac{1}{2k(a+1)}(1
          -\sqrt{\tfrac{n}{2-n}}):\ov E_{n,j}\s^{--}:_0\, .
\end{aligned}
$$
Since, for $1\le r\le n-2$,
$$
[: \ov E_{1,j} \sigma^{-+}:_\la \ov{E}_{n,r}]=[: \ov E_{n,j} \sigma^{--}:_\la \ov{E}_{n,r}]=0,
$$
we see that 
$$
[: \ov E_{1,j} \sigma^{-+}:_0, (\ov{E}_{n,r})_{-1/2}]=[: \ov E_{n,j} \sigma^{--}:_0,(\ov{E}_{n,r})_{-1/2}]=0,
$$
hence
$$
(G^{--}-\widetilde G^{--})_{1/2}\Omega(r_1,r_2)=0.
$$
We need therefore to check that
$\widetilde G^{--}_{1/2}\Omega(r_1,r_2)=0$.

Recall that we chose $\k+\hvee=(k_1+n)(\cdot,\cdot)$ so that
$$
\widetilde G^{--}=\frac{1}{2\sqrt{k_1+n}}(G-\sqrt{-1}G_J).
$$
We can therefore use \eqref{gpgm} and write
\begin{equation}\label{easyG--}
\widetilde G^{--}=\frac{1}{2\sqrt{k_1+n}}(2\sum_i:u^i\overline{u}_i:-\sum_{i,j}:\overline{[u^i,u^j]}\overline{u}_i\overline{u}_j:),
\end{equation}
where $\{u_i\}$ is a basis of $\mathfrak u_J^+$ and $\{u^i\}$ is a basis of $\mathfrak u_J^-$ dual to the basis $\{u_i\}$. We can choose
$$
\{u_i\}=\{\half(h_\theta+d_1)\}\cup\{E_{1,j}\mid 2\le j\le n\}\cup\{E_{j,n}\mid 2\le j\le n\}
$$
and
$$
\{u^i\}=\{\half(h_\theta-d_1)\}\cup\{E_{j,1}\mid 2\le j\le n\}\cup\{E_{n,j}\mid 2\le j\le n\}.
$$
By \cite[(2.47)]{DK}, 
$:u^i\overline{u}_i:_{1/2}=\sum_{j\in\ZZ}u_{j}^i(\ov u_i)_{-j+1/2}$ so
$$
\begin{aligned}
&\sum_i:u^i\overline{u}_i:_{1/2}\Omega(r_1,r_2)=\sum_iu_{0}^iv_{\nu}\otimes (\ov u_i)_{1/2}(\ov{E}_{n,n-r_2})_{-1/2}\dots (\ov{E}_{n,n-1})_{-1/2}\vac\\
&=\sum_{j=n-r_2}^{n-1}(-1)^{n-r_2-j}(E_{n,j})_0v_{\nu}\otimes (\ov{E}_{n,n-r_2})_{-1/2}\cdots\widehat{(\ov{E}_{n,j})_{-1/2}}\cdots (\ov{E}_{n,n-1})_{-1/2}\vac.
\end{aligned}
$$
Since $E_{n,j}$ is a root vector for a root orthogonal to $\nu$, we see that $(E_{n,j})_0v_{\nu}=0$.

It remains to check that 
$$
\sum_{i,j}:\overline{[u^i,u^j]}\overline{u}_i\overline{u}_j:_{1/2}\Omega(r_1,r_2)=0.
$$
Applying \cite[(2.47)]{DK} twice, we find
$$
:\overline{[u^i,u^j]}\overline{u}_i\overline{u}_j:_{1/2}=\sum_{r+s+t=1/2}\overline{[u^i,u^j]}_r(\overline{u}_i)_s(\overline{u}_j)_t,
$$
so
$$
\begin{aligned}
&\sum_{i,j}:\overline{[u^i,u^j]}\overline{u}_i\overline{u}_j:_{1/2}\Omega(r_1,r_2)\\
&=\sum_{i,j}\sum_{s\le 1/2}\sum_{r\le1/2} v_{\nu}\otimes\overline{[u^i,u^j]}_{1/2-t-s}(\overline{u}_i)_s(\overline{u}_j)_t(\ov{E}_{n,n-r_2})_{-1/2}\dots (\ov{E}_{n,n-1})_{-1/2}\vac\\
&=\sum_{i,j} v_{\nu}\otimes\overline{[u^i,u^j]}_{1/2}(\overline{u}_i)_{1/2}(\overline{u}_j)_{-1/2}(\ov{E}_{n,n-r_2})_{-1/2}\dots (\ov{E}_{n,n-1})_{-1/2}\vac\\
&+\sum_{i,j} v_{\nu}\otimes\overline{[u^i,u^j]}_{1/2}(\overline{u}_i)_{-1/2}(\overline{u}_j)_{1/2}(\ov{E}_{n,n-r_2})_{-1/2}\dots (\ov{E}_{n,n-1})_{-1/2}\vac\\
&+\sum_{i,j} v_{\nu}\otimes\overline{[u^i,u^j]}_{-1/2}(\overline{u}_i)_{1/2}(\overline{u}_j)_{1/2}(\ov{E}_{n,n-r_2})_{-1/2}\dots (\ov{E}_{n,n-1})_{-1/2}\vac.
\end{aligned}
$$

Since $\mathfrak u_J^-$ is isotropic,
$$
\begin{aligned}
&\overline{[u^i,u^j]}_{1/2}(\overline{u}_i)_{1/2}(\overline{u}_j)_{-1/2}(\ov{E}_{n,n-r_2})_{-1/2}\dots (\ov{E}_{n,n-1})_{-1/2}\vac=0.
\end{aligned}
$$
and
$$
\begin{aligned}
&\overline{[u^i,u^j]}_{1/2}(\overline{u}_i)_{-1/2}(\overline{u}_j)_{1/2}(\ov{E}_{n,n-r_2})_{-1/2}\dots (\ov{E}_{n,n-1})_{-1/2}\vac\\
\end{aligned}
$$
so
{\small
$$
\begin{aligned}
&\sum_{i,j}:\overline{[u^i,u^j]}\overline{u}_i\overline{u}_j:_{1/2}\Omega(r_1,r_2)\\
&=\sum_{i,j} v_{\nu}\otimes\overline{[u^i,u^j]}_{-1/2}(\overline{u}_i)_{1/2}(\overline{u}_j)_{1/2}(\ov{E}_{n,n-r_2})_{-1/2}\dots (\ov{E}_{n,n-1})_{-1/2}\vac\\
&=\!\!\!\!\!\!\!\sum_{n-r_2\le i<j\le n-1}\!\!\!\!\!\!\!(-1)^{i+j} v_{\nu}\otimes\overline{[E_{n,i},E_{n,j}]}_{-\tfrac{1}{2}}(\ov{E}_{n,n-r_2})_{-\tfrac{1}{2}}\cdots \widehat{(\ov E_{n,i})_{-\tfrac{1}{2}}}\cdots\widehat{(\ov{E}_{n,j})_{-\tfrac{1}{2}}}\cdots(\ov{E}_{n,n-1})_{-\tfrac{1}{2}}\vac\\
&=0.
\end{aligned}
$$}

It remains only to check that $\hat L_m\Omega(r_1,r_2)=0$ when $m>0$. We checked above that $\Omega(r_1,r_2)$ is singular for $D(2,1;a)^\natural=\langle J^0,J^+,J^-\rangle\oplus\langle J^{'0},J^{'+},J^{'-}\rangle$ hence, if $x,y\in D(2,1;a)^\natural$ and $m>0$,
$:xy:_m\Omega(r_1,r_2)=0$, $x_m\Omega(r_1,r_2)=0$, and $(\partial x)_m\Omega(r_1,r_2)=0$. Thus, by \eqref{GlambdaGbis},
$$
\hat L_m\Omega(r_1,r_2)=[G^{++}_{1/2},G^{--}_{m-1/2}]\Omega(r_1,r_2)=0.
$$

\end{proof}

Recall from \cite{KMP1} that an irreducible highest weight module $L(\mu,\ell)$  for $W^{k}_{\min}(D(2,1;a))$ has parameters $(\mu,\ell)$ where $\mu$ is a weight of $D(2,1;a)^\natural\cong sl(2)\oplus sl(2)$ and 
$\ell$ is the minimal energy, i. e. the minimal eigenvalue of $L_0$. In our realization $D(2,1;a)^\natural=\langle J^0,J^+,J^-\rangle\oplus\langle J^{'0},J^{'+},J^{'-}\rangle$, so the Cartan subalgebra of $D(2,1;a)^\natural$ is $\h^\natural=\langle J^0,J^{'0}\rangle$. If 
$\mu\in (\h^\natural)^*$, we have $\mu=\tfrac{r_1}{2}\theta_1+\tfrac{r_2}{2}\theta_2$, with $\theta_1(J^0)=2, \theta_1(J^{'0})=0, \theta_2(J^{0})=0, \theta_2(J^{'0})=2$. The integrability conditions say that 
$r_i\in\ZZ,\,i=1,2,\, 0\le r_1\le k_1,\, 0\leq r_2\le n-2$.
 A  necessary condition for the unitarity of $L(\mu,\ell)$ is (cf. \cite[\S\ 12.4]{KMP1})
\begin{equation}\label{1b}
\begin{aligned}
\ell\ge A(k,\mu)&=\frac{2 (a+1) k (a r_2+r_1)-a
   (r_1-r_2)^2}{4 (a+1)^2 k}\\
   &=\frac{2r_2(k_1+1)+2 r_1 (n-1)+(r_1-r_2)^2}{4 (k_1+n)}.
   \end{aligned}
\end{equation}

\begin{lemma}\label{parameterextremal}Given $r_1,r_2\in \ZZ_+$ with $0\le r_1\le k_1$ and $0\le r_2\le n-2$, set $\mu=\tfrac{r_1}{2}\theta_1+\tfrac{r_2}{2}\theta_2$ and $\mathbb L(r_1,r_2)$ the $W^{k}_{\min}(D(2,1;a))$-module generated by $\Omega(r_1,r_2)$. 
Then, as $W^{k}_{\min}(D(2,1;a)$-module
$$\mathbb L(r_1,r_2)=L(\mu,A(k,\mu)).
$$
\end{lemma}
\begin{proof}  To simplify notation we set $\nu=r_1\omega_1$ throughout the proof.

 We first prove that $\Omega(r_1,r_2)$ is singular for $ V^\a(\aa)$. If $x\in \aa$, then, by \eqref{thetaTv} and \cite[(2.49)]{DK},
$$
[\Theta_\q(ad(x))_m,\ov q_{-1/2}]=\ov{ad(x)(q)}_{m-1/2}.
$$
In particular, if $2\le i\le n-2$, 
$$
\begin{aligned}
(E_{i,i+1}^\aa)_0\cdot\Omega(r_1,r_2)&=v_{\nu}\otimes \Theta_\q(E_{i,i+1})_0\left((\ov{E}_{n,n-r_2})_{-1/2}\cdots (\ov{E}_{n,n-1})_{-1/2}\vac\right)\\
&=\sum_jv_{\nu}\otimes \left((\ov{E}_{n,n-r_2})_{-1/2}\cdots\ov{[E_{i,i+1},E_{n,j} ]}_{-1/2}\cdots(\ov{E}_{n,n-1})_{-1/2}\vac\right)=0.
\end{aligned}
$$ 
Moreover
$$
\begin{aligned}
(E_{n-1,2}^\aa)_1\cdot\Omega(r_1,r_2)&=v_{\nu}\otimes \Theta_\q(E_{n-1,2})_1\left((\ov{E}_{n,n-r_2})_{-1/2}\cdots (\ov{E}_{n,n-1})_{-1/2}\vac\right)\\
&=v_{\nu}\otimes \left((\ov{E}_{n,n-r_2})_{-1/2}\cdots\ov{[E_{n-1,2},E_{n,n-1} ]}_{1/2}\vac\right)=0.
\end{aligned}
$$ 
Thus $\Omega(r_1,r_2)$ is singular for $V^\a(\aa)$. We now check that $\Omega(r_1,r_2)$ is singular for $V^{k}(\C\xi)$.

Recall that we chose $\k$ so that $\k+\hvee=(k_1+n)(\cdot,\cdot)$. This implies that 
\begin{equation}\label{specialxi}
\xi=\tfrac{\sqrt{(k_1+1)(n-1)}}{\sqrt{2}(k_1+n)}\left(d_1+\Theta_\q(ad(d_1))\right).
\end{equation}
If $s>0$ and $n=2$,  then $r_2=0$ and
$$
(d_1)_s\Omega(r_1,0)=((d_1)_sv_{\nu})\otimes \vac=0.
$$
If $s>0$ and $n>2$, then
$$
\begin{aligned}
&(d_1)_s\Omega(r_1,r_2)=(d_1)_sv_{\nu}\otimes \left((\ov{E}_{n,n-r_2})_{-1/2}\cdots (\ov{E}_{n,n-1})_{-1/2}\vac\right)\\
&+v_{\nu}\otimes \Theta_\q(ad(d_1))_s\left((\ov{E}_{n,n-r_2})_{-1/2}\cdots (\ov{E}_{n,n-1})_{-1/2}\vac\right)\\
&=\sum_jv_{\nu}\otimes \left((\ov{E}_{n,n-r_2})_{-1/2}\cdots\ov{[d_1,E_{n,j} ]}_{s-1/2}\cdots(\ov{E}_{n,n-1})_{-1/2}\vac\right)\\
&=-\sqrt{\frac{n}{2-n}}\sum_jv_{\nu}\otimes \left((\ov{E}_{n,n-r_2})_{-1/2}\cdots(\ov E_{n,j} )_{s-1/2}\cdots(\ov{E}_{n,n-1})_{-1/2}\vac\right)=0.
\end{aligned}
$$

Next we compute the action of $\widetilde L_0$. Recall (cf. \eqref{L}) that $\widetilde \Gamma(\widetilde L)=L=L^\g-L^\aa+L^{\bar\q}$, so
$$ 
\begin{aligned}&\widetilde L_0\Omega(r_1,r_2)= (L^\g_0-L^{\mathfrak a}_0+L^{\overline{\mathfrak q}}_0)\Omega(r_1,r_2)=\\&\left(\frac{(\nu,\nu+2\rho)}{2(k_1+n)}-
\frac{(\nu^\aa_{r_1,r_2},\nu^\aa_{r_1,r_2}+2\rho^\aa)}{2(k_1+n)}+\frac{r_2}{2}\right)\Omega(r_1,r_2),
\end{aligned}
$$
where $\nu^\aa_{r_1,r_2}$ is the weight of $\Omega(r_1,r_2)$ under the action of $V^\a(\aa)$.
We compute $\nu^\aa_{r_1,r_2}$ explicitly. If $h\in\h_\aa$ then, since $\nu(h)=0$, 
$$
\begin{aligned}
&h^\aa_0\Omega(r_1,r_2)=v_{\nu}\otimes \Theta_\q(ad(h))_0\left((\ov{E}_{n,n-r_2})_{-1/2}\cdots (\ov{E}_{n,n-1})_{-1/2}\vac\right)\\
&=\sum_jv_{\nu}\otimes \left((\ov{E}_{n,n-r_2})_{-1/2}\cdots\ov{[h,E_{n,j} ]}_{-1/2}\cdots(\ov{E}_{n,n-1})_{-1/2}\vac\right)\\
&=\omega^\aa_{n-2-r_2}(h)\Omega(r_1,r_2).\end{aligned}
$$
Here, if $1\le i \le n-3$,  $\omega^\aa_i$ denotes the fundamental weight for $\aa$ w.r.t. the set of simple roots $\{\a_2,\ldots,\a_{n-2}\}$ while $\omega^\aa_{n-2}=\omega^\aa_{0}=0$.
The final outcome is that
\begin{equation} \label{tildeLomega}
\begin{aligned}&\widetilde L_0\Omega(r_1,r_2)= \widetilde \ell_0\Omega(r_1,r_2)\end{aligned}
\end{equation}
with
$$\widetilde\ell_0=\begin{cases}\frac{n-1}{2(k_1+n)}\left(\frac{ r_1 (n+r_1)}{n}-\frac{ r_2 (n-r_2-2)}{n-2}\right)+\frac{r_2}{2}&\text{if $n>2$}\\
\frac{n-1}{2(k_1+n)}\frac{ r_1 (n+r_1)}{n}&\text{if $n=2$.}
\end{cases}$$

Set  $L^{\xi}=\tfrac{1}{2k} : \xi^2 :$  and
$$L^{\ov{ \mathfrak s}}=\frac{1}{2k}(: \partial \sigma^{--}\sigma^{++}:
           +:\partial \sigma^{++} \sigma^{--}:
           +:\partial\sigma^{-+}\sigma^{+-}:
           +:\partial \sigma^{+-}\sigma^{-+}:).
           $$
           Note that $L^{\xi}$ is the Virasoro vector for $V^{k}(\C\xi)$ and that $L^{\ov{ \mathfrak s}}$ is the Virasoro vector for $F(\ov{\mathfrak s})$.
Recall from \eqref{LL} that 
$$
\hat L=\widetilde L- L^{\xi}-L^{\ov{ \mathfrak s}},
$$
so we need to compute the action of  $L^{\xi}_0$ and of $L^{\ov{ \mathfrak s}}_0$  
 on $\Omega(r_1,r_2)$.
    
    Recall the orthogonal decomposition $\q=\mathfrak s\oplus \mathfrak f_1$. It is clear that, if  $q\in\mathfrak f_1$, then       
$ [L^{\ov{ \mathfrak s}}  {}_\l q]= 0$.
Since $\ov E_{n,i}\in\mathfrak f_1$ for $2\le i\le n-1$, we see that
$[L^{\ov{ \mathfrak s}}_0, (\ov E_{n,i})_{-1/2}]=0$, so 
$$
L^{\ov{ \mathfrak s}}_0\Omega(r_1,r_2)=0.
$$
  Since $\Omega(r_1,r_2)$ is singular for $V^{k}(\C\xi)$, 
  $$
  L^{\xi}_0\Omega(r_1,r_2)=\frac{1}{2k}\xi_0^2\Omega(r_1,r_2).
  $$
Since
$$
\begin{aligned}
&(d_1)_0\Omega(r_1,r_2)=\\&(d_1)_0v_{\nu}\otimes \left((\ov{E}_{n,n-r_2})_{-1/2}\cdots (\ov{E}_{n,n-1})_{-1/2}\vac\right)\\
&+v_{\nu}\otimes \Theta_\q(ad(d_1))_0\left((\ov{E}_{n,n-r_2})_{-1/2}\cdots (\ov{E}_{n,n-1})_{-1/2}\vac\right)\\
&=\nu(d_1)\Omega(r_1,r_2)+\sum_jv_{\nu}\otimes \left((\ov{E}_{n,n-r_2})_{-1/2}\cdots\ov{[d_1,E_{n,j} ]}_{-1/2}\cdots(\ov{E}_{n,n-1})_{-1/2}\vac\right)\\
&=\begin{cases}(\nu(d_1)+\sqrt{\frac{2-n}{n}}(\frac{n}{n-2}r_2)\Omega(r_1,r_2)&\text{if $n>2$},\\
\nu(d_1)&\text{if $n=2$,}
\end{cases}
\end{aligned}
$$
and
$\nu(d_1)=\sqrt{\frac{2-n}{n}}r_1$, the final result is that
$$
(d_1)_0\Omega(r_1,r_2)=\begin{cases}\sqrt{\frac{2-n}{n}}(r_1+\frac{n}{n-2}r_2)\Omega(r_1,r_2)&\text{if $n>2$},\\
0&\text{if $n=2$,}
\end{cases}
$$
so 
$$
\xi_0^2\Omega(r_1,r_2)=\begin{cases}\frac{ (k_1+1) (n-1)((n-2)r_1+nr_2)^2}{2(k_1+n)^2n(2-n)}\Omega(r_1,r_2)&\text{if $n>2$},\\
0&\text{if $n=2$},
\end{cases}
$$
and
$$
L^\xi_0\Omega(r_1,r_2)=\begin{cases}\frac{ ((n-2)r_1+nr_2)^2}{4n(k_1+n)(n-2)}\Omega(r_1,r_2)&\text{if $n>2$},\\
0&\text{if $n=2$}.
\end{cases}
$$
Thus, if $n>2$,
$
\hat L_0\Omega(r_1,r_2)=\ell_0\Omega(r_1,r_2)
$
with
$$
\ell_0=\widetilde\ell_0-\frac{ ((n-2)r_1+nr_2)^2}{4n(k_1+n)(n-2)}=\frac{2r_2(k_1+1)+2 r_1 (n-1)+(r_1-r_2)^2}{4 (k_1+n)}=A(k,\mu).
$$
If $n=2$ then, since $r_2=0$, 
$$
\ell_0=\widetilde\ell_0=\frac{1}{2(k_1+2)}\frac{ r_1 (2+r_1)}{2}=\frac{2 r_1+r_1^2}{4 (k_1+2)}=A(k,\mu)
$$
as well.

It remains to compute the action of $J^0_0$ and $J^{'0}_0$ on $\Omega(r_1,r_2)$. 
Recall that $\tilde J^{'0}_0=-A^-(h)_0=\sqrt{-1}\Theta_\q(J)_0$. Since, for $2\le j \le n-1$, $E_{n,j}\in\mathfrak f_1$, we have $J(E_{n,j})=[J_1,E_{n,j}]=-\sqrt{-1}E_{n,j}$, so
$$
\begin{aligned}
&\widetilde J^{'0}_0\Omega(r_1,r_2)=\sqrt{-1}v_{\nu}\otimes \Theta_\q(J)_0\left((\ov{E}_{n,n-r_2})_{-1/2}\cdots(\ov E_{n,j} )_{-1/2}\cdots(\ov{E}_{n,n-1})_{-1/2}\vac\right)\\
&=\sqrt{-1}\sum_jv_{\nu}\otimes 
\left((\ov{E}_{n,n-r_2})_{-1/2}\cdots\ov{[J_1,E_{n,j} ]}_{-1/2}\cdots(\ov{E}_{n,n-1})_{-1/2}\vac\right)\\
&=\sum_jv_{\nu}\otimes 
\left((\ov{E}_{n,n-r_2})_{-1/2}\cdots(\ov{E}_{n,j} )_{-1/2}\cdots(\ov{E}_{n,n-1})_{-1/2}\vac\right)=r_2\Omega(r_1,r_2).
\end{aligned}
$$
We already observed that, since $\s^{\pm\pm}\in\ov{\mathfrak s}$ and $(\mathfrak s,\mathfrak f_1)=0$, $[(\s^{\pm\pm})_\l\mathfrak f_1]=0$ hence
$$J^{'0}_0\Omega(r_1,r_2)=\widetilde J^{'0}_0\Omega(r_1,r_2)=r_2\Omega(r_1,r_2),
$$ 
as wished. By the same observation, $J^{0}_0\Omega(r_1,r_2)=\widetilde J^{0}_0\Omega(r_1,r_2)$ and, since $\widetilde J^0= A^+(h)=-\sqrt{-1}A^+(J)=-\sqrt{-1}(J_1+\Theta_{\mathfrak s}(ad(J_1)-J))$, we see that
$$
\begin{aligned}
&\widetilde J^{0}_0\Omega(r_1,r_2)=-\sqrt{-1}(J_1)_0v_{\nu}\otimes \left((\ov{E}_{n,n-r_2})_{-1/2}\cdots(\ov{E}_{n,n-1})_{-1/2}\vac\right)=r_1\Omega(r_1,r_2).
\end{aligned}
$$
\end{proof}

 Let $\overset{\circ}{\omega}:\g_1\to \g_1$ be the conjugate linear involution defined by $x\mapsto -\ov x^T$. (Here $\ov x$ means complex conjugation and $x^T$ is the transpose of $x$). If $n=2$, we extend $\overset{\circ}{\omega}$ to $\g$ by setting 
 $$
 \overset{\circ}{\omega}(\begin{pmatrix}1&0\\0&1\end{pmatrix})= \begin{pmatrix}1&0\\0&1\end{pmatrix}.
 $$ Since $\overset{\circ}{\omega}$ is a conjugate linear Lie algebra homomorphism of $\g$ and $\k(\overset{\circ}{\omega}(x),\overset{\circ}{\omega}(y))=\ov{\k(x,y)}$, the involution $\overset{\circ}{\omega}$ extends to define a conjugate linear involution of $V^\k(\g)$ still denoted by $\overset{\circ}{\omega}$.
 
 If $M$ is a positive energy module $M$ for $V^\k(\g)$,  a $\overset{\circ}{\omega}$-invariant form on $M$ is a Hermitian form $H(\cdot,\cdot)$ on $M$ such that, for all $a\in V^\k(\g)$, 
$$
H(m_1,Y_M(a, z)m_2)=H(Y_M(A(z)a, z^{-1})m_1,m_2),\quad m_1,m_2\in M,
$$
where  $A(z):V^\k(\g)\to V^\k(\g)((z))$ is defined by 
$$
A(z)v=e^{zL^\g_1} (- z^{-2})^{L^\g_0}\overset{\circ}{\omega}(a),\quad a\in V^\k(\g).
$$
It is well known that, if $r_1\in\ZZ_+$ with $r_1\le k_1$, then $L(r_1\omega_1)$ is a unitary representation of $V^\k(\g)$. This means that there is a positive definite $\overset{\circ}{\omega}$-invariant  Hermitian form $H_{r_1}(\cdot,\cdot)$ on $L(r_1\omega_1)$.  Moreover this form is unique provided that we normalize it by setting $H_{r_1}(v_{r_1\omega_1},v_{r_1\omega_1})=1$. See \cite{KMP} for more details.

Since $\q$ is $\overset{\circ}{\omega}$-stable and $(\overset{\circ}{\omega}(q_1),\overset{\circ}{\omega}(q_2))=\ov{(q_1,q_2)}$, we can extend  $\overset{\circ}{\omega}$ to define a conjugate linear involution $\omega_{\ov\q}$ on $F(\ov\q)$. Moreover, since $(\cdot,\cdot)$ is positive definite on the set 
$
\{q\in\q\mid \overset{\circ}{\omega}(q)=-q\}
$,
 there exists a  unique positive definite Hermitian  $\omega_{\ov\q}$-invariant form $H_{\ov\q}(\cdot,\cdot)$ on $F(\ov\q)$ such that $H_{\ov\q}(\vac,\vac)=1$ (see \cite[5.1]{KMP}).
This means that $F(\ov\q)$ is a unitary representation of $F(\ov\q)$.

\begin{lemma} Set $\widetilde W=\Gamma(W^{\min}_{k}(D(2,1;a)))\subset V^\k(\g)\otimes F(\ov\q)$.
Then
$$(\overset{\circ}{\omega}\otimes \omega_{\ov\q})(\widetilde W)=\widetilde W.
$$
\end{lemma}
\begin{proof}
Since $\q$ is $\overset{\circ}{\omega}$-stable, it is clear that $(\overset{\circ}{\omega}\otimes \omega_{\ov\q})(\widetilde G)=\widetilde G$. Let $\mathfrak s_\R$ be the real span of $d_1,J_1,K_1,J_1K_1$. Since $\overset{\circ}{\omega}$ fixes pointwise  $\mathfrak s_\R$ and since $J,K$ stabilize $\mathfrak s_\R$, we see that $\overset{\circ}{\omega}J=J\overset{\circ}{\omega}$ and $\overset{\circ}{\omega}K=K\overset{\circ}{\omega}$. This implies that $(\overset{\circ}{\omega}\otimes \omega_{\ov\q})(\widetilde G_J)=\widetilde G_J$ hence $(\overset{\circ}{\omega}\otimes \omega_{\ov\q})(\widetilde G^{++})=\widetilde G^{--}$. Similarly one checks that  $(\overset{\circ}{\omega}\otimes \omega_{\ov\q})(\widetilde G^{+-})=\widetilde G^{-+}$.
From the explicit expressions given in \eqref{campi} one sees that $(\overset{\circ}{\omega}\otimes \omega_{\ov\q})(\xi)=\xi$, $(\overset{\circ}{\omega}\otimes \omega_{\ov\q})(\s^{++})=a\s^{--}$, and $(\overset{\circ}{\omega}\otimes \omega_{\ov\q})(\s^{+-})=a\s^{-+}$.
The relations in \eqref{terze} now imply that the linear span of  $\{\widetilde J^s, \widetilde J^{'s}\mid s\in\{0,+,-\}\}$ is $\overset{\circ}{\omega}\otimes \omega_{\ov\q}$-stable. Finally, from \eqref{quarte}, one obtains
$$
 (\overset{\circ}{\omega}\otimes \omega_{\ov\q})([{\tilde{G}^{+ +}}{}_{\lambda}
         \widetilde{G}^{- -}] )=[{\tilde{G}^{--}}{}_{\lambda}
         \widetilde{G}^{++}]=
        \widetilde{L} - \frac{1}{{ 2}(a+1)} (\partial +2\lambda)
         (\widetilde{J}^0 +a\tilde{J}^{\prime 0}) +
         \frac{\lambda^2}{6}\tilde{c}\vac\, ,
$$
which in turn implies
$$
 (\overset{\circ}{\omega}\otimes \omega_{\ov\q})(\tilde{J}^0 +a\tilde{J}^{\prime 0})=-\tilde{J}^0 -a\tilde{J}^{\prime 0},
 $$
 so  $(\overset{\circ}{\omega}\otimes \omega_{\ov\q})(\widetilde L)=\widetilde L$.
 
 So far we have shown that $\widetilde\Gamma(V^{N=4}_{\tilde c,a})$ as well as $V^k(\C\xi)$ and $F(\C_{\bar1}^4)$ are $\overset{\circ}{\omega}\otimes \omega_{\ov\q}$-stable, thus \eqref{Jtildas}, \eqref{GtildaandG}, \eqref{LL} imply that $\widetilde W$ is also $\overset{\circ}{\omega}\otimes \omega_{\ov\q}$-stable.
\end{proof}

Set $\omega=(\overset{\circ}{\omega}\otimes \omega_{\ov\q})_{|\widetilde W}$. Since $\widetilde W$ is a non-collapsing quotient of $W^{k}_{\min}(D(2,1;a))$, by Proposition 7.2 of \cite{KMP1}, the involution $\omega$ is induced by a conjugate linear involution   of $W^{k}_{\min}(D(2,1;a))$ (that we still denote by $\omega$) such that $\Gamma\circ\omega=\omega\circ \Gamma$.

The main result of this section is the following.

\begin{prop}\label{extremalareunitary}Set $a=\frac{k_1+1}{n-1}$ and $k=-\frac{(k_1+1)(n-1)}{k_1+n}$. Fix $0\le r_1\le k_1$ and $0\le r_2\le n-2$. 
Then $(H_{r_1}\otimes H_{\ov\q})_{|\mathbb L(r_1,r_2)\times\mathbb L(r_1,r_2)}$  is a positive definite $\omega$-invariant form on $\mathbb L(r_1,r_2)$.
In particular, if $\mu=\tfrac{r_1}{2}\theta_1+\tfrac{r_2}{2}\theta_2$, the highest weight $W^{\min}_{k}(D(2,1;a))$-module $L(\mu,A(k,\mu))$ is unitary.
\end{prop}
\begin{proof}
We already observed that $(H_{r_1}\otimes H_{\ov\q})$ is a positive definite $\overset{\circ}{\omega}\otimes \omega_{\ov\q}$-invariant
Hermitian form on $L(r_1\omega_1)\otimes F(\ov\q)$, so, to complete the proof, we need only to check that $(H_{r_1}\otimes H_{\ov\q})_{|\mathbb L(r_1,r_2)\times\mathbb L(r_1,r_2)}$ is $\omega$-invariant.

Let $U$ be the vertex subalgebra of $V^\k(\g)\otimes F(\ov \q)$ generated  by 
$$
\aa'=\langle x+\Theta_\q(ad(x))\mid x\in\aa\oplus\C\xi\rangle
$$
and by $\ov{\mathfrak s}$.
Note that $\omega_U=L^\aa+L^\xi+L^{\ov{\mathfrak s}}$ is a Virasoro vector for $U$ and that the generating fields in $\aa'$ are primary of conformal weight 1, while the generators in $\ov{\mathfrak s}$ are primary of conformal weight $\half$. 

Since the generators in $\aa'$ are primary of conformal weight 1 and the generators in $\ov{\mathfrak s}$ are primary of conformal weight $\half$  also for $L^\g+L^{\ov \q}$, we see that 
$$
[\hat L_\la U]=[((L^\g+L^{\ov \q})-(L^\aa+L^\xi+L^{\ov{\mathfrak s}}))_\la U]=0.
$$
This implies that $[\hat L_1,L^\aa_1+L^\xi_1+L^{\ov{\mathfrak s}}_1]=0$, so 
$$
e^{z\hat L_1}=e^{z(L^\g_1+L^{\ov q}_1)}e^{-z(L^\aa_1+L^\xi_1+L^{\ov{\mathfrak s}}_1)},
$$
and, in order to verify the $\omega$-invariance of $(H_{r_1}\otimes H_{\ov\q})$, we need only to verify that
$$(L^\aa_1+L^\xi_1+L^{\ov{\mathfrak s}}_1)w=0,\ w\in \widetilde W.
$$
By \eqref{WN=3embedding}, $[(L^\xi+L^{\ov{\mathfrak s}})_\lambda \widetilde W]=0$, so we are reduced to checking that
$L^\aa_1w=0$ if $w\in \widetilde W$. For this, it is enough to show that $[(x+\Theta_\q(ad(x)))_\lambda \widetilde \Gamma(V^{N=4}_{\tilde c,a})]=0$ for $x\in \aa$.

Since $V^{N=4}_{\tilde c,a}$ is generated as a vertex algebra by $\widetilde G^{\pm\pm}$ it suffices to prove that $[(x+\Theta_\q(ad(x))_\la \widetilde G^{\pm\pm}]=0$ for all $x\in\aa$. Since $J(x+\Theta_\q(ad(x)))=K(x+\Theta_\q(ad(x)))=x+\Theta_\q(ad(x))$ , we need only to check that $[(x+\Theta_\q(ad(x))_\la \widetilde G]=0$ for all $x\in\aa$, but this follows from \eqref{GU}, since $x+\Theta_\q(ad(x))=\tilde x-\Theta_\aa(ad(x))$ and $\widetilde G=\frac{1}{2\sqrt{k_1+n}}(G_\g-G_\aa)$.
\end{proof}

\begin{remark}Fix  $a=\frac{p}{q}$ with $p,q\in\nat$ relatively prime and set $k=-r\frac{pq}{p+q}$, with $r\in\nat$. Let $n_1=rp-1$ and $n_2=rq+1$.
By Lemma \ref{parameterextremal}, if $r_1=n_1-1$ or $r_2=n_2-2$, then the representations $\mathbb L(r_1,r_2)$ realize the extremal representations of $W^{\min}_k(D(2,1;a))$  \cite{KMP1}. Thus Proposition \ref{extremalareunitary} confirms \cite[Conjecture 2]{KMP1} and completes the classification of unitary representations for the case of $W^{k}_{\min}(D(2,1;a))$ in the Neveu-Schwarz sector.
\end{remark}
\section{Application to unitarity  of massless representations in the Ramond sector}
As in the previous Section, we restrict here to $\g$ either $gl(2)$ or $sl(n)$ with $n>2$ and, if $n=2$, we set $k_0=k_1+2$ so that $\k+\hvee=(k_1+2)(\cdot,\cdot)$.

Let $\s_R$ be the automorphism of $W_k^{\min}(D(2,1;a))$ defined by $\s_R(b)=(-1)^{p(b)}b$. Here $p(b)$ is the parity of $b$. 
If $M$ is a $\s_R$-twisted module and $a\in W^k_{\min}(D(2,1;a))$, we write
$$
Y_M(b,z)=\sum_{n\in\ZZ} b_nz^{-n-\D_b}
$$
for the corresponding field. 
Similarly $\s_R$ defines an involution of $V^\k(\g)\otimes F(\ov\q)$ (cf \eqref{q}; below we use notation from Section 5). It is clear from the  definition of the map $\Gamma$ in \eqref{embedding} that $\Gamma\circ \s_R=\s_R\circ \Gamma$, so, if $M$ is a $\s_R$-twisted $V^\k(\g)\otimes F(\ov\q)$-module, it is a $\s_R$-twisted $W^k_{\min}(D(2,1;a))$-module as well.

We follow the construction of twisted modules for $W$-algebras given in \cite{KW}. This construction is detailed in \cite{KMPR} for the special case we are currently studying.
As explained in Example 2.2 of \cite{KMPR}, by picking a maximal isotropic subspace $U^+$ of $\q$, we can construct a  $\s_R$-twisted module $F(\ov\q,U^+)$ of $F(\ov\q)$ as the Clifford module $Cl(\q)/(Cl(\q))^+$ of a certain Clifford algebra $Cl(\q)$.
In the case at hand, with notation as in \eqref{varie}, we choose
$$
U^{+}=span(\{E_{1,j}\mid 2\le j\le n-1\}\cup\{E_{n,j}\mid 2\le j\le n-1\}\cup\{\phi_1(h_2),\phi_1(f)\}).
$$
Set also
$$
U^{-}=span(\{E_{j,1}\mid 2\le j\le n-1\}\cup\{E_{j,n}\mid 2\le j\le n-1\}\cup\{\phi_1(h_1),\phi_1(e)\}).
$$
Let $\mathbf 1$ be the projection  of $1\in Cl(\q)$ onto the Clifford module  $F(\ov\q,U^{+})$.
\begin{lemma}\label{theta1} With notation as in Proposition 	\ref{Apamdef}, we have
$$
A^-(f)_0\cdot\mathbf 1=0,\quad A^+(e)_0\cdot \mathbf 1=0,\ 
A^-(h)_0\cdot\mathbf 1=-\vacc,\quad A^+(h)_0\cdot \mathbf 1=0,
$$
and
$$
\Theta_\q(ad(E_{i+1,i}))_0\vacc=0.
$$
\end{lemma}
\begin{proof}
If $a,b\in\q$, by \cite[(2.47)]{DK}, 
\begin{equation}\label{:ab:}
:\ov a\ov b:_0=\half(a,b)+\sum_{j\ge 0}\left(\ov a_{-1-j}\ov b_{1+j}-\ov b_{-j}\ov a_j\right).
\end{equation} 
Recall  that $L_{f}(x)=\phi_1(f\phi_{-1}(x))$ for all $x\in\mathfrak s$ (cf. \eqref{s}). Moreover 
$$
\phi_1^{-1}(U^{+}\cap\mathfrak s)=\left\{\begin{pmatrix} 0  & 0 \\a & b\end{pmatrix}\right\},\quad \phi_1^{-1}(U^{-}\cap\mathfrak s)=\left\{\begin{pmatrix} c  & d \\0 & 0\end{pmatrix}\right\}.
$$
It follows that $
L_{f}(U^{+}\cap\mathfrak s)= 0$ and $
L_{f}(U^{-}\cap\mathfrak s)\subset U^{+}$.
 Note also that 
$$
ad(x_{-\theta})(U^{+}\cap \mathfrak f_1)\subset U^{+}\cap \mathfrak f_1, \quad ad(x_{-\theta})(U^{-}\cap \mathfrak f_1)\subset U^{-}\cap \mathfrak f_1.
$$
Since $L_{f}$, $R_e$, and $ad(x_{-\theta})$ are nilpotent, 
$$Tr_{\mathfrak s}(L_{f})=Tr_{\mathfrak s}(R_{e})=Tr_{\mathfrak f_1}(ad(x_{-\theta}))=Tr_{U^{-}\cap\mathfrak f_1}(ad(x_{-\theta}))=0.
$$
Recall from \eqref{A-} that $A^-(f)=\Theta_{\mathfrak s}(L_f)+\Theta_{\mathfrak f_1}(ad(x_{-\theta}))$. Choose a basis $\{u_i\}$ of $U^{+}$ and let $\{u^i\}$ be the corresponding dual basis of $U^{-}$. By the above observations, we find
$$
\begin{aligned}
&A^-(f)_0\mathbf 1=
\left(\Theta_{\mathfrak s}(L_{f})_0+ \Theta_{\mathfrak f_1}(ad(x_{-\theta}))_0\right)\cdot\mathbf 1\\
&=\tfrac{1}{4}Tr_{\mathfrak s}(L_{f})+\tfrac{1}{4}Tr_{\mathfrak f_1}(ad(x_{-\theta}))-\half\sum_{q_i\in\mathfrak s}\ov q^i_0(\ov{L_{f}(q_i))}_0\vacc-\half\sum_{q_i\in\mathfrak f_1}\ov q^i_0\ov{ad(x_{-\theta})(q_i)}_0\vacc\\
&=-\half\sum_{u_i\in U^{+}\cap\mathfrak f_1}(\ov u_i)_0\ov{ad(x_{-\theta})(u^i)}_0\vacc=-\tfrac{1}{2}Tr_{U^{-}\cap\mathfrak f_1}(ad(x_{-\theta}))\vacc=0.
\end{aligned}
$$
Recall from \eqref{Re} that $R_{e}(x)=-\phi_1(\phi_{-1}(x)e)$ for all $x\in\mathfrak s$. We choose $\{\phi_1(h_1),\phi_1(e),\phi_1(f),$ $\phi_1(h_2)\}$ as a basis of $\mathfrak s$ so that $\{-\phi_1(h_2),\phi_1(f),\phi_1(e),-\phi_1(h_1)\}$ is its dual basis. It follows that
$$
A^+(e)\vacc=\Theta_{\mathfrak s}(R_{e})_0\mathbf 1=-\half\sum_{u_i\in U^{+}\cap\mathfrak s}(\ov u_i)_0\ov{R_e(u^i)}_0\vacc=-\half Tr_{U^{-}}(R_e)\vacc=0.
$$
Since $A^-(h)=\Theta_{\mathfrak s}(L_h)+\Theta_{\mathfrak f_1}(ad(h_\theta))$, we have
$$
\begin{aligned}
&A^-(h)_0\mathbf 1=
\left(\Theta_{\mathfrak s}(L_{h})_0+ \Theta_{\mathfrak f_1}(ad(h_{\theta}))_0\right)\cdot\mathbf 1\\
&=\tfrac{1}{4}Tr_{\mathfrak s}(L_{h})+\tfrac{1}{4}Tr_{\mathfrak f_1}(ad(h_{\theta}))-\half\sum_{q_i\in\mathfrak s}\ov q^i_0(\ov{L_{h}(q_i))}_0\vacc-\half\sum_{q_i\in\mathfrak f_1}\ov q^i_0\ov{ad(h_{\theta})(q_i)}_0\vacc.
\end{aligned}
$$
Since $L_hx=x$ for $x\in \mathfrak s\cap U^{-}$, $ad(h)(E_{j,1})=-E_{j,1}$, and $ad(h)(E_{j,n})=E_{j,n}$,
$$
\begin{aligned}
&A^-(h)_0\mathbf 1=-\half\sum_{u_i\in\mathfrak s\cap U^{+} } (\ov u_i)_0\ov u^i_0\vacc+\half\sum_{j=2}^{n-1}(\ov E_{1,j})_0(\ov E_{j,1})_0\vacc-\half\sum_{j=2}^{n-1}(\ov E_{n,j})_0(\ov E_{j,n})_0\vacc=-\vacc.
\end{aligned}
$$

Next we compute
$$
\begin{aligned}
&A^+(h)_0\mathbf 1=
\Theta_{\mathfrak s}(R_{h})_0\cdot\mathbf 1=\tfrac{1}{4}Tr_{\mathfrak s}(R_{h})-\half\sum_{q_i\in\mathfrak s}\ov q^i_0(\ov{R_{h}(q_i))}_0\vacc=-\half\sum_{u_i\in U^{+}\cap\mathfrak s}(\ov u_i)_0(\ov{R_{h}(u^i))}_0\vacc\\
&=-\half Tr_{U^{-}}(R_h)\vacc=0.
\end{aligned}
$$
It remains to compute
$$
\Theta_\q(ad(E_{i+1,i}))_0\vacc=\tfrac{1}{4}Tr_\q(ad(E_{i+1,i})\vacc-\half\sum_i\ov q^i_0\ov{ad(E_{i+1,i})(q_i)}_0\vacc.
$$
Since $U^{\pm,-+}$ are stable under the action of $\aa$, we can rewrite the above formula as
$$
\Theta_\q(ad(E_{i+1,i}))_0\vacc=\tfrac{1}{4}Tr_\q(ad(E_{i+1,i}))-\half Tr_{U^{-}}(ad(E_{i+1,i})\vacc,
$$
and, since $ad(E_{i+1,i})$ is nilpotent, we conclude that $\Theta_\q(ad(E_{i+1,i}))_0\vacc=0$.
\end{proof}

\begin{definition}
We say that a vector in a $\s_R$-twisted $W^{k}_{\min}(D(2,1;a))$-module $M$ is singular if (4.14), (4.15), and (4.16) of \cite{KMPR} hold. As explained in \cite[Section 6]{KMPR}, there are two types of singular vectors depending on the action of $G^{+-}_0$ and $G^{-+}_0$: a vector $v$
is called {\sl $+-$-singular} (resp.  {\sl $-+$-singular}) if 
\begin{equation}\label{singularn>0}
X_n v=0,\ n\in\nat,\ X=J^s,J^{'s}, G^{t},  L,\quad s\in\{+,0,-\},\ t\in\{++,+-,-+,--\}
\end{equation}
and
\begin{equation}\label{singularn=0}
X_0 v=0,\ X=J^{+},  J^{'+}, G^{++}, G^{+-} \text{ (resp. $G^{-+}$)}.
\end{equation}

If $s\in\{+-,-+\}$, a $(s)$--highest weight module in the Ramond sector is a  $\s^R$-twisted $W^k_{\min}(D(2,1;a))$-module that is generated by a $(s)$-singular vector. Given $(\mu,\ell)\in (\h^\natural)^*\times \C$, we let $L^{s}(\mu,\ell)$ denote the irreducible $(s)$--highest weight module of highest weight $(\mu,\ell)$.
\end{definition}

In the rest of this Section, we will give an explicit construction of extremal $\s_R$-twisted irreducible $(-+)$--highest weight modules over $W^{k}_{\min}(D(2,1;a))$. For a brief discussion of the construction of $(+-)$--highest weight modules,  see Remark \ref{finalremark} below.
\begin{lemma}\label{singularR}
Fix $0\le r_1\le k_1$. Let  $\Omega(r_1,r_2)$ be  the vector in  $L(r_1\omega_1)\otimes F(\ov \q,U^{+})$ defined by
$$
\begin{cases}
\Omega(r_1,0)=v_{r_1\omega_1}\otimes\vacc,\\
\Omega(r_1,r_2)=v_{r_1\omega_1}\otimes (\ov{E}_{n-r_2,1})_{0}\dots (\ov{E}_{n-1,1})_{0}\vacc,\ r_2=1,\ldots,n-2.
\end{cases}
$$
Then the vectors $\Omega(r_1,r_2)$ are $(-+)$-singular vectors for $W^{k}_{\min}(D(2,1;a))$.
\end{lemma}
\begin{proof}Set $\nu=r_1\omega_1$. Since $\s^{++}_0\mathbf 1=0$, in order to check that 
\begin{equation}\label{jsingularR}
J^{'+}_0\Omega(r_1,r_2)=J^{'-}_1\Omega(r_1,r_2)=J^{+}_0\Omega(r_1,r_2)=J^{-}_1\Omega(r_1,r_2)=0,
\end{equation} 
we can argue as in the proof of Lemma \ref{singularNS}, so that
it is enough to check that 
\begin{equation}\label{jtildesingularR}
\widetilde J^{'+}_0\Omega(r_1,r_2)=\widetilde J^{'-}_1\Omega(r_1,r_2)=\widetilde J^{+}_0\Omega(r_1,r_2)=\widetilde J^{-}_1\Omega(r_1,r_2)=0.
\end{equation}

By \eqref{thetaTv} and \cite[(2.49)]{DK},
\begin{equation}\label{thetaramond}
[\Theta_\q(ad(x))_m,\ov q_{0}]=\ov{ad(x)(q)}_{m}
\end{equation}
so,   using  Lemma \ref{theta1},
$$
\begin{aligned}
A^-(e)_1\cdot\Omega(r_1,r_2)&=\sum_jv_{\nu}\otimes \left((\ov{E}_{n-r_2,1})_{0}\cdots\ov{[E_{1,n},E_{j,1} ]}_{1}\cdots(\ov{E}_{n-1,1})_{0}\vac\right)=0,
\end{aligned}
$$ 
$$
\begin{aligned}
A^-(f)_0\cdot\Omega(r_1,r_2)&=\sum_jv_{\nu}\otimes \left((\ov{E}_{n-r_2,1})_{0}\cdots\ov{[E_{n,1},E_{j,1} ]}_{0}\cdots(\ov{E}_{n-1,1})_{0}\vac\right)\\
&+v_{\nu}\otimes (\ov{E}_{n-r_2,1})_{0}\cdots(\ov{E}_{n-1,1})_{0}A^-(f)_0\mathbf 1=0,
\end{aligned}
$$ 
$$
\begin{aligned}
A^+(e)_0\cdot\Omega(r_1,r_2)&=(x_\theta)_0v_{\nu}\otimes\left((\ov{E}_{n-r_2,1})_{0}\cdots (\ov{E}_{n-1,1})_{0}\mathbf 1\right)\\
&+v_{\nu}\otimes A^+(e)_0(\ov{E}_{n-r_2,1})_{0}\cdots (\ov{E}_{n-1,1})_{0}\mathbf 1\\
&=v_{\nu}\otimes (\ov{E}_{n-r_2,1})_{0}\cdots (\ov{E}_{n-1,1})_{0}A^+(e)_0\mathbf 1=0,
\end{aligned}
$$
and
$$
\begin{aligned}
A^+(f)_1\cdot\Omega(r_1,r_2)&=(x_{-\theta})_1v_{\nu}\otimes\left((\ov{E}_{n-r_2,1})_{0}\cdots (\ov{E}_{n-1,1})_{0}\mathbf 1\right)\\
&+v_{\nu}\otimes \Theta_{\mathfrak s}(R_{f})_1(\ov{E}_{n-r_2,1})_{0}\cdots (\ov{E}_{n-1,1})_{0}\mathbf 1=0,
\end{aligned}
$$
and these equations amount to \eqref{jtildesingularR} which in turn implies \eqref{jsingularR}.

We now check, that $G^{-+}_0\Omega(r_1,r_2)=0$.
As in the proof of Lemma \ref{singularNS}, we first check that 
$G^{-+}_0\Omega(r_1,r_2,)=\widetilde G^{-+}_0\Omega(r_1,r_2)$.

Since $[J^t_\la\s^{\pm\pm}]=[J^{'t}_\la\s^{\pm\pm}]=0$, \cite[(2.47)]{DK} gives that
$ :J^t\s^{\pm\pm}:_0=\sum_{j\in\ZZ}J^t_{-j}\s^{\pm\pm}_j=\sum_{j\in\ZZ}\s^{\pm\pm}_{-j}J^t_{j}$ and $ :J^{'t}\s^{\pm\pm}:_0=\sum_{j\in\ZZ}J^{'t}_{-j}\s^{\pm\pm}_j=\sum_{j\in\ZZ}\s^{\pm\pm}_{-j}J^{'t}_{j}$.
This implies that
$$
:J^t\s^{\pm\pm}:_0\Omega(r_1,r_2)=(-1)^{r_2}J^{t}_{0}(v_{\nu}\otimes  (\ov{E}_{n-r_2,1})_{0}\dots (\ov{E}_{n-1,1})_{0}\s^{\pm\pm}_{0}\vac.
$$
In particular, since $\s^{++}_0\mathbf 1= \s^{-+}_0\mathbf 1=0$, we deduce that $:J^t\s^{++}:_0\Omega(r_1,r_2)=:J^t\s^{-+}:_0\Omega(r_1,r_2)=0$.
The same argument shows that $:\xi \sigma^{-+}:_0\Omega(r_1,r_2)=
       :\sigma^{-+}\sigma^{--}\sigma^{++}:_0\Omega(r_1,r_2)=0$. Finally, since $:J^t\s^{\pm\pm}:_0\Omega(r_1,r_2)=s^{\pm\pm}_0J^t_0\Omega(r_1,r_2)$, we see from the first part of the proof that $:J^+\s^{--}:_0\Omega(r_1,r_2)=:J^{'+}\s^{--}:_0\Omega(r_1,r_2)=0$.
       The outcome of this discussion is that $G^{-+}_0\Omega(r_1,r_2)=\widetilde G^{-+}_0\Omega(r_1,r_2)=0$ as desired.
       
Recall that  $\widetilde G^{-+}=K(\widetilde G^{--})$ and that, by our choice of $k_0$ when $n=2$,  $\widetilde G^{--}$ is given by \eqref{easyG--}. Thus,
to compute $K(\widetilde G^{--})_0\Omega(r_1,r_2)$, it is enough to  choose a basis of $\mathfrak u_J^+$ as in Lemma \ref{singularNS} and compute
$$
(2\sum_i:u^i\overline{K(u_i)}:_0-\sum_{i,j}:\overline{K([u^i,u^j])}\overline{K(u_i)}\overline{K(u_j)}:_0)\Omega(r_1,r_2).
$$
By \cite[(2.47)]{DK}, 
$:u^i\overline{K(u_i)}:_{0}=\sum_{j\in\ZZ}u_{j}^i(\ov{K(u_i)})_{-j}$, so
$$
\begin{aligned}
&\sum_i:u^i\overline{K(u_i)}:_{0}\Omega(r_1,r_2)=\sum_iu_{0}^iv_{\nu}\otimes (\overline{K(u_i)}_{0}(\ov{E}_{n-r_2,1})_{0}\dots (\ov{E}_{n-1,1})_{0}\mathbf 1.
\end{aligned}
$$
We note that $K(u_i)\in\mathfrak u_J^-$ and, since $\mathfrak u_J^-$ is isotropic, we find 
$$
\begin{aligned}
&\sum_i:u^i\overline{K(u_i)}:_{0}\Omega(r_1,r_2)=(-1)^{r_2}\sum_iu_{0}^iv_{\nu}\otimes (\ov{E}_{n-r_2,1})_{0}\dots (\ov{E}_{n-1,1})_{0}\overline{K(u_i)}_{0}\mathbf 1\\
&=-(-1)^{r_2}\sum_{j=2}^{n-1}(E_{j,1})_{0}v_{\nu}\otimes (\ov{E}_{n-r_2,1})_{0}\dots (\ov{E}_{n-1,1})_{0} (\ov{E}_{n,j})_{0}\mathbf 1\\
&-(-1)^{r_2}\sum_{j=2}^{n-1}(E_{n,j})_{0}v_{\nu}\otimes (\ov{E}_{n-r_2,1})_{0}\dots (\ov{E}_{n-1,1})_{0}(\overline{E}_{j,1})_{0}\mathbf 1\\
&+(-1)^{r_2}\phi_1(h_2)_{0}v_{\nu}\otimes (\ov{E}_{n-r_2,1})_{0}\dots (\ov{E}_{n-1,1})_{0} (\ov{E}_{n,1})_{0}\mathbf 1\\
&+(-1)^{r_2}(E_{n,1})_{0}v_{\nu}\otimes (\ov{E}_{n-r_2,1})_{0}\dots (\ov{E}_{n-1,1})_{0}\overline{\phi_1(h_2)}_{0}\mathbf 1.
\end{aligned}
$$
As observed in the proof of Lemma \ref{singularNS}, $(E_{n,j})_{0}v_{\nu}=0$. By our choice of $U^{+}$, we have $(\ov{E}_{n,j})_{0}\mathbf 1=0$ for $1\le j\le n-1$ and $\overline{\phi_1(h_2)}_{0}\mathbf 1=0$. The outcome is that $\sum_i:u^i\overline{K(u_i)}:_{0}\Omega(r_1,r_2)=0$.

It remains to check that 
$$
\sum_{i,j}:\overline{K([u^i,u^j])}\overline{K(u_i)}\overline{K(u_j)}:_0
\Omega(r_1,r_2)=0.
$$
Applying \cite[(2.47)]{DK} twice, we find
$$
\begin{aligned}
&:\overline{K([u^i,u^j])}\overline{K(u_i)}\overline{K(u_j)}:_0=\half ([u^i,u^j],u_i)\ov{K( u_j)}_0-\half ([u^i,u^j],u_j)\ov{K( u_i)}_0\\
&+\sum_{r,t\ge0}\left(\overline{K([u^i,u^j])}_{-1-r}\overline{K(u_i)}_{-1-t}\overline{K(u_j)}_{r+2+t}-\overline{K([u^i,u^j])}_{-1-r}\overline{K(u_j)}_{r+1-t}\overline{K(u_i)}_{t}\right)\\
&+\sum_{r,t\ge0}\left(\overline{K(u_i)}_{-t-1}\overline{K(u_j)}_{-r+t+1}\overline{K([u^i,u^j])}_{r}-\overline{K(u_j)}_{-r-t}\overline{K(u_i)}_{t}\overline{K([u^i,u^j])}_{r}\right)
\end{aligned}
$$
so
$$
\begin{aligned}
&\sum_{i,j}:\overline{K([u^i,u^j])}\overline{K(u_i)}\overline{K(u_j)}:_0
\Omega(r_1,r_2)=\sum_i\ov{K([u_i,u^i]_{\mathfrak u_J^+})}_0\Omega(r_1,r_2)\\
&-\sum_{i,j}\overline{K(u_j)}_{0}\overline{K(u_i)}_{0}\overline{K([u^i,u^j])}_{0}\Omega(r_1,r_2).
\end{aligned}
$$
Note that 
$$
\begin{aligned}
&\sum_{i,j}\overline{K(u_j)}_{0}\overline{K(u_i)}_{0}\overline{K([u^i,u^j])}_0=\sum_{i,j}\overline{K([u^i,u^j])}_0\overline{K(u_j)}_{0}\overline{K(u_i)}_{0}+2\sum_i\overline{K([u_i,u^i]_{\mathfrak u_J^+})}_{0},
\end{aligned}
$$
hence
$$
\begin{aligned}
&\sum_{i,j}:\overline{K([u^i,u^j])}\overline{K(u_i)}\overline{K(u_j)}:_0
\Omega(r_1,r_2)=3\sum_i\ov{K([u_i,u^i]_{\mathfrak u_J^+})}_0\Omega(r_1,r_2)\\
&+\sum_{i,j}\overline{K([u^i,u^j])}_0\overline{K(u_j)}_{0}\overline{K(u_i)}_{0}\Omega(r_1,r_2).
\end{aligned}
$$
Since $\sum_iK([u_i,u^i]_{\mathfrak u_J^+})=(n-1)K(\phi_1(h_1))=(n-1)\phi_1(f)$, it follows that 
$$\sum_i\ov{K([u_i,u^i]_{\mathfrak u_J^+})}_0\Omega(r_1,r_2)=0.
$$
Finally
$$
\begin{aligned}
&\sum_{i,j}:\overline{K([u^i,u^j])}\overline{K(u_i)}\overline{K(u_j)}:_0
\Omega(r_1,r_2)\\
&=2\!\!\!\!\!\!\!\!\!\!\sum_{n-r_2\le t<s\le n-1}\!\!\!\!\!\!(-1)^{s+r+1}v_\nu\otimes\overline{K([E_{t,1},E_{s,1}])}_0(\ov{E}_{n-r_2,1})_{0}\cdots \widehat{(\ov E_{t,1})_{0}}\cdots\widehat{(\ov{E}_{s,1})_{0}}\cdots(\ov{E}_{n-1,1})_{0}\mathbf 1=0.
\end{aligned}
$$

It remains only to check that $\hat L_m\Omega(r_1,r_2)=0$ when $m>0$. We checked above that $\Omega(r_1,r_2)$ is singular for $D(2,1;a)^\natural=\langle J^0,J^+,J^-\rangle\oplus\langle J^{'0},J^{'+},J^{'-}\rangle$ hence, if $x,y\in D(2,1;a)^\natural$ and $m>0$,
$:xy:_m\Omega(r_1,r_2)$ as well as $x_m\Omega(r_1,r_2)$, and $(\partial x)_m\Omega(r_1,r_2)$ are all zero. Thus, by \eqref{GlambdaGbis},
$$
\hat L_m\Omega(r_1,r_2)=[G^{++}_{0},G^{--}_{m}]\Omega(r_1,r_2)=0
$$
as wished.
\end{proof}

\begin{lemma}\label{parameterextremalRamond}
Given $r_1,r_2\in \ZZ_+$ with $0\le r_1\le k_1$ and $0\le r_2\le n-2$, set $\mu=\tfrac{r_1}{2}\theta_1+\tfrac{r_2}{2}\theta_2$ and $\mathbb L^{-+}(r_1,r_2)$ the $W^{k}_{\min}(D(2,1;a))$-module generated by $\Omega(r_1,r_2)$. 
Then, as $W^{k}_{\min}(D(2,1;a)$-module
$$\mathbb L^{-+}(r_1,r_2)=L^{-+}(\mu,A(k,\mu))
$$
with
$$
\begin{aligned}
A(k,\mu)&=-\frac{(a+1)^2 k (k+1)+a\left((r_1+r_2+1)^2\right)}{4 (a+1)^2
   k}\\
   &=\frac{k_1 (n-2)+r_1^2+2 \left(r_2+1\right) r_1+r_2 \left(r_2+2\right)}{4
   \left(k_1+n\right)}.
   \end{aligned}
   $$
\end{lemma}
\begin{proof}
Set $\nu=r_1\omega_1$. 
Using \eqref{thetaramond}, arguing as in the proof of Lemma \ref{parameterextremal},
$$
\begin{aligned}
&\widetilde J^{'0}_0\Omega(r_1,r_2)=-A^-(h)\Omega(r_1,r_2)=\sqrt{-1}v_{\nu}\otimes \Theta_\q(J)_0\left((\ov{E}_{n-r_2,1})_{0}\dots (\ov{E}_{n-1,1})_{0}\vacc\right)\\
&=\sqrt{-1}\sum_jv_{\nu}\otimes 
\left((\ov{E}_{n-r_2,1})_{0}\cdots\ov{[J_1,E_{j,1} ]}_{0}\cdots(\ov{E}_{n-1,1})_{0}\vacc\right)\\
&-v_{\nu}\otimes 
\left((\ov{E}_{n-r_2,1})_{0}\cdots
(\ov{E}_{n-1,1})_{0}A^-(h)_0\vacc \right),
\end{aligned}
$$
so, by Lemma \ref{theta1},
$$
\begin{aligned}
&&\widetilde J^{'0}_0\Omega(r_1,r_2)=\sum_jv_{\nu}\otimes 
\left((\ov{E}_{n-r_2,1})_{0}\cdots(\ov{E}_{j,1} )_{0}\cdots(\ov{E}_{n-1,1})_{0}\vacc\right)=(r_2+1)\Omega(r_1,r_2).
\end{aligned}
$$
Since $E_{j,1}\in\mathfrak f_1$ for $2\le j\le n-1$,
$$
\begin{aligned}
&\widetilde J^{0}_0\Omega(r_1,r_2)=A^+(h)_0\Omega(r_1,r_2)=-\sqrt{-1}(J_1)_0v_{\nu}\otimes \left((\ov{E}_{n-r_2,1})_{0}\cdots(\ov{E}_{n-1,1})_{0}\vacc\right)\\
&+v_{\nu}\otimes \left((\ov{E}_{n-r_2,1})_{0}\cdots(\ov{E}_{n-1,1})_{0}A^+(h)_0\vacc\right)=r_1\Omega(r_1,r_2).
\end{aligned}
$$
On the other hand, by \eqref{:ab:},
$$
:\s^{--}\s^{++}:_0\vacc=\half k\vacc-\s^{++}_0\s^{--}_0\vacc=-\half k\vacc,\ :\s^{+-}\s^{-+}:_0\vacc=\half k\vacc-\s^{-+}_0\s^{+-}_0\vacc=-\half k\vacc.
$$
Since $E_{j,1}\in\mathfrak f_1$ when $2\le j\le n-1$ and $\s^{\pm\pm}\in\ov{\mathfrak s}$, we see that $[:\s^{\pm\pm}\s^{\pm\pm}:_0,(\ov E_{j,1})_0]=0$, thus, by \eqref{Jtildas},
$$
\begin{aligned}
&J^{'0}_0\Omega(r_1,r_2)= \left(\widetilde{J}^{\prime 0}_0  +\tfrac{1}{k}: \sigma^{--}\sigma^{++}:_0
           + \tfrac{1}{k}:\sigma^{+-} \sigma^{-+}:_0\right)\Omega(r_1,r_2)\\
           &= (r_2+1)\Omega(r_1,r_2)+\tfrac{1}{k}v_\nu\otimes  (\ov{E}_{n-r_2,1})_{0}\cdots(\ov{E}_{n-1,1})_{0}\left(: \sigma^{--}\sigma^{++}:_0+:\sigma^{+-} \sigma^{-+}:_0\right)\vacc\\
           &=r_2\Omega(r_1,r_2),
           \end{aligned}
$$
as wished. Similarly
$$
\begin{aligned}
&J^{0}_0\Omega(r_1,r_2)= \left(\widetilde{J}^{0}_0  +\tfrac{1}{k}: \sigma^{--}\sigma^{++}:_0
           - \tfrac{1}{k}:\sigma^{+-} \sigma^{-+}:_0\right)\Omega(r_1,r_2)\\
           &= r_1\Omega(r_1,r_2)+\tfrac{1}{k}v_\nu\otimes  (\ov{E}_{n-r_2,1})_{0}\cdots(\ov{E}_{n-1,1})_{0}\left(: \sigma^{--}\sigma^{++}:_0-:\sigma^{+-} \sigma^{-+}:_0\right)\vacc\\
           &=r_1\Omega(r_1,r_2).
           \end{aligned}
$$
It remains to check the action of $\hat L_0$. As in Lemma \ref{parameterextremal}, we start by checking that $\Omega(r_1,r_2)$ is singular for $V^\a(\aa)$ with respect to the Borel subalgebra spanned by $\h\cap\aa$ and $E_{i+1,i}$ with $2\le i\le n-2$:
$$
\begin{aligned}
(E_{i+1,i}^\aa)_0\cdot\Omega(r_1,r_2)&=v_{\nu}\otimes \Theta_\q(ad(E_{i+1,i}))_0\left((\ov{E}_{n-r_2,1})_{0}\cdots (\ov{E}_{n-1,1})_{0}\vacc\right)\\
&=\sum_jv_{\nu}\otimes \left((\ov{E}_{n-r_2,1})_{0}\cdots\ov{[E_{i+1,i},E_{j,1} ]}_{0}\cdots(\ov{E}_{n-1,1})_{0}\vacc\right)\\
&+v_{\nu}\otimes (\ov{E}_{n-r_2,1})_{0}\cdots(\ov{E}_{n-1,1})_{0}\Theta_\q(ad(E_{i+1,i}))_0\vacc.
\end{aligned}
$$

Moreover, since $\aa=0$ if $n=2$, we can assume $n>2$ and compute
$$
\begin{aligned}
(E_{2,n-1}^\aa)_1\cdot\Omega(r_1,r_2)&=v_{\nu}\otimes \Theta_\q(ad(E_{2,n-1}))_1\left((\ov{E}_{n-r_2,1})_{0}\cdots (\ov{E}_{n-1,1})_{0}\vacc\right)\\
&=v_{\nu}\otimes \left((\ov{E}_{n-r_2,1})_{0}\cdots\ov{[E_{2,n-1},E_{n-1,1} ]}_{1}\vacc\right)\\
&+v_{\nu}\otimes \left((\ov{E}_{n-r_2,1})_{0}\cdots(\ov E_{n-1,1})_{0}\Theta_\q(ad(E_{2,n-1}))_1\vacc\right)=0.
\end{aligned}
$$ 
We now check that $\Omega(r_1,r_2)$ is singular for $V^{k}(\C\xi)$. We use our assumption on $\k$ so that we can assume that $\xi$ is given by \eqref{specialxi}. Then, if $s>0$ and $n=2$,  then $r_2=0$ and
$$
(d_1)_s\Omega(r_1,0)=((d_1)_sv_{\nu})\otimes \vacc=0.
$$
If $s>0$ and $n>2$, then
$$
\begin{aligned}
&(d_1)_s\Omega(r_1,r_2)=(d_1)_sv_{\nu}\otimes \left((\ov{E}_{n-r_2,1})_{0}\cdots (\ov{E}_{n-1,1})_{0}\vacc\right)\\
&+v_{\nu}\otimes \Theta_\q(ad(d_1))_s\left((\ov{E}_{n-r_2,1})_{0}\cdots (\ov{E}_{n-1,1})_{0}\vacc\right)\\
&=\sum_jv_{\nu}\otimes \left((\ov{E}_{n-r_2,1})_{0}\cdots\ov{[d_1,E_{j,1} ]}_{s}\cdots(\ov{E}_{n-1,1})_{-1/2}\vac\right)\\
&=-\sqrt{-1}\sqrt{\frac{n}{n-2}}\sum_jv_{\nu}\otimes \left((\ov{E}_{n-r_2,1})_{0}\cdots(\ov E_{j,1} )_{s}\cdots(\ov{E}_{n-1,1})_{0}\vacc\right)=0.
\end{aligned}
$$
Arguing as in Lemma \ref{parameterextremal}, taking into account the change in the choice of a Borel subalgebra for $\aa$ and the fact that $L^{\bar\q}_0\Omega(r_1,r_2)=\tfrac{1}{16}\dim \q\,\Omega(r_1,r_2)$, we find
$$ 
\begin{aligned}&\widetilde L_0\Omega(r_1,r_2)= \left(\frac{(\nu,\nu+2\rho)}{2(k_1+n)}-
\frac{(\nu^\aa_{r_1,r_2},\nu^\aa_{r_1,r_2}-2\rho^\aa)}{2(k_1+n)}+\tfrac{1}{16} \dim \q\right)\Omega(r_1,r_2),
\end{aligned}
$$
where $\nu^\aa_{r_1,r_2}$ is the weight of $\Omega(r_1,r_2)$ under the action of $V^\a(\aa)$.
 If $h\in\h_\aa$ then, since $\nu(h)=0$, 
$$
\begin{aligned}
&h^\aa_0\Omega(r_1,r_2)=v_{\nu}\otimes \Theta_\q(ad(h))_0\left((\ov{E}_{n-r_2,1})_{0}\cdots (\ov{E}_{n-1,1})_{0}\vacc\right)\\
&=\sum_jv_{\nu}\otimes \left((\ov{E}_{n-r_2,1})_{0}\cdots\ov{[h,E_{j,1} ]}_{0}\cdots(\ov{E}_{n-1,1})_{0}\vacc\right)\\
&+v_{\nu}\otimes \left((\ov{E}_{n-r_2,1})_{0}\cdots(\ov{E}_{n-1,1})_{0}\Theta_\q(ad(h))_0\vacc\right)\\
&=-\omega^\aa_{n-2-r_2}(h)\Omega(r_1,r_2)+v_{\nu}\otimes (\ov{E}_{n-r_2,1})_{0}\cdots(\ov{E}_{n-1,1})_{0}\Theta_\q(ad(h))_0\vacc.
\end{aligned}
$$
Since
$$
\Theta_\q(ad(h))_0\vacc=\tfrac{1}{4}Tr_\q(ad(h))-\half\sum \ov q^i_0ad(h)(q_i)_0\vacc=-\half Tr_{U^{-}}(ad(h))\vacc=0,
$$
we find that \eqref{tildeLomega} holds with
$$\widetilde\ell_0=\begin{cases}\frac{n-1}{2(k_1+n)}\left(\frac{ r_1 (n+r_1)}{n}-\frac{r_2 \left(n-r_2-2\right)}{n-2}\right)+\frac{n-1}{4}&\text{if $n>2$}\\
\frac{n-1}{2(k_1+n)}\frac{ r_1 (n+r_1)}{n}+\frac{1}{4}&\text{if $n=2$.}
\end{cases}$$
Since $\hat L=\widetilde L-L^{\xi}-L^{\ov{\mathfrak s}}$, we need to compute the action of  $L^{\xi}_0$ and of $L^{\ov{ \mathfrak s}}_0$  
 on $\Omega(r_1,r_2)$.
Since $\ov E_{j,1}\in\mathfrak f_1$ for $2\le i\le n-1$, we see that
$[L^{\ov{ \mathfrak s}}_0, (\ov E_{j,1})_{0}]=0$.
Moreover $
L^{\ov{ \mathfrak s}}_0\vacc=\tfrac{1}{4}\vacc$, so 
$$
L^{\ov{ \mathfrak s}}_0\Omega(r_1,r_2)=\tfrac{1}{4}\Omega(r_1,r_2).
$$
  Since $\Omega(r_1,r_2)$ is singular for $V^{k}(\C\xi)$, 
  $$
  L^{\xi}_0\Omega(r_1,r_2)=\frac{1}{2k}\xi_0^2\Omega(r_1,r_2).
  $$
Since
$$
\begin{aligned}
(d_1)_0\Omega(r_1,r_2)&=(d_1)_0v_{\nu}\otimes \left((\ov{E}_{n-r_2,1})_{0}\cdots (\ov{E}_{n-1,1})_{0}\vacc\right)\\
&+v_{\nu}\otimes \Theta_\q(ad(d_1))_0\left((\ov{E}_{n-r_2,1})_{0}\cdots (\ov{E}_{n-1,1})_{0}\vacc\right)\\
&=\nu(d_1)\Omega(r_1,r_2)+\sum_jv_{\nu}\otimes \left((\ov{E}_{n-r_2,1})_{0}\cdots\ov{[d_1,E_{j,1} ]}_{0}\cdots(\ov{E}_{n-1,1})_{0}\vacc\right)\\
&-\half Tr_{U^{-.-+}}(ad(d_1))v_{\nu}\otimes \left((\ov{E}_{n-r_2,1})_{0}\cdots(\ov{E}_{n-1,1})_{0}\vacc\right)\\
&=\begin{cases}(\nu(d_1)-\sqrt{\frac{2-n}{n}}n(\frac{r_2}{n-2}-1)\Omega(r_1,r_2)&\text{if $n>2$},\\
\nu(d_1)\Omega(r_1,0)&\text{if $n=2$},
\end{cases}
\end{aligned}
$$
and
$\nu(d_1)=\sqrt{\frac{2-n}{n}}r_1$, the final result is that
$$
(d_1)_0\Omega(r_1,r_2)=\begin{cases}\sqrt{\frac{2-n}{n}}(r_1-\frac{n}{n-2}r_2+n)\Omega(r_1,r_2)&\text{if $n>2$},\\
0&\text{if $n=2$},
\end{cases}
$$
so 
$$
L^\xi_0\Omega(r_1,r_2)=\begin{cases}\frac{((n-2) (n+r_1)-n
   r_2)^2}{4 (n-2) n
   (k_1+n)}&\text{if $n>2$},\\
0&\text{if $n=2$}.
\end{cases}
$$
Thus, if $n>2$,
$$
\ell_0=\widetilde \ell_0-\frac{((n-2) (n+r_1)-n
   r_2)^2}{4 (n-2) n
   (k_1+n)}-\frac{1}{4}=\frac{k_1 (n-2 ) + r_1^2 + 2 r_1 (r_2+1) + r_2 ( r_2+2)}{4 (k_1 + n)}
$$
as claimed.

If $n=2$, since $r_2=0$,
$$
\ell_0=\widetilde \ell_0-\frac{1}{4}=\frac{r_1^2+2 r_1}{4 \left(k_1+2\right)}=A(k,\mu)
$$
as well.
\end{proof}
Recall the conjugate linear involution $\omega_{\ov\q}$ of $\q$, defined in Section \ref{Application}, and observe that the Hermitian form $(\omega_{\ov \q}(\cdot),\cdot)$ is negative definite on $\q$.
This implies that  $F(\ov \q,U^{+})$ is a $\s^R$-twisted unitary representation of $F(\ov \q)$, i. e. $F(\ov \q,U^{+})$ admits a $\omega_{\ov \q}$-invariant Hermitian form $H^R_{\ov \q}$ that is positive definite. We normalize $H^R_{\ov \q}$ so that $H^R_{\ov q}(\mathbf 1,\mathbf 1)=1$.
\begin{prop}\label{extremalareunitaryR}Set $a=\frac{k_1+1}{n-1}$ and $k=-\frac{(k_1+1)(n-1)}{k_1+n}$. Fix $0\le r_1\le k_1$ and $0\le r_2\le n-2$. 
Then $(H_{r_1}\otimes H^R_{\ov\q})_{|\mathbb L^{-+}(r_1,r_2)\times\mathbb L^{-+}(r_1,r_2)}$  is a positive definite $\omega$-invariant form on $\mathbb L^{-+}(r_1,r_2)$.

In particular, if $\mu=\tfrac{r_1}{2}\theta_1+\tfrac{r_2}{2}\theta_2$, the highest weight $W^{\min}_{k}(D(2,1;a))$-module $L^{-+}(\mu,A(k,\mu))$ is unitary.
\end{prop}
\begin{proof}
The same proof of Proposition \ref{extremalareunitary} shows that $(H_{r_1}\otimes H^R_{\ov\q})_{|\mathbb L^{-+}(r_1,r_2)\times\mathbb L^{-+}(r_1,r_2)}$ is a $\omega$-invariant Hermitian form which is  manifestly positive definite.
\end{proof}

\begin{remark}\label{finalremark}
Lemma \ref{parameterextremalRamond} proves that the representations $\mathbb L^{-+}(0,r_2)$ and $\mathbb L^{-+}(r_1,n-2)$ realize the Ramond twisted extremal representations of \cite{KMPR} corresponding to choosing $\eta_{\min}=\epsilon_2-\epsilon_3$ (see \cite[10.5]{KMPR}). Thus Proposition \ref{extremalareunitaryR} proves their unitarity.

The Ramond twisted extremal representations corresponding to choosing $\eta_{\min}=-\epsilon_2+\epsilon_3$ are precisely the irreducible Ramond extremal $(+-)$--highest weight modules. These can be constructed in the same way by simply choosing
$$
U^{+}=span(\{E_{j,1}\mid 2\le j\le n-1\}\cup\{E_{j,n}\mid 2\le j\le n-1\}\cup\{\phi_1(h_2),\phi_1(e)\}),
$$
and
$$
U^{-}=span(\{E_{1,j}\mid 2\le j\le n-1\}\cup\{E_{n,j}\mid 2\le j\le n-1\}\cup\{\phi_1(h_1),\phi_1(f)\}),
$$
as maximal isotropic subspaces of $\q$. In particular, the obvious analog of Proposition \ref{extremalareunitaryR} proves their unitarity. Thus, we have proved that all extremal Ramond twisted representations of $W^{\min}_{k}(D(2,1;a))$ are unitary.
\end{remark}

\subsection*{Acknowledgements}
Pierluigi M\"oseneder Frajria and Paolo Papi  are partially supported by the PRIN project 2022S8SSW2 - Algebraic and geometric aspects of Lie theory - CUP B53D2300942 0006, a project cofinanced
by European Union - Next Generation EU fund.
Victor Kac is partially supported by a Simons grant.


\begin{thebibliography}{10}

\bibitem{AKMP}
{\sc D.~Adamovi{\'c}, V.~G. Kac, P.~M{\"o}seneder~Frajria, and P.~Papi}, {\em
  Defining relations for minimal unitary quantum affine {W}-algebras}, Commun.
  Math. Phys., {\bf 405}, No.~2 (2024), 33~pp.

\bibitem{AKMPP-JJM}
{\sc D.~Adamovi\'{c}, V.~G. Kac, P.~M\"{o}seneder~Frajria, P.~Papi, and
  O.~Per\v{s}e}, {\em Conformal embeddings of affine vertex algebras in minimal
  {$W$}-algebras {II}: decompositions}, Jpn. J. Math., {\bf 12}, No.~2 (2017),
  pp.~261--315.

\bibitem{AKMPP}
\leavevmode\vrule height 2pt depth -1.6pt width 23pt, {\em Conformal embeddings
  of affine vertex algebras in minimal {$W$}-algebras {I}: structural results},
  J. Algebra, {\bf 500} (2018), pp.~117--152.

\bibitem{BGP}
{\sc L.~Bedulli, A.~Gori, and F.~Podest\`a}, {\em Homogeneous hyper-complex
  structures and the {J}oyce's construction}, Differential Geom. Appl., {\bf
  29}, No.~4 (2011), pp.~547--554.

\bibitem{CGL}
{\sc T.~Creutzig, D.~Gaiotto, and A.~R. Linshaw}, {\em S-duality for the large
  {$N=4$} superconformal algebra}, Comm. Math. Phys., {\bf 374}, No.~3 (2020),
  pp.~1787--1808.

\bibitem{DK}
{\sc A.~De~Sole and V.~G. Kac}, {\em Finite vs affine {$W$}-algebras}, Jpn. J.
  Math., {\bf 1}, No.~1 (2006), pp.~137--261.

\bibitem{DL}
{\sc C.~Dong and X.~Lin}, {\em Unitary vertex operator superalgebras}, J.
  Algebra, {\bf 397} (2014), pp.~252--277.

\bibitem{ET1}
{\sc T.~Eguchi and A.~Taormina}, {\em Unitary representations of the {$N=4$}
  superconformal algebra}, Phys. Lett. B, {\bf 196}, No.~1 (1987), pp.~75--81.

\bibitem{ET2}
\leavevmode\vrule height 2pt depth -1.6pt width 23pt, {\em Character formulas
  for the {$N=4$} superconformal algebra}, Phys. Lett. B, {\bf 200}, No.~3
  (1988), pp.~315--322.

\bibitem{ET3}
\leavevmode\vrule height 2pt depth -1.6pt width 23pt, {\em On the unitary
  representations of {$N=2$} and {$N=4$} superconformal algebras}, Phys. Lett.
  B, {\bf 210}, No.~1-2 (1988), pp.~125--132.

\bibitem{Gunaydin}
{\sc M.~G\"unaydin, J.~L. Petersen, A.~Taormina, and A.~Van~Proeyen}, {\em On
  the unitary representations of a class of {$N=4$} superconformal algebras},
  Nuclear Phys. B, {\bf 322}, No.~2 (1989), pp.~402--430.

\bibitem{Joyce}
{\sc D.~Joyce}, {\em Compact hypercomplex and quaternionic manifolds}, J.
  Differential Geom., {\bf 35}, No.~3 (1992), pp.~743--761.

\bibitem{KMPD}
{\sc V.~G. Kac, P.~M\"{o}seneder~Frajria, and P.~Papi}, {\em Multiplets of
  representations, twisted {D}irac operators and {V}ogan's conjecture in affine
  setting}, Adv. Math., {\bf 217}, No.~6 (2008), pp.~2485--2562.

\bibitem{KMP}
\leavevmode\vrule height 2pt depth -1.6pt width 23pt, {\em Invariant
  {H}ermitian forms on vertex algebras}, Commun. Contemp. Math., {\bf 24},
  No.~5 (2022), 41~pp.

\bibitem{KMP1}
\leavevmode\vrule height 2pt depth -1.6pt width 23pt, {\em Unitarity of minimal
  {$W$}--algebras and their representations {I}}, Commun. Math. Phys., {\bf
  401}, No.~1 (2023), pp.~79--145.

\bibitem{KMPSF}
\leavevmode\vrule height 2pt depth -1.6pt width 23pt, {\em Spectral flow for
  minimal {$W$}-algebras and application to unitarity of their
  representations}.
\newblock arXiv:2508.06873, 2025.

\bibitem{KMPR}
\leavevmode\vrule height 2pt depth -1.6pt width 23pt, {\em Unitarity of minimal
  {$W$}-algebras and their representations {II}: Ramond sector}, Jpn. J. Math.,
  {\bf 20} (2025), pp.~169--281.

\bibitem{KMPX}
{\sc V.~G. Kac, P.~M\"oseneder~Frajria, P.~Papi, and F.~Xu}, {\em Conformal
  embeddings and simple current extensions}, Int. Math. Res. Not. IMRN,
  (2015), pp.~5229--5288.

\bibitem{KRW}
{\sc V.~G. Kac, S.-S. Roan, and M.~Wakimoto}, {\em Quantum reduction for affine
  superalgebras}, Commun. Math. Phys., {\bf 241}, No.~2-3 (2003), pp.~307--342.

\bibitem{KacT}
{\sc V.~G. Kac and I.~T. Todorov}, {\em Superconformal current algebras and
  their unitary representations}, Commun. Math. Phys., {\bf 102}, No.~2 (1985),
  pp.~337--347.

\bibitem{KW1}
{\sc V.~G. Kac and M.~Wakimoto}, {\em Quantum reduction and representation
  theory of superconformal algebras}, Adv. Math., {\bf 185}, No.~2 (2004),
  pp.~400--458.

\bibitem{KW}
\leavevmode\vrule height 2pt depth -1.6pt width 23pt, {\em Quantum reduction in
  the twisted case}, in Infinite dimensional algebras and quantum integrable
  systems, vol.~237 of Progr. Math., Birkh\"auser, Basel, 2005, pp.~89--131.

\bibitem{Kazamachar}
{\sc Y.~Kazama and H.~Suzuki}, {\em Characterization of {$N=2$} superconformal
  models generated by the coset space method}, Phys. Lett. B, {\bf 216},
  No.~1-2 (1989), pp.~112--116.

\bibitem{Kazamanew}
\leavevmode\vrule height 2pt depth -1.6pt width 23pt, {\em New {$N=2$}
  superconformal field theories and superstring compactification}, Nuclear
  Phys. B, {\bf 321}, No.~1 (1989), pp.~232--268.

\bibitem{Kostant}
{\sc B.~Kostant}, {\em Powers of the {E}uler product and commutative
  subalgebras of a complex simple {L}ie algebra}, Invent. Math., {\bf 158},
  No.~1 (2004), pp.~181--226.

\bibitem{M}
{\sc K.~Miki}, {\em The representation theory of the {${\rm SO}(3)$} invariant
  superconformal algebra}, Internat. J. Modern Phys. A, {\bf 5}, No.~7 (1990),
  pp.~1293--1318.

\bibitem{STVP}
{\sc A.~Sevrin, W.~Troost, and A.~van Proeyen}, {\em Extended supersymmetric
  {$\sigma$}-models on group manifolds ({II}). {C}urrent algebras}, Nuclear
  Phys. B, {\bf 311} (1988/89), pp.~465--492.

\end{thebibliography}

\vskip5pt
    \footnotesize{
\noindent{\bf V.K.}: Department of Mathematics, MIT, 77
Mass. Ave, Cambridge, MA 02139;\newline
{\tt kac@math.mit.edu}
\vskip5pt
\noindent{\bf P.M-F.}: Politecnico di Milano, Polo regionale di Como,
Via Anzani 42, 22100, Como, Italy;\newline {\tt pierluigi.moseneder@polimi.it}
\vskip5pt
\noindent{\bf P.P.}: Dipartimento di Matematica, Sapienza Universit\`a di Roma, P.le A. Moro 2,
00185, Roma, Italy;\newline {\tt papi@mat.uniroma1.it}, Corresponding author
}

\end{document}